\documentclass[11pt]{article}
\usepackage[left=1in,top=1in,right=1in,bottom=1in,nohead]{geometry}
\usepackage[onehalfspacing]{setspace}
\usepackage{titling}
\usepackage{amsmath,amssymb,amsthm,mathtools}
\usepackage{booktabs}
\usepackage{multirow}
\usepackage{multicol} %side-by-side equations
\usepackage{comment}
\usepackage{algorithmicx}
\usepackage{algpseudocode}
\usepackage{algorithm}
\usepackage{caption}
\usepackage{subfig}
\usepackage{color}
\usepackage[mathscr]{eucal}
\usepackage[numbers,sort]{natbib}
\usepackage{adjustbox}
\usepackage{bbm}
\usepackage{enumitem}
\usepackage{url}
\usepackage[bookmarks,colorlinks=true,urlcolor=blue,linkcolor=blue,citecolor=blue]{hyperref}

\allowdisplaybreaks[4]

\SetLabelAlign{Case}{Case #1.}
\newlist{nested}{enumerate}{5}
\setlist[nested]{nosep,labelwidth=*,align=Case,label*=.\arabic*}
\setlist[nested,1]{label=\arabic*}

\graphicspath{{images/}} % Directory in which figures are stored

\newtheorem{proposition}{Proposition}
\newtheorem{lemma}{Lemma}
\newtheorem{theorem}{Theorem}
\newtheorem{observation}{Observation}
\newtheorem{corollary}{Corollary}
\theoremstyle{definition}
\newtheorem{definition}{Definition}
\theoremstyle{remark}
\newtheorem{remark}{Remark}

\newcommand{\dro}{DRO}    
\newcommand{\droV}{DRO-V}

\usepackage{authblk}    
\title{Effective Scenarios in Multistage Distributionally Robust Optimization with a Focus on Total Variation Distance}
\author[1]{Hamed Rahimian\thanks{\texttt{hrahimi@clemson.edu}}}
\author[2]{G\"{u}z\.{i}n Bayraksan\thanks{\texttt{bayraksan.1@osu.edu}}}
\author[3]{Tito Homem-de-Mello\thanks{\texttt{tito.hmello@uai.cl}}}
\affil[1]{Department of Industrial Engineering, Clemson University, Clemson SC 29634, USA}
\affil[2]{Department of Integrated Systems Engineering, The Ohio State University, Columbus OH 43210, USA}
\affil[3]{School of Business, Universidad Adolfo Iba\~{n}ez, Santiago, Chile}

\date{}

\newcommand{\suchthat}{\text{s.t.}}

\DeclareMathOperator*{\argmax}{arg\,max}

\newcommand{\conv}[1]{\text{conv}\left(#1\right)}  

\DeclarePairedDelimiterX\sset*[2]{\lbrace}{\rbrace}%
{ #1 \,\delimsize| \,\mathopen{} #2 }

\newcommand{\ee}[2]{\mathbb{E}_{#1} \left[ #2 \right]}

\newcommand{\vvar}[2]{\mathrm{VaR}_{#1} \left[ #2 \right]}

\newcommand{\ccvar}[2]{\mathrm{CVaR}_{#1} \left[ #2 \right]}

\newcommand{\rro}[2]{\rho_{#1} \left[ #2 \right]}

\newcommand{\bs}[1]{\boldsymbol{#1}} %for lower case
\newcommand{\Bs}[1]{\mathbb{#1}} %for upper case mathbb
\newcommand{\Cs}[1]{\mathcal{#1}} %for upper case mathcal
\newcommand{\measurespace}{\left( \Omega, \Cs{F} \right)}

\newcommand{\hxi}[1]{\xi_{[#1]}}
\newcommand{\hXi}[1]{\Xi_{[#1]}}
\newcommand{\hx}[1]{x_{[#1]}}
\newcommand{\hstarx}[1]{x^{*}_{[#1]}}
\newcommand{\hbarx}[1]{\bar{x}_{[#1]}}

\newcommand{\fM}[1]{\mathfrak{M}_{#1+1|\hxi{#1}}}
\newcommand{\condbs}[2]{\bs{#1}_{#2+1|\hxi{#2}}}
\newcommand{\condbsnode}[2]{\bs{#1}_{#2+1|\omega_{#2}}}

\newcommand{\cPcond}[1]{\Cs{P}_{#1+1|\hxi{#1}}}
\newcommand{\cPcondnode}[1]{\Cs{P}_{#1+1|\omega_{#1}}}

\newcommand{\cPcondA}[1]{\Cs{P}^{\text{A}}_{#1|\hxi{#1-1}}}
\newcommand{\cPcondAnode}[1]{\Cs{P}^{\text{A}}_{#1|\omega_{#1-1}}}
\newcommand{\cPcondCA}[1]{\Cs{P}^{\text{CA}}_{#1+1|\hxi{#1}}}
\newcommand{\cPcondCAnode}[1]{\Cs{P}^{\text{CA}}_{#1+1|\omega_{#1}}}
\newcommand{\rhocond}[1]{\rho_{#1|\hxi{#1-1}}}

\newcommand{\cat}[2]{\Cs{A}_{#1}(#2)}

\begin{document}
	\maketitle
	\begin{abstract}
		We study multistage distributionally robust optimization (DRO) to hedge against ambiguity in quantifying the underlying uncertainty of a problem.
		Recognizing that not all the realizations and scenario paths might have an ``effect" on the optimal value, we investigate the question of how to define and identify critical scenarios for nested multistage DRO problems.
		%resulting from solving the problem. 
		Our analysis extends the work of Rahimian, Bayraksan, and Homem-de-Mello   [Math. Program. 173(1--2): 393--430, 2019], which was in the context of a static/two-stage setting, to the multistage setting. 
		%Recognizing that not all realizations and scenario paths might have ``effect" on the optimal value, 
		To this end, we define the notions of effectiveness of scenario paths and the conditional effectiveness of realizations along a scenario path for a general class of  multistage DRO problems. We then propose easy-to-check conditions to identify the effectiveness of scenario  paths in the multistage setting when the distributional ambiguity is modeled via the total variation distance. Numerical results show that these notions provide useful insight on the underlying uncertainty of the problem. 
		
		%version 2
		%In this paper, we study a multistage distributionally robust optimization problem to hedge against ambiguity in quantifying the underlying uncertainty of the problem.
		%Recognizing that not all realizations and scenario paths might have an ``effect" on the optimal value, we investigate the question of how to identify critical scenarios resulting from solving the problem. 
		%To this end, we define the notions of effectiveness of scenario paths and the conditional effectiveness of realizations along the path for a general class of problems. We propose easy-to-check conditions to identify the effectiveness when the distributional ambiguity is modeled via the total variation distance. Numerical results show that these notions provide useful insight on the underlying uncertainty of the problem. 
		
		\medskip
		
	\end{abstract}

	% % % % % % % % % % % % % % % % % % % % % % % % % % % % % % % % % % % % % % % % % %
	\section{Introduction}\label{sec: MULTI.intro} %% 1.
	% % % % % % % % % % % % % % % % % % % % % % % % % % % % % % % % % % % % % % % % % % % %

	Many decision-making problems are dynamic and stochastic, where the realizations of the stochastic process are revealed  over time and the decisions are made sequentially with the available information \citep{powell2019unified}. 
	%Such sequential decision-making problems  arise  in various domains such as  financial planning \citep{birge2007optimization},  energy systems management \citep{wallace2003stochastic}, logistics and transportation \citep{powell2003stochastic}, and water resources management \citep{zhang2016decomposition}. 
	% Huseyin Topaloglu doesn't use accent on "g" in his publications, including the one cited here.
	%Given the many important  real-world applications of dynamic optimization under uncertainty, 
	Such problems have been modeled using multistage stochastic programming  \citep{Stochastic,Shapiro_Lecture_SP} and adjustable robust optimization \citep{ben2004adjustable}, among others.
	%bertsimas2010optimality
	%,postek2016multistage,zhen2018adjustable 
	%, stochastic optimal control \citep{bertsekas1995DP}, and Markov decision processes \citep{puterman2005MDP}. 
	These approaches traditionally assume that the  distribution of the stochastic process is known. However, this is rarely true in real life. 
	Motivated by the fact that, in practice, a modeler might have some---albeit imperfect---distributional information about the uncertain parameters, a modeling approach, referred to as {\it distributionally robust}, has recently attracted considerable attention \cite{rahimian2019distributionally}.
	Pioneered by \cite{scarf1958}, such an approach assumes that the underlying probability distribution  is unknown and lies in an {\it ambiguity set} of probability distributions. When used in a decision-making framework, the distributionally robust  optimization (\dro) approach may protect  the decision-maker from the ambiguity in the underlying probability %distribution of uncertain parameters 
	\citep{delage2010,bertsimas2010minmax}. 
	
	In this paper, we study a \dro\  approach to a sequential decision-making problem in the spirit of multistage stochastic programming. We refer to this model as multistage \dro. 
	Most of the papers in this area focus on the computational aspects, proposing decomposition algorithms like specialized nested Benders' decomposition \citep[e.g.,][]{park2020multistage} or stochastic dual dynamic programming \citep[e.g.,][]{huang2017,philpott2018distributionally,duque2020distributionally,yu2020multistage}, or modeling aspects on how to form the ambiguity sets  \citep[e.g.,][]{Bertsimas2021,pflug2014}. In this paper, we study multistage \dro\ from a different point of view. Our aim is to gain insight into the scenarios that can impact the optimal value. % of the problem. 
	
	%There is a vast literature on two-stage/static DRO models. Different forms of ambiguity sets are proposed and studied in the literature: from moment-based models \citep{delage2010,wiesemann2014}, to discrepancy-based models \citep{mohajerin2018,gao2016}. %%, to shape-reserving models \citep{popescu2005semidefinite,hanasusanto2015chance}, and kernel-based models \citep{bertsimas2017bootstrap,shafieezadeh2019regularization,zhu2020kernel}. 
	%We refer interested readers to \citet{rahimian2019distributionally} for a comprehensive review of \dro\ models and their connections to  risk-averse optimization, as well as to \citet{kuhn2019wasserstein} and \citet{bayraksan2015} for discussions on Wasserstein-based and $\phi$-divergence-based ambiguity sets, respectively. 
	
	While the  literature on static/two-stage DRO models is quite mature by now  \cite[e.g.,][]{rahimian2019distributionally,kuhn2019wasserstein,bayraksan2015}, there are relatively few papers on multistage DRO. %gallego2001minimax,
	Some works consider moment-based models \citep{see2010robust,shapiro2020time,bertsimas2019adaptive,yu2020multistage}, and others consider discrepancy-based models: nested Wasserstein distance \citep{pflug2014}, $\infty$-Wasserstein \citep{Bertsimas2021}, $\chi^2$ distance \citep{klabjan2013robust}, modified $\chi^2$ distance \citep{philpott2018distributionally}, $L_{\infty}$ norm \citep{huang2017}, and general class of $\phi$-divergences \citep{park2020multistage}.  In \cite{duque2020distributionally}, the authors allow for the inclusion of moment constraints as well as discrepancy-based constraints, such as Wasserstein and total variation metrics. %, in the description of the ambiguity set. \cite{hanasusanto2013} study kernel-based models, in conjunction with $\chi^2$ distance.  
	In \cite{silva2021HMM}, the authors study data-driven multistage problems in which the input process can be represented by a Hidden Markov Model (HMM), and apply DRO to account for estimation errors of the HMM transition matrix.
	The authors in \cite{shapiro2020tutorial} and \cite{pichler2021mathematical} study general multistage \dro\ models  from the   theory of rectangular ambiguity sets and the connection to decomposability and time consistency of the resulting dynamic risk measures.

	\begin{comment}
	As mentioned before, a first task when building a \dro\ model is how to form the ambiguity set of probability  distributions. The ambiguity set should be chosen based on the application domain and the distributional information in hand, with an eye on tractability of the resulting optimization model and the quality of produced policies \citep{rahimian2019distributionally}. 
	\end{comment}

	Modeling distributional ambiguity in a multistage \dro\ is more complicated than the static/two-stage setting, and it is tied to the time  consistency  of the policies \citep[e.g.,][]{ruszczynski2010risk,carpentier2012dynamic,shapiro2009time}. 
	%%A basic  requirement of multistage stochastic optimization models is the implementablity of policies  \citep[Chapter~3]{Shapiro_Lecture_SP}.  That is, decisions must  only depend on the past information available at the time of decision-making. Implementability is a feasibility requirement and is not influenced by the ambiguity set.  On top of implementability,
	%A desired property of multistage stochastic optimization models, especially those that involve risk measures, is that of time (dynamic) consistency \citep{carpentier2012dynamic,de2016building}. 
	%%The interpretation of this given by  \citet{homem2016risk} is that ``decisions made today should agree with the planning made yesterday for the scenario that actually occurred." 
	%It is known that not every (risk-averse) multistage stochastic optimization model is time consistent \citep{shapiro2009time,shapiro2020tutorial}. Moreover, under suitable assumptions, \dro\ models can be equivalently written as a risk-averse optimization problem. Thus, it is  natural to think that not every multistage \dro\ model will automatically be time consistent. 
	%%Hence, the question is:  how to form the ambiguity set so that it leads to  a   time-consistent multistage \dro\ model? Should we model the ambiguity of the underlying probability distribution of the entire stochastic process, in the spirit of the so-called {\it end-of-horizon} risk measures? Should we model the ambiguity in a nested way, in the spirit of {\it dynamic programming}? 
	In order to enforce  time consistency of a multistage \dro, we model the ambiguity in a nested way (see Section \ref{sec: MULTI.DP}). That is, at each stage, we form a {\it conditional} ambiguity set, conditioned on the available information so far, to model the ambiguity in the distribution of the next stage's uncertain parameters. 
	%Such a formulation naturally leads to dynamic programming equations (see Section \ref{sec: MULTI.DP}), and hence, Bellman's principle of optimality ensures that policies that satisfy these equations will be time consistent \citep{shapiro2020tutorial}. 
	%%Such a nested formulation has desired properties of implementability and consistency of the yielded optimal policies. 
	Associated with an optimal policy to such a formulation is a worst-case probability distribution. 
	%Natural question that arises are: 
	What does this distribution look like? What information can be gained from it? 
	Such questions are raised in the context of static/two-stage \dro\ models in \cite{rahimian2019}, and in this paper, we study similar questions for a multistage \dro. 
	
	Assuming a stochastic process with a finite number of scenarios, the authors in \cite{rahimian2019} define the notions of {\it effective} and {\it ineffective}  scenarios for a  static/two-stage \dro\ (reviewed in Section \ref{sec: MULTI.background_def}). 
	Effective scenarios are those that lead to
	changes in the optimal value if removed from the ambiguity set (more precisely, forced to have a zero probability). Likewise, ineffective
	scenarios are those that can be removed from the ambiguity set without causing any change to the optimal value.
	It is shown in \cite{rahimian2019}  that one cannot determine whether
	a scenario is effective or not simply by looking at the worst-case probabilities. Indeed, it is possible to have an effective scenario with a zero worst-case probability and an ineffective scenario with a positive worst-case probability. Such a phenomenon is later tied to the existence of a strict monotone risk measure in  \cite{shapiro2020tutorial}, a concept related to the  time consistency of risk-averse multistage models. 
	In \cite{rahimian2019}, the authors provide easy-to-check conditions to identify effective/ineffective scenarios when the distributional ambiguity is modeled via the total variation distance around a {\it nominal} distribution, and in \cite{rahimian2019controlling}, these are generalized  to static/two-stage \dro\ with  continuous distributions in the context of inventory problems.  % in a stochastic program \citep{heitsch2003,heitsch2009reduction}. 
	These easy-to-check conditions provide computational and practical advantages over the na\"{i}ve resolving  of the corresponding problems multiple times (once they are forced to have a zero probability in the ambiguity set) and observing the cases in which there were any changes to the optimal value. 
	In \cite{arpon2018scenario}, the idea of effective scenarios is applied to risk-averse optimization problems, thereby yielding a \textit{scenario reduction} method for such problems.  
	
	%%% ADD LIT. REVIEW
	The notion of effective scenarios resonates with related concepts proposed in the literature, some of them quite recently. One such concept is that of a \textit{coreset} \citep{agaHV:05}:  given a set  of points $\Omega$, a coreset is ``an easily computable subset $\Cs{S} \subset \Omega$,  so that solving the underlying problem on $\Cs{S}$ gives an approximate solution to the original problem," an idea that has found  applications particularly in some statistical learning methods such as clustering, k-medians, and regression, among others. Another related concept is that of \textit{supporting constraints} \citep{camgar:18}, defined for a class of problems in which each constraint  corresponds to a scenario. In that context, a constraint is said to be supporting  if its removal changes the optimal solution of the problem. The authors use this notion to estimate  the probability that, given a problem with $m$ constraints, a new randomly  selected $(m+1)$th constraint  is violated by the optimal solution of that problem. The idea of effective scenarios can also be viewed from the perspective of \textit{problem-driven scenario reduction}, where the goal is to find what are the ``essential" scenarios that must be kept from an original set of scenarios  so that solving the problem over the reduced set yields the same optimal solution (or optimal value) as the original one; see  \cite{fairbrother2019problem} for the development of such ideas in the context of risk-averse optimization problems, leading to the notion of \textit{risk regions} which turn out to be very much related to the results in \cite{arpon2018scenario}. In \cite{mundru2020}, a new distance between two distributions that considers the objective function values is defined for problem-driven scenario reduction in two-stage problems. 
	All these concepts notwithstanding, to the best of our knowledge, no concept related to the notion of effective scenarios has been proposed in the context of \textit{multistage} problems.

	In this paper, we generalize the notion of effective/ineffective scenarios to a multistage \dro. As one might anticipate, these notions are nontrivial and more subtle in a multistage setting and several questions arise, even under a finite stochastic process: 
	\begin{enumerate}
		\item Should we study them for the entire realizations of the stochastic process---a {\it scenario path}---for a finite stochastic process represented by a scenario tree? %---in the spirit of end-of-horizon risk measures? 
		\item Should we study them for the realizations along a scenario path, at each stage {\it separately}? 
		\item Is there any connection between the effectiveness of a scenario path and the effectiveness of  realizations along the path? 
	\end{enumerate}
	In this paper, in agreement with our desire to have a time-consistent  multistage \dro\ model, we limit ourselves to a nested formulation and provide rigorous mathematical definitions for both the effectiveness of a scenario path  and the effectiveness of separate realizations along a path. %These definitions are in similar spirit to the definitions of effective/ineffective for a static \dro. %: whether the  optimal value changes once they are forced to have a zero probability in the respective ambiguity set. 
	The effectiveness of a scenario path is defined with respect to  a change in the optimal value once that scenario path is removed (Section \ref{sec: MULTI.gen_eff_path}). On the other hand, the effectiveness of a realization along a scenario path is defined {\it conditionally}, given the available information (Section \ref{sec: MULTI.gen_eff_realization}).  
	%, and the change in the optimal value must be studied conditionally 
	Henceforth, we refer to them as {\it conditionally effective/ineffective} realizations. 
	To provide an answer to the third question above,  we focus on a multistage \dro, where the  conditional distributional ambiguity is modeled via the total variation distance (referred to as a multistage \droV). 
	In this setting, we provide an affirmative answer; yes. %, there is a connection between the effectiveness of a scenario path and the conditional effectiveness of realizations along the path. 
	%This result is interesting---and somewhat surprising---in particular because the nested formulation of the studied multistage \dro\ model does not necessarily focus on the full scenario paths, \alertHR{but only on the next stages problems conditionally}. %rather on partial scenario paths. %of the   because of the fact that the studied nested formulation of the multistage \dro\ focuses on the next stage problem conditionally. 
	%Nevertheless, given the nested formulation, it is not surprising to be able to provide an affirmative answer to the second question above. 
	
	By exploiting the structure of the conditional ambiguity sets formed via the total variation distance, we show that the easy-to-check  conditions proposed in \cite{rahimian2019} are applicable to identify the conditional effectiveness of realizations along a scenario path. 
	We then show that a scenario path is effective if and only if all realizations along the path are conditionally effective (Theorem \ref{thm: MULTI.path_eff_ineff}). This result is interesting---and somewhat surprising---because it says that checking effectiveness in a one-step-ahead fashion suffices to determine the effectiveness of the entire path.
	
	%Such sufficient and necessary conditions are of utmost importance, particularly because multistage problems are usually solved via decomposition algorithms, such as variants of the nested Benders' decomposition \cite{Stochastic}. Thus, having an optimal policy and cost-to-go  functions at each stage, conditional effectiveness of realizations along a scenario path can easily be identified using the results in \cite{rahimian2019}. 
	%That being said, our motivation to study the conditional effectiveness of realizations is not only from a modeling/conceptual perspective but also from a computational perspective.
	%Similar to the static DRO, while some scenario paths may remain unidentified, in our numerical experiments such unidentified scenario paths represented a small fraction of the total number of scenario paths. 

	Besides providing a rigorous answer to the third question above, which yields both a practical/computational advantage and a conceptual/managerial perspective in identifying effectiveness of scenario paths, 
	our choice to work with  the total variation distance has other benefits.
	The first one is risk interpretation: As we shall see in Section \ref{sec: MULTI.risk}, when the total variation distance is used to define the conditional ambiguity sets, the resulting multistage \dro\ is equivalent to minimizing a composite coherent risk measure, consists of nested coherent risk measures. %Conditional risk mappings in this case are  a convex combination of conditional value-at-risk (CVaR) and worst case of the cost function under the nominal (conditional) distribution, parameterized by the size of ambiguity set. 
	The second benefit is computational: %when the total variation distance is used to model the distributional ambiguity in the multistage \dro, 
	The problem reduces to a  computationally tractable optimization model and is amenable to  a numerical decomposition algorithm. %
	%, in the spirit of nested decomposition \citep{Stochastic}. 
	This is of particular importance in a multistage setting as the stochastic process is typically discretized and hence, it gives a computational advantage over other metrics such as a general Wasserstein distance. 
	Finally, the total variation distance resides in the intersection of three classes of discrepancy measures %between probability distributions
	commonly used in DRO: It is a metric, a $\phi$-divergence, and a special case of the Wasserstein distance when the distance between two scenarios is $0$ (when the scenarios are the same) or $1$ (when they are different). 
	
	To summarize, the contributions of this paper  are as follows.
	%\begin{enumerate}[label=(\roman*)]
	%	
	%	\item 
	We introduce, to the best of our knowledge for the first time, the notions of effectiveness for realizations and scenario paths for a general class of nested multistage \dro\ 
	%with a nested formulation 
	under a finite stochastic process.  
	%	
	%	
	%	\item  
	We provide easy-to-check conditions to identify the conditional effectiveness of realizations for multistage  \droV. 
	%, where the conditional distributional ambiguity is modeled  via the total variation distance. 
	Our main result ties the effectiveness of a scenario path to the conditional effectiveness of realizations along the path. 
	%	
	%	\item   
	We numerically illustrate how the concepts introduced and 
	the main results of the paper
	%Computational results indicate that easy-to-check conditions work well by identifying a relatively large set of scenario paths. Furthermore, we 
	provide managerial insight to the underlying uncertainty of a multistage \dro.  
	
	%\end{enumerate}

	The rest of this paper is outlined as follows. 
	In Section \ref{sec: background}, we  define the class of multistage \dro\ problems we are interested in, establish notation, and review background information on effective scenarios for a static  \dro. %, introduced in \cite{rahimian2019}. 
	%present alternative formulations, and discuss the time consistency  of  the presented formulations. Moreover, 
	In Section \ref{sec: MULTI.effective_scens}, we define the notions of effectiveness of scenario paths and conditional effectiveness of realizations for a multistage \dro. %Moreover,  
	In Section \ref{sec: MULTI.TV}, we define a multistage \droV\   and present its risk interpretation. 
	Section \ref{sec: MULTI.thms} provides easy-to-check conditions to identify the conditional effectiveness of realizations and the effectiveness of scenario paths for a multistage \droV. %In particular, we tie the effectiveness of scenario paths to the  conditional effectiveness of realizations along the path. 
	%In Section \ref{sec: MULTI.primal}, we present a decomposition algorithm to solve multistage \droV. 
	We then present numerical experiments and provide managerial insights in Section \ref{sec: MULTI.numerics}. 
	Finally, we end with conclusions  in Section \ref{sec: MULTI.discuss}.

	% % % % % % % % % % % % % % % % % % % % % % % % % % % % % % % % % % % % % % % % % % % %	
	\section{Background and Notation} 
	\label{sec: background}
	\subsection{Multistage DRO}
	\label{sec: MULTI.DP}
	We study a class of  multistage stochastic optimization problems with $T > 1$ stages. To formally define the class of problems we are interested in, consider a measurable space $\measurespace$. The uncertain parameters are gradually realized over time and are represented by a stochastic process $\xi:=(\xi_1, \dots, \xi_T)$ on the probability space, where $\xi_{t}:\Omega \mapsto \Xi_{t} \subseteq \Bs{R}^{d_t}$ and is composed of the random parameters in stage $t$. We assume that $\xi_{t}$, $t=1, \ldots, T$,  has  finitely many possible realizations, %, so we can represent the process using a scenario tree. 
	where $\xi_1$ is a degenerate random vector (i.e., constant). 
	Let $\hxi{t}:=(\xi_{1}, \ldots, \xi_{t})$ denote  the history of the stochastic process up to (and including) time $t$. The filtration $\{\Cs{F}_{t}\}_{t=1}^{T}$  associated with $\xi$ is defined by $\Cs{F}_{t}:=\sigma(\hxi{t})$, where $\sigma(\hxi{t})$ is the $\sigma$-algebra generated by $\hxi{t}$, and $\{\emptyset, \Omega\}=\Cs{F}_1 \subset \Cs{F}_2 \subset \ldots \subset \Cs{F}_T=\Cs{F}$. %, and we assume $\Cs{F}_{1}=\{\emptyset, \Omega\}$ (i.e., $\xi_{1}$ is deterministic) and $\Cs{F}_{T}=\Cs{F}$. 
	
	As information on the stochastic process $\xi$  becomes available at stage $t$, a decision must be made based on the available information so far.  These decisions define a {\it decision rule} or {\it policy} $x:=[x_1, x_2, \ldots, x_T]$. 
	A basic  requirement of multistage stochastic optimization problems is the {\it unimplementablity} and {\it nonanticipativity} of policies  \citep[Chapter~3]{Shapiro_Lecture_SP}.  That is, decision  $x_t$, $t=1, \ldots, T$, must depend only on the information available up to time $t$ and must not depend  on   neither the future realizations of the stochastic process nor the future decisions. 
	Similar to $\hxi{t}$, we use $\hx{t}$ to denote a policy up to (and including) time $t$, i.e., $\hx{t}:=(x_1, \ldots, x_t)$.
	%If each decision $x_t:=x_t(\hxi{t})$ is $\Cs{F}_{t}$-measurable, $t=1, \ldots, T$, then a policy $x$ is {\it implementable} or  {\it nonanticipative}. 
	
	Consider a multistage \dro\ problem  of the form 
	\begin{equation}
	\label{eq: MULTI.robust}
	\begin{split}
	\min_{x_{1} \in \Cs{X}_{1}} \ g_{1}(x_{1}, \xi_{1}) & + \max_{\bs{p}_{2} \in \Cs{P}_{2|\hxi{1}}} \ \mathbb{E}_{\bs{p}_{2}}   \left[   \min_{x_{2} \in \Cs{X}_{2}}  g_{2}(x_{2}, \xi_{2}) + \max_{\bs{p}_3 \in \Cs{P}_{3|\hxi{2}}} \ \mathbb{E}_{\bs{p}_{3}} \Bigg[ \ldots +  \Bigg. \right. \\
	& \Bigg.  \left.  \max_{\bs{p}_{T} \in  \Cs{P}_{T|\hxi{T-1}}}  \ \mathbb{E}_{\bs{p}_{T}} \left[   \min_{x_{T} \in \Cs{X}_{T}} g_{T}(x_{T}, \xi_{T}) \right]  \ldots  \Bigg]  \right].
	\end{split}
	\tag{T-\dro}
	\end{equation}
	In stage $t$, $t=1, \ldots, T$, the set-valued mapping $\Cs{X}_{t}:=\Cs{X}_{t}(\hx{t-1}, \hxi{t}) \subset \Bs{R}^{n_{t}}$  denotes a non-empty, compact polyhedral  feasibility set, and $g_{t}: \Bs{R}^{n_t} \times \Bs{R}^{d_t} \mapsto \Bs{R} $ is a random, real-valued  polyhedral  function, with the decision $x_{t}$ and the realized uncertainty $\xi_{t}$ given\footnote{To simplify the exposition and focus on the concepts rather than technicalities, throughout we assume that $\Cs{X}_{t}$ is a polyhedral set and $g_{t}$ is a random polyhedral function, $t=1, \ldots, T$. Nevertheless, the results in this paper can be extended to general convex sets and functions under appropriate constraint qualification conditions.}. 
	%\alertHR{Edit assumptions about convexity.}
	%The optimization is performed over  policies $x:=[x_1, x_2, \ldots, x_T]$. 
	For nonanticipativity, we force  $x_t=x_t(\hxi{t})$ to be $\Cs{F}_{t}$-measurable, $t=1, \ldots, T$, by assuming both $\Cs{X}_{t}(\hx{t-1}, \cdot)$ and $g_t(x_t, \cdot)$ are $\Cs{F}_{t}$-measurable.
	Moreover, $\cPcond{t} \subseteq \fM{t}$ is the  {\it conditional ambiguity set} for the conditional distribution of $\xi_{t+1}$, conditioned on $\hxi{t}$, $t=1, \ldots, T-1$. The set $\fM{t}$ contains all conditional probability distributions induced by $\xi_{t+1}$, given $\hxi{t}$. Also, $\ee{\condbs{p}{t+1}}{\cdot}$ denotes the conditional expectation  with respect to $\condbs{p}{t+1} \in \cPcond{t}$. %,  $t=1, \ldots, T-1$. 

	The inner maximization problems in \eqref{eq: MULTI.robust} hedge against the worst-case probability distribution in the conditional ambiguity set $\cPcond{t}$, $t=1, \ldots, T-1$. 
	Associated with an optimal policy $x^*:=[x_1^*,x_2^*, \ldots, x_T^*]$ to  \eqref{eq: MULTI.robust} is an optimal worst-case probability distribution $\bs{p}^*:=\bs{p}^*(x^*):=[\bs{p}^*_{2|\hxi{1}},\bs{p}^*_{3|\hxi{2}}, \ldots, \bs{p}^*_{T|\hxi{T-1}}]$. An optimal probability distribution $\bs{p}^*$, or in general, a (push-forward) probability distribution $\bs{p}:=[\bs{p}_{2|\hxi{1}},\bs{p}_{3|\hxi{2}}, \ldots, \bs{p}_{T|\hxi{T-1}}]$ induced by $\xi$ on $\measurespace$, should be understood in similar sense to a policy. That is, at each stage $t$, $t=1, \ldots, T-1$, given the available information  on $\hxi{t}$, the conditional probability of $\xi_{t+1}$ can be determined.

	%Figure \ref{fig: nested} illustrates the logic of the nested formulation \eqref{eq: MULTI.robust}, under a binary scenario tree in four stages (levels). Pictorially, each of the seven  boxes contains a parent node and a set of children nodes. Each box then  should be interpreted as given we  are at the parent node, a conditional ambiguity set is formed to model the conditional distributional ambiguity for the respective children nodes. 
	
	While there are various ways to model the conditional distributional ambiguity in \eqref{eq: MULTI.robust}, we do not assume any particular structure on $\cPcond{t}$ until Section \ref{sec: MULTI.TV}. %, we narrow down our focus to the case that $\cPcond{t}$ contains all conditional probability distributions $\bs{p}_{t}$ that are sufficiently close to a {\it nominal} conditional distribution, where closeness is defined in the sense of the {\it total variation distance}. 
	Note that when $\cPcond{t}$ contains only the {\it nominal} conditional distribution of $\xi_{t+1}$, conditioned on $\hxi{t}$, \eqref{eq: MULTI.robust} reduces to the classical multistage stochastic program. Hence, by the monotonicity and the translation invariance of the expectation, this formulation can be written in a nonnested form. 

	\begin{remark}
		\label{rem: MULTI.consitency_eoh}
		It is tempting to form  a multistage \dro\ problem of the form 
		\begin{equation}
		\label{eq: MULTI.eoh}
		\begin{array}{rl}
		\min\limits_{x_{1}, x_{2}, \ldots, x_{T}}  \max\limits_{\Bs{P} \in \Cs{P} } \;  &  \ee{\Bs{P}}{g_{1}(x_{1}, \xi_{1})+ g_{2}(x_{2}, \xi_{2})+ \ldots +g_{T}(x_{T}, \xi_{T})}\\
		\suchthat  & x_{t} \in \Cs{X}_{t}, \; t=1, 2, \ldots T,
		\end{array}
		\end{equation}
		where the ambiguity set $\Cs{P}$ is a subset of all probability measures on $\measurespace$. 
		%Figure \ref{fig: eoh} illustrates the logic of the so-called {\it end-of-horizon} formulation \eqref{eq: MULTI.eoh}. %, under a binary scenario tree in four stages. 
		%Pictorially, there is only one box, representing an ambiguity set on the joint probability distribution of the stochastic process $\xi=(\xi_1, \ldots, \xi_4)$. 
		While such a formulation is valid, it is not automatically time consistent. In the sense of the sequential optimality conditions  of a  policy  \citep{shapiro2009time,ruszczynski2010risk,shapiro2012time}, time consistency implies that %: ``a policy is time consistent if and only if the future planned decisions are actually going to be implemented" \citep{rudloff2014time}.  In other words, according to \citet{homem2016risk}: `
		``decisions made today should agree with the planning made yesterday for the scenario that actually occurred" \citep{homem2016risk}.  
		In fact, \eqref{eq: MULTI.eoh} is time consistent if and only if the  risk measure corresponding to the inner maximization problem in \eqref{eq: MULTI.eoh}  is decomposable into a (nested) sequence of conditional risk mappings and  satisfies a strict monotonocity property \citep{shapiro2012,shapiro2020tutorial,shapiro2016decomposability}. The decomposability property lends itself to a rectangular set associated with set $\Cs{P}$. We refer interested readers to \cite{shapiro2015rec} for topics on the existence and construction of a rectangular set.   
		
	\end{remark}

	%\begin{figure}[!htbp]
	%	\centering
	%	\subfloat[][Nested formulation.]{\label{fig: nested}\includegraphics[width=0.45\linewidth]{tree/nested.pdf}}
	%	\subfloat[][End-of-horizon formulation.]{\label{fig: SVDOriginalModified}\includegraphics[width=0.45\linewidth]{tree/eoh.pdf}}
	%	\caption{\label{fig: formulations} Nested versus end-of-horizon formulation, under a binary scenario tree in four stages.}
	%\end{figure}	

	%\subsection{Dynamic Programming Reformulation}

	It is known the time-consistent formulation \eqref{eq: MULTI.robust} naturally leads to the so-called Bellman's principle of optimality  to derive dynamic programming (DP) equations. %These equations, combined with the time consistency of an optimal policy to \eqref{eq: MULTI.robust} obtained from them, form a basis for our analysis on the effectiveness of scenario paths and the conditional effectiveness of realizations along a scenario path  (see Section \ref{sec: MULTI.effective_scens} for more details). 
	These equations %, %combined  with the time consistency of an optimal policy satisfying them (see Remark \ref{rem: time_consistency_DP}),  
	play an important role in the analysis of the effectiveness of the scenario paths and the conditional effectiveness of realizations along a scenario path in this paper. We now present these equations for future references.
	
	Consider  \eqref{eq: MULTI.robust} at a given stage $t$, $t=1, \ldots, T$, when all information from the previous  stages, given by $\hx{t-1}$ and $\hxi{t}$, is known. Throughout the paper, we use the vacuous notation $x_{[0]}$ for convenience. %That is, we have the following optimization problem:
	%\begin{equation}
	%\label{eq: MULTI.VF_t}
	%\begin{split}
	%Q_{t}(\hx{t-1}, \hxi{t}):=\min_{x_{t} \in \Cs{X}_{t}} \ g_{t}(x_{t}, \xi_{t}) + \max_{\condbs{p}{t} \in  \cPcond{t} } \ \mathbb{E}_{\condbs{p}{t}}   \left[   \min_{x_{t+1} \in \Cs{X}_{t+1}}  g_{t+1}(x_{t+1}, \xi_{t+1}) + \right. \\ 
	%\left.  \max_{\bs{p}_{t+2} \in  \Cs{P}_{t+2|\hxi{t+1}} } \ \mathbb{E}_{\bs{p}_{t+2}} \Bigg[ \ldots +  
	%\max_{\bs{p}_{T} \in \Cs{P}_{T|\hxi{T-1}} } \ \mathbb{E}_{\bs{p}_{T}} \left[   \min_{x_{T} \in \Cs{X}_{T}} g_{T}(x_{T}, \xi_{T}) \right]  \ldots  \Bigg]  \right]. 
	%\end{split}
	%\end{equation}
	Going backward in time, we  can derive DP reformulation of \eqref{eq: MULTI.robust}, where the cost-to-go (value) function at stage $t$, $t=1, \ldots, T$, is as follows
	\begin{equation}
	\label{eq: MULTI.VF}
	Q_{t}(\hx{t-1}, \hxi{t})=\min_{x_{t} \in \Cs{X}_{t}} \ g_{t}(x_{t}, \xi_{t}) + \max_{\condbs{p}{t+1} \in \cPcond{t}} \ee{\condbs{p}{t+1}}{Q_{t+1}(\hx{t}, \hxi{t+1})}. 
	\end{equation}
	We refer to  $\max_{\condbs{p}{t+1} \in \cPcond{t}} \ee{\condbs{p}{t+1} }{Q_{t+1}(\hx{t}, \hxi{t+1})}$ as the {\it worst-case expected problem} at stage $t$. We assume that $Q_{t}(\hx{t-1}, \hxi{t})$ is finite and integrable for any distribution in $\fM{t}$, $t=1, \ldots, T$.
	By convention, we  set  the worst-case expected problem at stage $T+1$ to zero. Moreover, we denote the optimal value simply with $Q_{1}$.
	
	%: $\max_{\condbs{p}{T} \in \cPcond{T}} \ee{\condbs{p}{T} }{Q_{T+1}(\hx{T}, \hxi{T+1})} \equiv 0$.  
	%At the first stage, solving the following problem
	%\begin{equation*}
	%	\min_{x_{1} \in \Cs{X}_{1}} \ g_{1}(x_{1}, \xi_{1}) + \max_{\bs{p}_{2|\hxi{1}} \in \Cs{P}_{2|\hxi{1}}} \ee{\bs{p}_{2|\hxi{1}}}{Q_{2}(x_1, \hxi{2})},
	%\end{equation*}
	%gives the optimal value of T-\droV. 
	
	\begin{remark}
		\label{rem: time_consistency_DP}
		Throughout this paper we assume that an optimal policy satisfies the DP equations, possibly obtained via a nested Benders' decomposition algorithm \citep{Stochastic}. An optimal policy $x^{*}=[x_1^*, x_2^*(\hxi{2}), \ldots, x_T^*(\hxi{T})]$ that satisfies the DP equations  is time consistent because for every stage $t$, $t=1, \ldots, T$, $x^*_{t}=x_t^*(\hxi{t})$ is an optimal solution to problem \eqref{eq: MULTI.VF} for  every realization of the random vector $\hxi{t}$.  %, with respect to any marginal probability distribution in stage $t$. 
		%This principle of optimality, further discussed in Remark \ref{rem: MULTI.inherited} in the context of \eqref{eq: MULTI.robust} formed via the total variation distance, is a key for the proofs of results in Section \ref{sec: MULTI.TV_main}. \qed
	\end{remark}

	\subsection{Scenario Tree Notation} 
	Recall that we assume the stochastic process in $T$ stages is represented by $\xi=(\xi_1, \dots, \xi_T)$, where $\xi_{t}: \Omega \mapsto  \Xi_{t} \subset \Bs{R}^{d_t}$.  %For the rest of the paper, let us assume that $\xi_{t}$, $t=1, \ldots, T$,  has  finitely many possible realizations, %%, so we can represent the process using a scenario tree. 
	%where $\xi_1$ is a degenerate random vector. 
	We denote the support of $\hxi{t}$ by $\hXi{t}:=\times_{t^{\prime}=1}^{t} \Xi_{t^{\prime}}$, $t=1, \ldots, T$. We let $\hXi{T}:=\Xi$. %, we have  $\hXi{t}=\Set*{\xi_{[t]}}{ \exists \ \xi_{t^{\prime}} \in \Xi_{t^{\prime}}, \forall t^{\prime} > t \ \st \ (\hxi{t}, \xi_{t+1}, \dots, \xi_{T}) \in \Xi}$, $t=1, \ldots, T-1$. 
	Because we assume a finite sample space, a scenario tree with $T$ stages represents all the
	possible ways that the stochastic process $\xi$ evolves. %We refer the interested readers to   \cite[Section~6.7.1]{Shapiro_Lecture_SP} and \citet[Section~3]{Collado2012} for a detailed construction of this scenario tree, and in this section we define our notation. 
	
	%Let $\Cs{T}=(\Cs{V}, \Cs{S}) $ denote the scenario tree, which is a rooted-out tree, rooted at node $\texttt{`root'}$,  Moreover, $\Cs{V}$ denotes the set of all nodes and $\Cs{S}$ denotes the set of arcs. 
	We associate each realization of the random vector $\hxi{t}$ to a node in stage $t$ of the scenario tree, $t=1, \ldots, T$. We denote the set of nodes at stage $t$ by $\Omega_{t}$, and $\omega_t \in \Omega_t$ denotes an element in that set. %Hence, the nodes in  $\Omega_t$ corresponds to elementary events in $\Cs{F}_{t}$. 
	By assumption, $\Omega_1$ is a singleton, containing the root node $\omega_{1}$. 
	We  use the notation $\xi_{t}^{\omega_t}$ ($\xi_{[t]}^{\omega_t}$) to indicate a specific realization of the random vector $\xi_{t}$ ($\xi_{[t]}$) at node $\omega_t$.
	Because of the one-to-one correspondence between $\hXi{t}$ and $\Omega_t$, we may use them interchangeably, both in words and notation. 
	For example, 
	%because a node $\omega_{t} \in \Omega_{t}$, $t=1, \ldots, T$, is associated with $\hxi{t}^{\omega_{t}}$,  
	we may use $\omega_{t}$ as  a shorthand notation for $\hxi{t}^{\omega_{t}}$.  %Similarly, we use the notation $\xi_{t}^{\omega_t}$ to indicate a specific realization of the random vector $\xi_{t}$.   
	In particular, there is a one-to-one correspondence between $\omega_T \in \Omega_T$ at the last stage and a  realization $\xi \in \Xi$. %with some abuse of notation, 
	We use  $\omega_T \in \Omega_{T}$  to refer to  a generic  node in  stage $T$ of the tree as well as a realization $\xi \in \Xi$ of the stochastic process. By construction, a  realization $\xi \in \Xi$ is determined by a path from the root node $\omega_{1}$ to a leaf node $\omega_{T} \in \Omega_T$. Such a generic path is uniquely identified with $\omega_{T}$, and we  refer to $\omega_{T}$ as a {\it scenario path} throughout the paper.  
	%In this paper, we are interested in the effectiveness of scenario paths and the conditional effectiveness of realizations along a scenario path. %Figure \ref{fig: notation}  depicts the distinction between a scenario path and the realizations along the path. % on a binary scenario tree in four stages. 

	A node $\omega_{t}$ ($t>1$), has a unique immediate ancestor in stage $ t-1 $, denoted by $a(\omega_t) \in \Omega_{t-1}$, and a node $\omega_{t}$ $(t<T)$  has a set of immediate children in stage $t+1$, denoted by  $\Cs{C}(\omega_t) \subset \Omega_{t+1}$.
	By definition, $\Cs{C}(\omega_{t})=\sset*{\omega_{t+1} \in \Omega_{t+1}}{a(\omega_{t+1})=\omega_{t}}$ and $\Omega_{t+1}=\bigcup_{\omega_{t} \in \Omega_{t}} \Cs{C}(\omega_{t})$. For a subset  $\Cs{S}_{t+1} \subseteq \Omega_{t+1}$, $t=1,\ldots, T-1$, we also define $a(\Cs{S}_{t+1}):=\sset*{\omega_{t} \in \Omega_{t}}{\exists \;  \omega_{t+1} \in \Cs{S}_{t+1} \ \suchthat \ a(\omega_{t+1})=\omega_{t}}$ to denote the set of all nodes $\omega_{t} \in \Omega_{t}$ that have at least one child in $\Cs{S}_{t+1}$. For $t=1, \ldots, T$, we  define a projection mapping $\Pi_{t}: \Omega_{T} \mapsto  \Omega_{t}$ which returns the stage-$t$ node that the scenario path $\omega_{T}$ is going through. By definition, we have $\Pi_{1}(\omega_{T})=\omega_{1}$ and  $\Pi_{T}(\omega_{T})=\omega_{T}$ for all $\omega_T \in \Omega_T$. %Figure \ref{fig: notation} illustrates this notation. 
	%, and $\Pi_{t}(\omega_{T})=\omega_{t}$ for all $\omega_{T} \in \Omega_{T}$.
	%Hence, two scenario paths $\omega^{\prime}_{T}$ and $\omega^{\prime\prime}_{T}$ are indistinguishable from each other at all levels $t^{\prime} \le t$ if  $H(t^{\prime},\omega^{\prime}_{T})=H(t^{\prime},\omega^{\prime\prime}_{T})$. 
	%We denote the set of nodes on the path from $\omega_{t}$ to a node $\omega_{T}$ by $\Cs{S}(\omega_{t})$. Hence, the set of nodes from a node $\omega_{1}$ to a node $\omega_{t-1}$ can be represented by $\Cs{S}(\omega_{1}) \setminus \Cs{S}(\omega_{t})$. %%By definition, a stage-$T$ scenario is a path from $\omega_{1}$ to a leaf node $\omega_{T}$. 
	%Let $\Cs{N}(\omega_{t})$ be the set of  scenario paths passing through a node $\omega_{t}$. 
	%That is, we have $\Cs{N}(\omega_{T})=\{\omega_{T}\}$ and $\Cs{N}(\omega_{t})=\bigcup_{\omega_{t+1} \in \Cs{C}(\omega_{t})} \Cs{N}(\omega_{t+1})$, for $t=1, \ldots, T-1$. 
	%The subtree rooted at node $\omega_{t}$ is denoted by  $\Cs{G}(\omega_{t})$. %=(\Cs{V}(\omega_{t}), \Cs{S}(\omega_{t})) $, where $\Cs{V}(\omega_{t})$ is the set of nodes reachable from stage-$T$ scenarios passing through node $\omega_{t}$, and $\Cs{S}$ denotes the set of arcs of the subgraph induced by $\Cs{V}(\omega_{t})$.
	%Now, let us define some shorthand notation for probabilities. %We refer to $\bs{q}$ as the nominal probability distribution. 
	%We use $\bs{q}_t(\omega_t)$ to denote the nominal stage-$t$ marginal probability  $\bs{q}_{t}(\hxi{t}^{\omega_{t}})$, i.e., $\bs{q}_t(\omega_t):=\nomProb{ \xi_{[t]}  = \xi_{[t]}^{\omega_t}  }$. 
	For $t=1, \ldots, T-1$, we  use  $q_{t+1| \omega_{t}}(\omega_{t+1}) $  to denote the  nominal 
	conditional  probability  of $\xi_{t+1}^{\omega_{t+1}}$ given $\xi_{[t]}^{\omega_{t}}$% with respect to $\bs{q}_{t+1|\hxi{t}^{\omega_{t}}}$
	, i.e., $q_{t+1| \omega_{t}}(\omega_{t+1}):=\textrm{Prob} \big\{ \xi_{t+1} = \xi_{t+1}^{\omega_{t+1}} \;|\;  \xi_{[t]} = \xi_{[t]}^{\omega_{t}} \big\}$, 
	We may  use $\cPcondnode{t}$ as a shorthand notation for the ambiguity set  $\Cs{P}_{t+1|\hxi{t}^{\omega_{t}}} $.
	We also use  $p_{t+1}(\omega_{t+1})$, $\bs{p}_{t+1} \in \cPcondnode{t}$, in a similar manner as the nominal conditional distribution $q_{t+1| \omega_{t}}(\omega_{t+1}) $.  We  use boldface letter to denote a probability vector. 

	\subsection{Background  on Effective Scenarios for a Two-Stage \dro}
	\label{sec: MULTI.background_def}
	
	Consider a special case of \eqref{eq: MULTI.robust} with $T=2$.  %and assume that there exists a finite set of possible scenarios that can occur. 
	To minimize notational overload, let us drop stage indices. Let $\Omega=\{\omega^1, \ldots, \omega^n\}$ denote the set of scenarios, with $\omega$ as a generic element in $\Omega$. Given this setup, a two-stage \dro\ problem can be written as 
	\begin{equation}
	\label{eq: robust}
	\min_{x \in \Cs{X} } \ \left\{f(x):=g(x)  + \max_{\bs{p} \in \Cs{P} }  \ 	\sum_{\omega \in \Omega} p(\omega) Q(x,\omega) 
	%\left[ \min_{x_{2} \in \Cs{X}_{2}}  g_{2}(x_{2}, \omega_2) \right] 
	\right\}, \tag{2-\dro}
	\end{equation}
	where $Q(x,\omega)$ is defined as in \eqref{eq: MULTI.VF} for a fixed $x \in \Cs{X}$ and $\omega \in \Omega$, and $\Cs{P}$ is an ambiguity set of probability distributions. 
	Let us consider a subset of scenarios $\Cs{S} \subset \Omega$. The idea behind identifying the effectiveness of scenarios in $\Cs{S}$ is to check whether the optimal value of \eqref{eq: robust} changes once these scenarios are ``removed" from the ambiguity set, or more precisely, they are forced to have a zero probability. % in the ambiguity set of distributions. 
	To that end,  the authors in \citep{rahimian2019} define the so-called {\it assessment problem of scenarios in $\Cs{S}$} as follows 
	\begin{equation}
	\label{eq: removal}
	\min_{x \in \Cs{X} } \ \left\{f^{\text{A}}(x; \Cs{S}):=g(x) + \max_{\bs{p} \in \Cs{P}^{\text{A}}(\Cs{S})} \ \sum_{\omega \in \Omega} p(\omega) Q(x,\omega) %\left[ \min_{x_{2} \in \Cs{X}_{2}}  g_{2}(x_{2}, \omega_2) \right]
	\right\},
	\end{equation}
	where $\Cs{P}^{\text{A}}(\Cs{S}):=\Cs{P} \cap \{p(\omega)=0, \ \omega \in \Cs{S}\}$ is the ambiguity set of distributions for the assessment problem of scenarios in $\Cs{S}$. 
	Observe that we have
	\begin{equation}
	\label{eq: cost_rel}
	\min_{x \in \Cs{X} } \ f^{\text{A}}(x; \Cs{S}) \le f^{\text{A}}(x^{*};\Cs{S}) \le \min_{x \in \Cs{X} } \ f(x),
	\end{equation}
	where $x^*$ is an optimal solution to \eqref{eq: robust}. %and  $\bar{x}$ is an optimal solution to \eqref{eq: removal}. 
	
	\begin{definition}{\citep[Definition~1]{rahimian2019}}
		\label{def: effective}
		A subset $\Cs{S} \subset \Omega$ is called effective if  $\min_{x  \in \Cs{X}}  f^{\text{A}}(x;\Cs{S}) < \min_{x \in \Cs{X}  } f(x)$. A subset $\Cs{S} \subset \Omega$ is called ineffective if it is not effective.
	\end{definition}
	%In words, a subset of scenarios $\Cs{S} \subset \Omega$   is called effective  if the optimal value of the corresponding assessment problem \eqref{eq: removal} is strictly smaller than the optimal value of \eqref{eq: robust}.  
	
	\begin{remark}
		\label{rem: sufficient_eff}
		A sufficient condition for a subset $\Cs{S}$ of scenarios to be  effective is $f^{\text{A}}(x^{*};\Cs{S}) < \min_{x \in \Cs{X} } \ f(x)$. If the assessment problem \eqref{eq: removal} is not well defined (i.e., $\Cs{P}^{\text{A}}(\Cs{S})$ is infeasible), the corresponding subset of
		scenarios is effective by definition. \qed
	\end{remark}
	
	%\begin{remark}
	%\label{rem: undefined}
	%    If the assessment problem \eqref{eq: removal} is not well defined, the corresponding subset of
	%scenarios is effective by definition. \qed
	%\end{remark}

	% % % % % % % % % % % % % % % % % % % % % % % % % % % % % % % % % % % % % % % % % % % %	
	\section{Effective Scenarios for a Multistage \dro}
	% % % % % % % % % % % % % % % % % % % % % % % % % % % % % % % % % % % % % % % % % % % %	
	\label{sec: MULTI.effective_scens}	
	
	%In this section, we introduce the notions of effectiveness of scenarios for  \eqref{eq: MULTI.robust} with a finite stochastic process. 
	In this section, we extend the notion of effective scenarios, introduced in \cite{rahimian2019} for a static \dro, to 
	\eqref{eq: MULTI.robust}. 
	%This notion is more subtle in a multistage setting.
	%We start this section by reviewing the results in \cite{rahimian2019}. 
	We address the effectiveness of scenarios for %a multistage \dro\ in the form of 
	\eqref{eq: MULTI.robust}  in two ways:  effectiveness of scenario paths (Section \ref{sec: MULTI.gen_eff_path}) and  conditional effectiveness of realizations  (Section \ref{sec: MULTI.gen_eff_realization}). % conditioned on the history of the stochastic process
	%. % for a two-stage/static setting. 
	%In Section  \ref{sec: MULTI.gen_eff_path}, we introduce the effectiveness of scenario paths, and in Section \ref{sec: MULTI.gen_eff_realization}, we study the conditional effectiveness of realizations along a path, conditioned on the history of the stochastic process. 

	\subsection{Effective Scenario Paths for a Multistage \dro}
	\label{sec: MULTI.gen_eff_path}
	
	The idea behind identifying the effectiveness of  scenario paths is to verify whether the optimal value of \eqref{eq: MULTI.robust} changes when a scenario path (or, more generally, a set of scenario paths)  is removed from the problem. 
	As in \citep{rahimian2019}, the first task is then  to define what is meant by ``removing" a set of scenario paths $\Cs{S} \subset \Omega_{T}$ from the problem.
	We remove scenario paths in $\Cs{S}$ by forcing the probability of these scenario paths to be zero. 
	By the construction of a scenario tree, the probability of a scenario path $\omega_T \in \Cs{S}$ is calculated by the product of conditional probabilities, associated with the nodes along the path.
	Thus, forcing the probability of a scenario path to zero can be achieved by setting the conditional probability of the corresponding leaf node $\omega_T \in \Omega_T$ to zero.  It is worth noting that setting the conditional probability of any other node along the path, besides the leaf node, to zero forces the probability of {\it all} scenario paths that are going through that node, to zero, which is undesirable. Moreover, because the ambiguity sets in \eqref{eq: MULTI.robust} are defined conditionally, we need to set the conditional probability of the corresponding leaf node $\omega_T$ to zero in an appropriate conditional ambiguity set, more precisely, in  $\Cs{P}_{T|a(\omega_T)}$.

	\subsubsection{Assessment Problem of Scenario Paths}
	\label{sec: MULTI.assessment_path}
	
	To identify the effectiveness of scenario paths in $\Cs{S} \subset \Omega_{T}$ we define an appropriate assessment problem. 
	In order to elaborate on the notion of removal of scenario paths and the assessment problem in the multistage setting, let us define more notation. 
	Recall the definition of $a(\Cs{S})$, the set of all nodes $\omega_{T-1} \in \Omega_{T-1}$ that have at least one child in $\Cs{S} \subset \Omega_{T}$.  %:=\Set*{\omega_{T-1}}{\exists \omega_{T} \in \Cs{S} \ \text{s.t.} \ a(\omega_{T})=\omega_{T-1}}$%and for any $\omega_{T-1} \in a(\Cs{S})$
	%, with some abuse of notation. In words, $a(\Cs{S})$ denotes the set of all nodes $\omega_{T-1} \in \Omega_{T-1}$ that have at least one child in $\Cs{F}$. 
	For each $\omega_{T-1} \in a(\Cs{S})$, let us define $\Cs{S}(\omega_{T-1}):=\Cs{C}(\omega_{T-1}) \cap \Cs{S}$ for the set of its children in $\Cs{S}$. Thus, %$\Cs{S}(\omega_{T-1})$, $\omega_{T-1} \in a(\Cs{S})$, gives a partition of scenario paths in $\Cs{S}$, i.e., 
	$\Cs{S}=\bigcup_{\omega_{T-1} \in a(\Cs{S})} \Cs{S}(\omega_{T-1})$. 
	
	For any $\omega_{T-1} \in a(\Cs{S})$, we remove the scenario paths in $\Cs{S}(\omega_{T-1})$ by restricting the ambiguity set $\Cs{P}_{T|\omega_{T-1}}$ to those probabilities $\bs{p}_{T}$ for which $p_{T}(\omega_{T})=0, \ \omega_{T} \in \Cs{S}(\omega_{T-1})$; see Figure \ref{fig: path_removal} for an illustration. 
	This ensures that the scenario paths in $\Cs{S}(\omega_{T-1})$ are not in the support of  any worst-case probability distribution $\bs{p}$
	%=[\bs{p}_{2|\hxi{1}},\bs{p}_{3|\hxi{2}}, \ldots, \bs{p}_{T|\hxi{T-1}}]$ 
	induced by $\xi$.
	We shall call the  problem that removes all scenario paths in $\Cs{S}$ the  {\it assessment problem of scenario paths in $\Cs{S}$}. 
	More formally, this problem can be formulated as
	\begin{equation}
	\label{eq: MULTI.robust_assessment}
	\min_{x_{1} \in \Cs{X}_{1}} \ g_{1}(x_{1}, \xi_{1})  + \max_{\bs{p}_{2} \in \Cs{P}_{2|\hxi{1}}} \ \mathbb{E}_{\bs{p}_{2}}   \left[   \ldots +  \max_{\bs{p}_{T} \in  \cPcondA{T}(\Cs{S})}  \ \mathbb{E}_{\bs{p}_{T}} \left[   \min_{x_{T} \in \Cs{X}_{T}} g_{T}(x_{T}, \xi_{T}) \right]  \ldots  \right],
	\end{equation}
	where for $\omega_{T-1} \in a(\Cs{S})$, the last-stage conditional ambiguity set of distributions for the assessment of the scenario paths in $\Cs{S}(\omega_{T-1})$ is defined as  $\cPcondAnode{T}(\Cs{S}):=\Cs{P}_{T|\omega_{T-1}} \cap \{p_{T}(\omega_{T})=0, \ \omega_{T} \in \Cs{S}(\omega_{T-1})\}$.  For $\omega_{T-1} \notin a(\Cs{S})$, we set  $\cPcondAnode{T}(\Cs{S})=\Cs{P}_{T|\omega_{T-1}}$. Notice that \eqref{eq: MULTI.robust_assessment} is similar to problem \eqref{eq: MULTI.robust}, except for  the last stage where $\cPcond{T}$ is replaced with $\cPcondA{T}(\Cs{S})$, with the interpretation mentioned above. 
	Moreover, notice that the assessment problem \eqref{eq: MULTI.robust_assessment} does not eliminate scenario paths from the sample space; it only enforces $p_{T}(\omega_{T})=0$, $\omega_{T} \in \Cs{S}(\omega_{T-1})$, for all $\omega_{T-1} \in a(\Cs{S})$.
	Also, when $\Cs{S}=\emptyset$, \eqref{eq: MULTI.robust_assessment} reduces to \eqref{eq: MULTI.robust}.

	\begin{figure}[!htbp]
		\centering
		\includegraphics[width=0.55\linewidth]{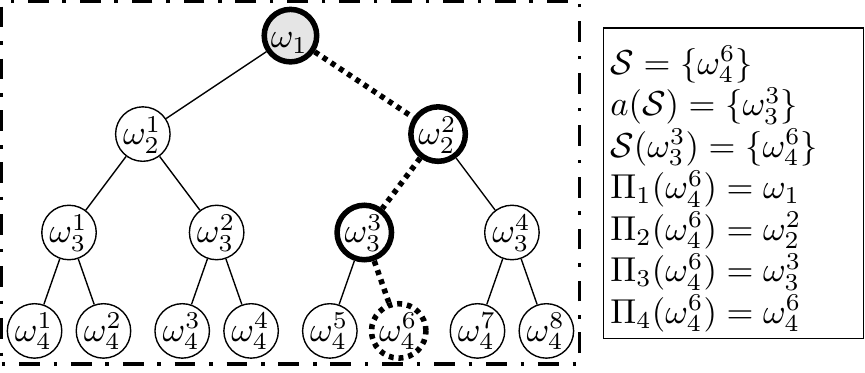}
		\caption{\label{fig: path_removal} Removal of the scenario  path $\omega_4^6 \in  \Omega_4$. %We have $a(\Cs{S})=\{\omega_3^3\}$ and $\Cs{S}(\omega_3^3)=\{\omega_4^6 \}$. 
			The path from the root node $\omega_1$ to the leaf node $\omega_4^6$ is shown with thick dotted lines, indicating the removal of this scenario path. To remove the scenario path $\omega_4^6$, we must add a constraint $p_{4}(\omega_4^6)=0$ to the ambiguity set $\Cs{P}_{4|\omega_3^3}$. For nodes not shown in thick black  circles, including $\omega_4^6$,  we use the same notation for the cost-to-go function $Q_{t}(\cdot, \cdot)$ at stages $t=2,3,4$, both in \eqref{eq: MULTI.robust} and the assessment problem  \eqref{eq: MULTI.robust_assessment}. % of the scenario path $\omega_4^6 $. 
			To check the effectiveness of the scenario path $\omega_4^6$, according to Definition \ref{def: MULTI.path_eff}, we compare the optimal values of \eqref{eq: MULTI.robust} and \eqref{eq: MULTI.robust_assessment} at the root node $\omega_1$, shown in gray.}
	\end{figure}

	\subsubsection{DP Reformulation of the Assessment Problem}
	\label{sec: MULTI.DP_path_assessment}
	In order to mathematically define the notion of effectiveness of scenario paths  for \eqref{eq: MULTI.robust}, we need to establish the relationship between the optimal values of the original problem \eqref{eq: MULTI.robust} and the assessment problem \eqref{eq: MULTI.robust_assessment}  precisely. To do this, we first write the DP reformulation of \eqref{eq: MULTI.robust_assessment} in a similar manner to those of \eqref{eq: MULTI.robust}. 
	
	%\subsection{Dynamic Programming Reformulation for the Assessment Problem of a Scenario Path}
	
	Consider the assessment problem \eqref{eq: MULTI.robust_assessment} at a given stage $t$, $t=1, \ldots, T$, when all information from previous stages (given by $\hx{t-1}$ and $\hxi{t}$) is known. Let $Q_{t}^{\text{A}}\!(\hx{t-1},\!\hxi{t};\!\Cs{S})$ denote the optimal value to this problem.  
	%That is, we have the following optimization problem:
	%\begin{equation}
	%\label{eq: MULTI.VF_path_t}
	%\begin{split}
	%Q_{t}^{\text{A}}(\hx{t-1}, \hxi{t}; \Cs{S}):=\min_{x_{t} \in \Cs{X}_{t}} \ g_{t}(x_{t}, \xi_{t}) + \max_{\bs{p}_{t+1} \in  \Cs{P}_{t+1|\hxi{t}} } \ \mathbb{E}_{\bs{p}_{t+1}}   \left[   
	%%\min_{x_{t+1} \in \Cs{X}_{t+1}}  g_{t+1}(x_{t+1}, \xi_{t+1}) + \right. \\ 
	%%\left. \max_{\bs{p}_{t+2} \in  \Cs{P}_{t+2|\hxi{t+1}} } \ \mathbb{E}_{\bs{p}_{t+2}} \Bigg[
	%\ldots + \right. \\
	%\left. \max_{\bs{p}_{T} \in \cPcondA{T}(\Cs{S}) } \ \mathbb{E}_{\bs{p}_{T}} \left[   \min_{x_{T} \in \Cs{X}_{T}} g_{T}(x_{T}, \xi_{T}) \right]  \ldots  \right].
	%\end{split}
	%\end{equation}
	Because there is no inner maximization problem at stage $T$ (i.e., $Q_{T+1}(\cdot, \cdot)\equiv 0$; see Section \ref{sec: MULTI.DP}),  we set  
	\begin{equation}
	\label{eq: MULTI.VF_path_DP_T}
	Q_{T}^{\text{A}}(\hx{T-1}, \hxi{T}; \Cs{S}):=Q_{T}(\hx{T-1}, \hxi{T}). 
	\end{equation}
	For $t=T-1$, on the other hand,  we have
	\begin{equation}
	\label{eq: MULTI.VF_path_DP_T-1}
	\begin{split}
	Q_{T-1}^{\text{A}}(\hx{T-2}, \hxi{T-1}; \Cs{S})=&\min_{x_{T-1} \in \Cs{X}_{T-1}} \!  g_{T-1}(x_{T-1}, \xi_{T-1}) + \\
	& \max_{\bs{p}_{T} \in \cPcondA{T}(\Cs{S}) } \! \ee{\bs{p}_{T}}{Q_{T}^{\text{A}}(\hx{T-1}, \hxi{T}; \Cs{S})}.
	\end{split}
	\end{equation}
	%Otherwise, the stage-$T-1$ cost function is the same as the one in \eqref{eq: MULTI.VF}, i.e., $F_{T-1}^{\text{A}}(\hx{T-2}, \hxi{T-1})=Q_{T-1}(\hx{T-2}, \hxi{T-1})$. 
	%For $t=T-2$, we have
	%\begin{equation}
	%Q_{T-2}^{\text{A}}(\hx{T-1}, \hxi{T-2}):=\min_{x_{T-2} \in \Cs{X}_{T-2}} \ g_{T-2}(x_{T-2}, \xi_{T-2}) + \max_{\bs{p}_{T-1} \in  \Cs{P}_{T-1|\hxi{T-2}}} \ee{\bs{p}_{T-1}}{F_{T}^{\text{A}}(\hx{T-2}, \hxi{T-1})},
	%\end{equation}
	%where $F_{T}^{\text{A}}(\hx{T-2}, \hxi{T-1})= Q_{T}^{\text{A}}(\hx{T-2}, \hxi{T-1})$ if $(\hxi{T-1},\xi_{T})=\hxi{T}$. Otherwise, $F_{T}^{\text{A}}(\hx{T-2}, \hxi{T-1})= Q_{T}(\hx{T-2}, \hxi{T-1})$. 
	Going backward in time, one can obtain $Q_{t}^{\text{A}}(\hx{t-1}, \hxi{t}; \Cs{S})$  for $t=T-2, \ldots, 1$, as 
	\begin{equation}
	\label{eq: MULTI.VF_path_DP_t}
	Q_{t}^{\text{A}}(\hx{t-1}, \hxi{t}; \Cs{S})=\min_{x_{t} \in \Cs{X}_{t}} \ g_{t}(x_{t}, \xi_{t}) + \max_{\bs{p}_{t+1} \in  \Cs{P}_{t+1|\hxi{t}} } \ee{\bs{p}_{t+1}}{Q_{t+1}^{\text{A}}(\hx{t}, \hxi{t+1}; \Cs{S})}.
	\end{equation}
	%At the first stage, solving the following problem
	%\begin{equation*}
	%\min_{x_{1} \in \Cs{X}_{1}} \ g_{1}(x_{1}, \xi_{1}) + \max_{\bs{p}_{2|\hxi{1}} \in \Cs{P}_{2|\hxi{1}}} \ee{\bs{p}_{2|\hxi{1}}}{Q_{2}^{\text{A}}(x_1, \hxi{2}; \Cs{S})},
	%\end{equation*}
	%gives the optimal value of \eqref{eq: MULTI.robust_assessment}.
	Equations \eqref{eq: MULTI.VF_path_DP_T}--\eqref{eq: MULTI.VF_path_DP_t} present the DP reformulation of \eqref{eq: MULTI.robust_assessment}.
	Note that the maximization in \eqref{eq: MULTI.VF_path_DP_T-1} is done over $\cPcondA{T}(\Cs{S})$, whereas it is done over $\Cs{P}_{t+1|\hxi{t}}$  in \eqref{eq: MULTI.VF_path_DP_t}. 
	%there are differences between \eqref{eq: MULTI.VF_path_DP_T-1} and \eqref{eq: MULTI.VF_path_DP_t}. 
	For brevity, we  denote the optimal value of  \eqref{eq: MULTI.robust_assessment} as $Q_{1}^{\text{A}}(\Cs{S})$.
	Throughout the paper, we use notation $Q_{t}^{\text{A}}(\cdot, \cdot; \Cs{S})$ to differentiate the cost-to-go functions for the assessment problem of the scenario paths in $\Cs{S}$, presented in this section,  from the cost-to-go functions $Q_{t}(\cdot, \cdot)$ of the original problem \eqref{eq: MULTI.robust}, presented in Section \ref{sec: MULTI.DP}.

	\begin{remark}
		\label{rem: MULTI.DP_path}
		In stages $t=2, \ldots, T-1$, we have $Q_{t}^{\text{A}}(\hx{t-1},\hxi{t}^{\omega_{t}};\Cs{S})\!=\!Q_{t}(\hx{t-1}, \hxi{t}^{\omega_{t}})$ if %there is not any  $(\omega_{t+1}, \ldots, \omega_{T})$ such that $(\hxi{t}^{\omega_{t}}, \xi_{t+1}^{\omega_{t+1}}, \ldots, \xi_{T}^{\omega_{T}})$ could be associated with a $\omega_T \in   \Cs{S}$. 
		no scenario path $\omega_T \in \Cs{S}$ is going through $\omega_t$. 
		%, i.e., $H(t,\omega_{T}) \neq \omega_{t}$ for all $\omega_{T} \not\in \Cs{S}$.
		In other words, 
		\begin{equation}
		\label{eq: MULTI.Proj_Observation}
		%Q_{t}^{\text{A}}(\hx{t-1}, \hxi{t}^{\omega_{t}};\Cs{S})=Q_{t}(\hx{t-1}, \hxi{t}^{\omega_{t}}) \ \text{if} \ \omega_{t} \notin \bigcup_{\omega_{T} \in \Cs{S}} \Set*{ \Pi_{t}(\omega_{T}) }{\exists \; \hat{\omega}_{T} \in \Cs{S} \ \st \ \Pi_{t}(\omega_{T}) = \Pi_{t}(\hat{\omega}_{T}) }
		Q_{t}^{\text{A}}(\hx{t-1}, \hxi{t}^{\omega_{t}};\Cs{S})=Q_{t}(\hx{t-1}, \hxi{t}^{\omega_{t}}) \quad \text{if} \quad \omega_{t} \notin \bigcup_{\omega_{T} \in \Cs{S}} \Pi_{t}(\omega_{T}).
		\end{equation}
		%for any feasible policy $x$ to \eqref{eq: MULTI.robust}. 
		This property, illustrated in Figure \ref{fig: path_removal}, will be frequently used in the proofs in Section \ref{sec: MULTI.thms}.
		%$\omega_{t}$ is not the grandancestor of any of the  scenario paths $\omega_{T} \in \Cs{S}$.
	\end{remark}

	\subsubsection{Definition and Properties}
	
	In this section, we  precisely define  the effectiveness of scenario paths for \eqref{eq: MULTI.robust} and state some relevant properties. 
	%In order to  do this, we need to compare the optimal values of \eqref{eq: MULTI.robust_assessment} and \eqref{eq: MULTI.robust}, for which we rely on the respective DP reformulations. 
	
	%Observe that \eqref{eq: MULTI.robust_assessment} is more restricted than \eqref{eq: MULTI.robust}. Thus, the optimal value of \eqref{eq: MULTI.robust_assessment} is smaller than or equal to the optimal value of \eqref{eq: MULTI.robust} (we shall shortly show this precisely). This observation helps us to define effective scenario paths in words. 
	%A subset of scenario paths $\Cs{S} \subset \Omega_{T}$   is called {\it effective}  if the optimal value of the corresponding assessment problem \eqref{eq: MULTI.robust_assessment} is strictly smaller than the optimal value of \eqref{eq: MULTI.robust}. 

	\begin{proposition}
		\label{prop: compare_org_asses_path}
		Consider two sets $\Cs{S}^{1} \subseteq \Cs{S}^{2} \subset \Omega_{T}$. Then, $Q_{1}^{\text{A}}(\Cs{S}^2) \le  Q_{1}^{\text{A}}(\Cs{S}^1) $.  %$Q_{1}^{\text{A}}(\Cs{S}^2)$, the optimal value of the assessment problem \eqref{eq: MULTI.robust_assessment} for $\Cs{S}^2$, is smaller than or equal to  $Q_{1}^{\text{A}}(\Cs{S}^1)$, the optimal value of the assessment problem \eqref{eq: MULTI.robust_assessment} for $\Cs{S}^1$, i.e., $Q_{1}^{\text{A}}(\Cs{S}^2) \le  Q_{1}^{\text{A}}(\Cs{S}^1) $. 
	\end{proposition}
	
	\begin{proof}
		Let $x^{*}(\Cs{S}^{1}):=[x_1^*(\Cs{S}^{1}), \ldots, x_T^*(\Cs{S}^{1})]$ denote an optimal policy obtained by solving the DP reformulation of \eqref{eq: MULTI.robust_assessment} for $\Cs{S}^{1}$. %, and $\bar{x}(\Cs{S}^{2}):=[\bar{x}_1(\Cs{S}^{2}), \ldots, \bar{x}_T(\Cs{S}^{2})]$ denote an optimal policy to \eqref{eq: MULTI.robust_assessment} for $\Cs{S}^{2}$. %Also, let $x=[x_1, \ldots, x_T]$ be any feasible policy to \eqref{eq: MULTI.robust}. 
		For any $\hx{T-2}$ and $\hxi{T-1}$, the corresponding worst-case expected value problem at stage $T-1$
		in \eqref{eq: MULTI.robust_assessment} for $\Cs{S}^{2}$ is more restricted than the corresponding problem in \eqref{eq: MULTI.robust_assessment} for $\Cs{S}^{1}$. 
		This, combined with the  suboptimality of $x^{*}(\Cs{S}^{1})$ to problem \eqref{eq: MULTI.robust_assessment} for $\Cs{S}^{2}$, implies that 
		%for $\hstarx{T-1}(\Cs{S}_{1})$ and $\hxi{T}$, 
		\begin{equation*}
		\begin{array}{ll}
		& Q_{T-1}^{\text{A}}\Big(\hstarx{T-2}(\Cs{S}^{1}), \hxi{T-1}; \Cs{S}^{2}\Big) \\
		& \le g_{T-1}\Big(x^{*}_{T-1}(\Cs{S}^{1}), \xi_{T-1}\Big) + \max_{\bs{p}_{T} \in \cPcondA{T}(\Cs{S}^{2}) } \ee{\bs{p}_{T}}{Q_{T}\Big(\hstarx{T-1}(\Cs{S}^{1}), \hxi{T}\Big)} \\
		& \le g_{T-1}\Big(x^{*}_{T-1}(\Cs{S}^{1}), \xi_{T-1}\Big) + \max_{\bs{p}_{T} \in \cPcondA{T}(\Cs{S}^{1}) } \ee{\bs{p}_{T}}{Q_{T}\Big(\hstarx{T-1}(\Cs{S}^{1}), \hxi{T}\Big)} \\
		& = Q_{T-1}^{\text{A}}\Big(\hstarx{T-2}(\Cs{S}^{1}), \hxi{T-1}; \Cs{S}^{1}\Big),
		\end{array}
		\end{equation*}
		where we also used the fact that $Q_{T}^{\textrm{A}}\Big(\hstarx{T-1}(\Cs{S}^{1}), \hxi{T};\Cs{S}^{2}\Big)=Q_{T}\Big(\hstarx{T-1}(\Cs{S}^{1}), \hxi{T}\Big)$. The equality above is due to the time consistency of $x^{*}(\Cs{S}^{1})$  (recall Remark \ref{rem: MULTI.consitency_eoh}). 
		Going backward in time for $t=T-2, \ldots, 2$, we have 
		\begin{equation*}
		\begin{array}{ll}
		&  Q_{t}^{\text{A}}\Big(\hstarx{t-1}(\Cs{S}^{1}), \hxi{t}; \Cs{S}^{2}\Big)  \\ 
		& \le g_{t}\Big(x^{*}_{t}(\Cs{S}^{1}), \xi_{t}\Big) + \max_{\bs{p}_{t+1} \in \cPcond{t} } \ee{\bs{p}_{t+1}}{Q_{t+1}^{\text{A}}\Big(\hstarx{t}(\Cs{S}^{1}), \hxi{t+1}; \Cs{S}^{2}\Big)} \\
		& \le g_{t}\Big(x^{*}_{t}(\Cs{S}^{1}), \xi_{t}\Big) + \max_{\bs{p}_{t+1} \in \cPcond{t} } \ee{\bs{p}_{t+1}}{Q_{t+1}^{\text{A}}\Big(\hstarx{t}(\Cs{S}^{1}), \hxi{t+1}; \Cs{S}^{1}\Big)} \\
		& = Q_{t}^{\text{A}}\Big(\hstarx{t-1}(\Cs{S}^{1}), \hxi{t}; \Cs{S}^{1}\Big),
		\end{array}
		\end{equation*}
		where the second inequality is due to $Q_{t+1}^{\text{A}}\Big(\!\hstarx{t}(\!\Cs{S}^{1}\!),\hxi{t+1};\!\Cs{S}^{2}\!\Big)\!\le\! Q_{t+1}^{\text{A}}\Big(\!\hstarx{t}(\!\Cs{S}^{1}\!), \hxi{t+1};\!\Cs{S}^{1}\!\Big)$ and the equality is due to the time consistency of $x^{*}(\Cs{S}^{1})$. 
		Consequently, we have
		\begin{equation*}
		\begin{array}{ll}
		Q_{1}^{\text{A}}(\Cs{S}^{2}) & = \min_{x_{1} \in \Cs{X}_{1}} \ g_{1}(x_{1}, \xi_{1}) + \max_{\bs{p}_{2} \in \Cs{P}_{2|\hxi{1}}} \ee{\bs{p}_{2}}{Q_{2}^{\text{A}}(\hx{1}, \hxi{2}; \Cs{S}^{2})} \\
		& \le 
		g_{1}\Big(x^{*}_{1}(\Cs{S}^{1}), \xi_{1}\Big) + \max_{\bs{p}_{2} \in \Cs{P}_{2|\hxi{1}}} \ee{\bs{p}_{2}}{Q_{2}^{\text{A}}\Big(x^{*}_{[1]}(\Cs{S}^1), \hxi{2}; \Cs{S}^{2}\Big)}\\
		& \le
		g_{1}\Big(x^{*}_{1}(\Cs{S}^{1}), \xi_{1}\Big) + \max_{\bs{p}_{2} \in \Cs{P}_{2|\hxi{1}}} \ee{\bs{p}_{2}}{Q_{2}^{\text{A}}\Big(x^{*}_{[1]}(\Cs{S}^1), \hxi{2}; \Cs{S}^{1}\Big)}\\
		& = Q_{1}^{\text{A}}(\Cs{S}^{1}),
		\end{array}
		\end{equation*}
		where the first inequality is due to the  suboptimality of $x^{*}(\Cs{S}^{1})$ to problem \eqref{eq: MULTI.robust_assessment}. 
	\end{proof}
	
	By taking $\Cs{S}^{2}=\Cs{S}$ and $\Cs{S}^{1}=\emptyset$ in the proof of  Proposition \ref{prop: compare_org_asses_path}, we have 
	\begin{equation}
	\label{eq: MULTI.cost_rel}
	Q_{1}^{\text{A}}(\Cs{S}) \le  g_{1}(x^{*}_{1}, \xi_{1}) + \max_{\bs{p}_{2} \in \Cs{P}_{2|\hxi{1}}} \ee{\bs{p}_{2}}{Q_{2}^{\text{A}}(x^{*}_{[1]}, \hxi{2}; \Cs{S})} \le Q_1,   
	\end{equation}
	where $Q_1$ is the optimal value of \eqref{eq: MULTI.robust} and $x^*$ is an optimal policy to \eqref{eq: MULTI.robust}. With this relationship, we can now define effective scenario paths. %, generalizing Definition \ref{def: effective} for \eqref{eq: robust} to \eqref{eq: MULTI.robust}. 

	\begin{definition}[Effectiveness of scenario paths]
		\label{def: MULTI.path_eff}
		%At an optimal policy $x^{*}=[x_1^*, \ldots, x_T^*]$, a subset of scenario paths $\Cs{S} \subset \Omega_{T}$   is called {\it effective}  if by its removal the optimal value of the corresponding assessment problem is strictly smaller than the optimal value of T-\drV. A scenario path is \underline{ineffective} if it is not effective.
		A subset of scenario paths $\Cs{S} \subset \Omega_{T}$   is called {\it effective}  if $Q_{1}^{\text{A}}(\Cs{S}) < Q_{1}$.
		%$\min_{x_{1} \in \Cs{X}_{1}} \ g_{1}(x_{1}, \xi_{1}) + \max_{\bs{p}_{2|\hxi{1}} \in \Cs{P}_{2|\hxi{1}}} \ee{\bs{p}_{2|\hxi{1}}}{Q_{2}^{\text{A}}(\hx{1}, \hxi{2}; \Cs{S})} < \min_{x_{1} \in \Cs{X}_{1}} \ g_{1}(x_{1}, \xi_{1}) + \max_{\bs{p}_{2|\hxi{1}} \in \Cs{P}_{2|\hxi{1}}} \ee{\bs{p}_{2|\hxi{1}}}{Q_{2}(\hx{1}, \hxi{2})}$. 
		A subset of scenario paths is called {\it ineffective} if it is not effective.
	\end{definition}
	%In the above definition, a scenario path is removed by forcing the probability of $\xi_{T}$ to be zero in the ambiguity set $\cPcondgamma{T}$, i.e., $\bs{p}_{T}(\xi_{T})=0$. 
	
	In words, a subset of scenario paths $\Cs{S} \subset \Omega_{T}$   is called effective if the optimal value of the corresponding assessment problem \eqref{eq: MULTI.robust_assessment} is strictly smaller than the optimal value of \eqref{eq: MULTI.robust}.
	%(see Figure \ref{fig: path_removal} for an illustration of the comparison). 
	Similar to Remark \ref{rem: sufficient_eff}  for \eqref{eq: robust}, 
	%\begin{remark}
	%    \label{rem: MULTI.sufficient_eff}
	a sufficient condition for a subset of scenario paths $\Cs{S} \subset \Omega_T$ to be  effective is $g_{1}(x^{*}_{1}, \xi_{1}) + \max_{\bs{p}_{2} \in \Cs{P}_{2|\hxi{1}}} \ee{\bs{p}_{2}}{Q_{2}^{\text{A}}(x^{*}_{[1]}, \hxi{2}; \Cs{S})} < Q_1$.
	%\end{remark}
	Moreover, 
	%\begin{remark}
	%\label{rem: MULTI.undefined}
	the assessment problem of scenario paths in $\Cs{S}$ might not be well defined. For instance, if for some $\omega_{T-1} \in a(\Cs{S})$,  too many scenario paths are restricted to have a zero worst-case probability (e.g.,  $\Cs{S}(\omega_{T-1})= \Cs{C}(\omega_{T-1})$), then, the inner maximization problem at stage $T-1$ might become infeasible. In this case, we set the optimal value of  \eqref{eq: MULTI.robust_assessment} to $+\infty$ by convention and  $\Cs{S}$ is effective by definition.  
	%\end{remark}
	
	One can conjecture that the effectiveness of a subset of scenario paths might be affected in interaction with other subsets of scenario paths. The following proposition addresses the effectiveness of  union of an effective subset of scenario paths and  intersection of an ineffective subset of scenario paths with any other subset of $\Omega_{T}$.

	\begin{proposition}
		\label{prop: uni_int}
		%\begin{enumerate}[label=(\roman*),noitemsep]
		%\item \label{prop3i} 
		(i) The union of an effective subset of scenario paths with any other subset of $\Omega_{T}$ is effective.
		%\item \label{prop3ii} 
		(ii) The intersection of an ineffective subset of scenario paths with any other subset of $\Omega_{T}$ is ineffective.
		%\end{enumerate}
	\end{proposition}
	
	\begin{proof}
		The proof is similar to \citep[Proposition~2]{rahimian2019}, but it uses Proposition \ref{prop: compare_org_asses_path} and Definition \ref{def: MULTI.path_eff}. For brevity, we skip the details. 
		\begin{comment}
		
		To prove (i), suppose $\Cs{S}^{1}$ is an effective subset and $\Cs{S}^{2}$ is an arbitrary subset of $\Omega_{T}$. 
		First, because $\Cs{S}^{1} \cup \Cs{S}^{2}  \supseteq  \Cs{S}^{1}$, we have $Q_{1}(\Cs{S}^{1} \cup \Cs{S}^{2}) \le Q_{1}(\Cs{S}^{1})$. On the other hand, $Q_{1}(\Cs{S}^{1}) < Q_{1}$ because $\Cs{S}^{1}$ is  effective. These imply $Q_{1}(\Cs{S}^{1} \cup \Cs{S}^{2}) < Q_{1}$, and hence, 
		$\Cs{S}^{1} \cup \Cs{S}^{2}$ is effective by Definition \ref{def: MULTI.path_eff}.
		To prove (ii), suppose $\Cs{S}^{1}$ is an ineffective subset and $\Cs{S}^{2}$ is an arbitrary subset of $\Omega_{T}$. 
		First, because $\Cs{S}^{1} \supseteq \Cs{S}^{1} \cap \Cs{S}^{2} \supseteq \emptyset$, we have 
		$Q_{1}(\Cs{S}^{1}) \le Q_{1}(\Cs{S}^{1} \cap \Cs{S}^{2}) \le Q_{1}$. On the other hand, $Q_{1}(\Cs{S}^{1})=Q_{1}$ because $\Cs{S}^{1}$ is ineffective. These imply $Q_{1}(\Cs{S}^{1} \cap \Cs{S}^{2})=Q_{1}$, and  hence $\Cs{S}^{1} \cap \Cs{S}^{2}$ is ineffective by Definition \ref{def: MULTI.path_eff}.
		\end{comment}
	\end{proof}
	
	\begin{corollary}
		\label{cor: subset}
		A subset of an ineffective subset of scenario paths is ineffective.
	\end{corollary}
	
	%\begin{proof}
	%	The proof is immediate from Proposition \ref{prop: uni_int}\ref{prop3ii}. 
	%\end{proof}

	It is worth noting that Definition \ref{def: MULTI.path_eff} for \eqref{eq: MULTI.robust} is not an immediate extension of Definition \ref{def: effective} for \eqref{eq: robust}, and there are subtle but important differences.  To explain this, let us focus on \eqref{eq: MULTI.cost_rel} and \eqref{eq: cost_rel}, the relationships that form the bases for establishing these definitions. In particular, let us examine the  terms $g_{1}(x^{*}_{1}, \xi_{1}) + \max_{\bs{p}_{2} \in \Cs{P}_{2|\hxi{1}}}  \ee{\bs{p}_{2}}{Q_{2}^{\text{A}}(x^{*}_{[1]}, \hxi{2}; \Cs{S})}$ in \eqref{eq: MULTI.cost_rel} and $f^{\text{A}}(x^*\!,\!\Cs{S})= g(x^*) + \max_{\bs{p} \in \Cs{P}^{\text{A}}(\Cs{S})}  \sum_{\omega \in \Omega} p(\omega) Q(x^*,\omega)$ in \eqref{eq: cost_rel} in more detail. On the one hand, the ambiguity set  $\Cs{P}_{2|\hxi{1}}$ is not impacted by the removed subset of scenario paths $\Cs{S}$, whereas $\Cs{P}^{\text{A}}(\Cs{S})$ is a restricted set and impacted by $\Cs{S}$. On the other hand, the cost-to-go function   $Q_{2}^{\text{A}}(x^{*}_{[1]}, \hxi{2}; \Cs{S})$ is impacted by the removed subset of scenario paths $\Cs{S}$, whereas $Q(x^*,\omega)$ is not. So, \eqref{eq: MULTI.cost_rel} does not follow from \eqref{eq: cost_rel}. 
	These differences are themselves due to the differences of the assessment problems \eqref{eq: removal} and \eqref{eq: MULTI.robust_assessment}. 
	By a similar reasoning, there are subtle, important differences in Proposition \ref{prop: uni_int} and \citep[Proposition~2]{rahimian2019}.

	\subsection{Conditionally Effective Realizations for a Multistage \dro}
	\label{sec: MULTI.gen_eff_realization}
	
	%In Section \ref{sec: MULTI.gen_eff_path}, we discussed that although the notion of effective/ineffective scenario paths for \eqref{eq: MULTI.robust} is defined with a similar idea (i.e., removal of scenario paths by forcing their probability to zero), as that for \eqref{eq: robust},  there are key differences.
	In this section, we define conditional effectiveness of realizations along a scenario path, given the available information on the history of the stochastic process and decisions. We explain that this notion, unlike the effectiveness of scenario paths, resembles the effectiveness of scenarios for \eqref{eq: robust}. %, albeit surprisingly at the first glance. 
	
	%Recall that in order to identify effective scenario paths $\Cs{S} \subset \Omega_T$,  one has to compare the optimal values of  \eqref{eq: MULTI.robust} and \eqref{eq: MULTI.robust_assessment}. 
	To identify {\it conditionally effective realizations}, as it might be perceived from the name, first, one has to consider all information from previous stages, and then, evaluate what would happen to the optimal cost going forward if a subset of realizations is conditionally removed. %Similar to the assessment problem of scenario path as defined in \eqref{eq: MULTI.robust_assessment}, 
	Suppose that $\hx{t-1}$ and $\hxi{t}$, $t=1, \ldots, T-1$, are given.
	The idea behind identifying conditionally effective realizations is to verify whether the cost-to-go function  \eqref{eq: MULTI.VF} at stage $t$ changes when a realization   (or, more generally, a set of realizations) in stage $t+1$  is conditionally removed from the problem. 
	%As before, the first task is then  to define what is meant by ``conditionally removing" a set of realizations $\Cs{S}_{t+1} \subset \Omega_{t+1}$ from the problem,  $t=1, \ldots, T-1$. 
	We ``conditionally remove" realizations  $\Cs{S}_{t+1} \subset \Omega_{t+1}$ by forcing the conditional probability of these realizations to be zero, conditioned on the history of the stochastic process. 
	%Recognizing that the past information on the history of the stochastic process is already available, two points are in order to check the conditional effectiveness of a realization $\omega_{t+1} \in \Cs{S}_{t+1}$. 
	
	Note that because the ambiguity sets in \eqref{eq: MULTI.robust} are defined conditionally, we need to set the conditional probability of $\omega_{t+1} \in \Cs{S}_{t+1}$ to zero in an appropriate conditional ambiguity set, more precisely, in $\Cs{P}_{t+1|a(\omega_{t+1})}$.  Moreover, because $\omega_{t+1}$ is forced to have a zero conditional probability, all scenario paths that are going through $\omega_{t+1}$ will have a zero conditional probability, conditioned on $a(\omega_{t+1})$. Hence, all such scenario paths will have a zero probability. This implies that the way the removal of scenario paths is defined in Section \ref{sec: MULTI.gen_eff_path} is different from how we define the conditional removal of realizations in this section. 
	Furthermore, once scenario paths are removed in Section \ref{sec: MULTI.gen_eff_path}, we compare the optimal objective function values at the root node $\omega_1$. Here, we look for changes in the cost-to-go value functions $Q_{t}(\hx{t-1}, \hxi{t})$.
	In this section, we follow a similar  process to what did in Section \ref{sec: MULTI.gen_eff_path}; nevertheless, the developments have different interpretations.

	\subsubsection{Conditional Assessment Problem for Realizations}
	\label{sec: MULTI.assessment_realization}
	
	%As before, to identify the conditional effectiveness of realizations in $\Cs{S}_{t+1} \subset \Omega_{t+1}$, $t=1, \ldots, T-1$, we need to define an appropriate assessment problem. 
	Suppose that the history of the stochastic process up to and including stage $t$, $t=1, \ldots, T-1$, is given. 
	%To discuss the notion of the conditional removal of realizations and the desired assessment problem, we define some notation, similar to those we defined in Section \ref{sec: MULTI.assessment_path}.  
	%Recall the definition of $a(\Cs{S}_{t+1})$. 
	Recall that $\Cs{S}_{t+1} \subset \Omega_{t+1}$ is the set of realizations to be conditionally removed. 
	For each $\omega_{t} \in a(\Cs{S}_{t+1})$, let us define $\Cs{S}(\omega_{t}):=\Cs{C}(\omega_{t}) \cap \Cs{S}_{t+1}$ to  denote the set of all children of $\omega_{t}$ in $\Cs{S}_{t+1}$. Thus, %$\Cs{S}(\omega_{t})$, $\omega_{t} \in a(\Cs{S}_{t+1})$, give a partition of realizations in $\Cs{S}_{t+1}$, i.e., 
	$\Cs{S}_{t+1}=\bigcup_{\omega_{t} \in a(\Cs{S}_{t+1})} \Cs{S}(\omega_{t})$. 
	For any $\omega_t \in a(\Cs{S}_{t+1})$, 
	%Let us assume that the history $\hxi{t}$ is associated with node $\omega_t \in \Omega_t$, $t=1, \ldots, T-1$. 
	we conditionally remove the realizations  in $\Cs{S}(\omega_{t})$ by restricting the ambiguity set $\Cs{P}_{t+1|\omega_{t}}$ to those probabilities $\bs{p}_{t+1}$ for which $p_{t+1}(\omega_{t+1})=0, \ \omega_{t+1} \in \Cs{S}(\omega_{t})$ (see Figure \ref{fig: realization_removal} for an illustration). 
	This ensures that all scenario paths that are going through any node in $\Cs{S}(\omega_{t})$ as part of their process are not in the support of  any  worst-case probability distribution $\bs{p}$ %$=[\bs{p}_{2|\hxi{1}},\bs{p}_{3|\hxi{2}}, \ldots, \bs{p}_{T|\hxi{T-1}}]$ 
	induced by $\xi$.
	Conditioned on $\hx{t-1}$ and $\hxi{t}$, we define %shall call the  problem that removes all realizations  in  $\Cs{S}_{t+1}$ 
	the  {\it conditional assessment problem of realizations in $\Cs{S}_{t+1}$} as 
	%More formally,  for $t=1, \ldots, T-1$, this problem can be formulated as
	\begin{equation}
	\label{eq: MULTI.VF_t_assessment}
	\begin{split}
	Q_{t}^{\text{CA}}(\hx{t-1}, \hxi{t}; \Cs{S}_{t+1}):=& \min_{x_{t} \in \Cs{X}_{t}} \ g_{t}(x_{t}, \xi_{t}) + \max_{\condbs{p}{t+1} \in  \cPcondCA{t}(\Cs{S}_{t+1})} \ \mathbb{E}_{\condbs{p}{t+1}}    \Bigg[ \ldots + \Bigg.\\
	\Bigg. & \max_{\bs{p}_{T} \in \Cs{P}_{T|\hxi{T-1}}  } \ \mathbb{E}_{\bs{p}_{T}} \left[   \min_{x_{T} \in \Cs{X}_{T}} g_{T}(x_{T}, \xi_{T}) \right]  \ldots \Bigg],
	\end{split}
	\end{equation}
	where for $\omega_t \in a(\Cs{S}_{t+1})$,  the conditional ambiguity set of distributions for the conditional assessment problem of realizations in $\Cs{S}_{t+1}$ is defined as $\cPcondCAnode{t}(\Cs{S}_{t+1}):=\cPcondnode{t} \cap \Big\{p_{t+1|\omega_{t}}(\omega_{t+1})=0, \ \omega_{t+1} \in \Cs{S}(\omega_{t}) \Big\}$. For  $\omega_t \notin a(\Cs{S}_{t+1})$, $\cPcondCAnode{t}(\Cs{S}_{t+1}):=\cPcondnode{t}$. 
	Notice that problem \eqref{eq: MULTI.VF_t_assessment} is similar to problem  \eqref{eq: MULTI.VF}, except for in stage $t+1$ where $\cPcond{t}$ is replaced with $\cPcondCA{t}(\Cs{S}_{t+1})$, with the interpretation mentioned above.  
	Moreover, notice that the assessment problem \eqref{eq: MULTI.VF_t_assessment} does not eliminate the realizations $\Cs{S}_{t+1}$ from the sample space $\Omega_{t+1}$; it only enforces $p_{t+1|\omega_{t}}(\omega_{t+1})=0$,   $\omega_{t+1} \in \Cs{S}(\omega_{t})$, for $\omega_t \in a(\Cs{S}_{t+1})$. As before, when $\Cs{S}_{t+1}=\emptyset$, \eqref{eq: MULTI.VF_t_assessment} reduces to \eqref{eq: MULTI.VF}.

	%Problem \eqref{eq: MULTI.VF_t_assessment} is more restricted than \eqref{eq: MULTI.VF}. Thus, the optimal value of \eqref{eq: MULTI.VF_t_assessment} is smaller than or equal to the optimal value of \eqref{eq: MULTI.VF} (we shall shortly show this precisely). Before getting into the details, this observation helps us to define conditional effective realizations in words. 
	%

	\begin{figure}
		\centering
		\includegraphics[width=0.55\linewidth]{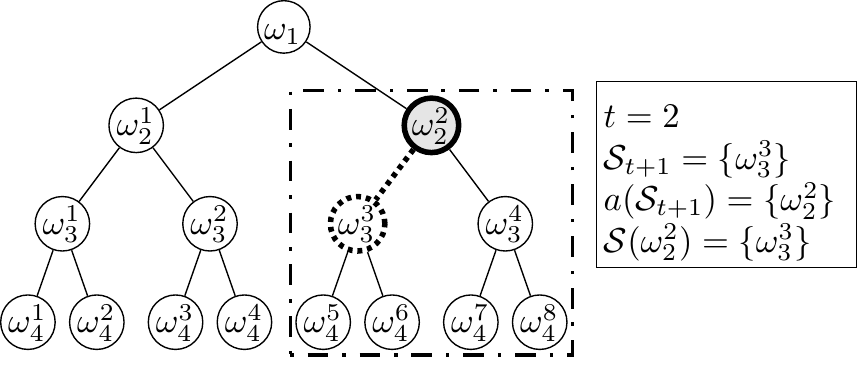}
		\caption{\label{fig: realization_removal} Conditional removal of  realization $\omega_3^3 \in \Omega_3$. The path from  node $\omega_3^3$ to its ancestor node $\omega_2^2=a(\omega_3^3)$ is shown with a thick dotted line, indicating the conditional removal of the  realization $\omega_3^3$. To conditionally remove the realization $\omega_3^3$ we must add a constraint $p_{3|\omega_2^2}(\omega_3^3)=0$ to the ambiguity set $\Cs{P}_{3|\omega_2^2}$. For nodes not shown in  thick black circles, including $\omega_3^3$,  we use the same notation for the cost-to-go function $Q_{t}(\cdot, \cdot)$ at stages $t=2,3,4$, both in \eqref{eq: MULTI.VF} and the conditional assessment problem  \eqref{eq: MULTI.VF_t_assessment}. % of the realization $\omega_3^2$.
			To check the conditional effectiveness of the realization $\omega_3^3$, we focus on the subtree rooted at $\omega_2^2=a(\omega_3^3)$, shown by the dotted box, and compare the optimal values of \eqref{eq: MULTI.VF} and \eqref{eq: MULTI.VF_t_assessment} (or, equivalently, \eqref{eq: MULTI.VF_assessment}) at the node $\omega_2^2$, shown in gray, given $x_1$.}
	\end{figure}

	\subsubsection{DP Reformulation of the Assessment Problem}
	\label{sec: MULTI.DP_realization_assessment}
	To mathematically define conditionally effective realizations for \eqref{eq: MULTI.robust}, we need to establish the relationship between the optimal values of  \eqref{eq: MULTI.VF} and \eqref{eq: MULTI.VF_t_assessment}. As in for the effectiveness of scenario paths in Section \ref{sec: MULTI.DP_path_assessment}, we first need to derive the DP reformulation of \eqref{eq: MULTI.VF_t_assessment}. 
	
	Consider the optimization problem \eqref{eq: MULTI.VF_t_assessment} at a given stage $t$,  $t=1, \ldots, T-1$, when all information from previous stages (given by $\hx{t-1}$ and $\hxi{t}$) is known. By the construction of \eqref{eq: MULTI.VF_t_assessment}, forcing the conditional probability of a realization at stage $t+1$ to zero does not affect the way the cost-to-go functions at stage $t^{\prime}$, for  $t^{\prime} > t$, are written. 
	%Recall that removal of realizations $\Cs{S}_{t+1}$ does not affect $t^{\prime}$-th-stage cost function for  $t^{\prime} > t$. 
	Thus,  we have  
	\begin{equation}
	\label{eq: MULTI.VF_assessment}
	Q_{t}^{\text{CA}}(\hx{t-1}, \hxi{t}; \Cs{S}_{t+1}):=\min_{x_{t} \in \Cs{X}_{t}} \ g_{t}(x_{t}, \xi_{t}) + \max_{\condbs{p}{t} \in  \cPcondCA{t}(\Cs{S}_{t+1})} \ \ee{\condbs{p}{t}}{Q_{t+1}(\hx{t},\hxi{t+1})},
	\end{equation}
	where $Q_{t^{\prime}+1}(\hx{t^{\prime}},\hxi{t^{\prime}+1})$, $t^{\prime}=t, \ldots, T$, are defined as in \eqref{eq: MULTI.VF}. 
	For brevity, we denote the optimal value of the conditional assessment problem of the realizations $\Cs{S}_{2} \subset \Omega_2$, in stage $t=2$, with $Q_{1}^{\text{CA}}(\Cs{S}_{2})$. 
	Throughout the paper, we use notation $Q_{t}^{\text{CA}}(\cdot, \cdot; \Cs{S}_{t+1})$ to differentiate the cost-to-go function at stage $t$ for the conditional assessment problem of the realizations in $\Cs{S}_{t+1}$, presented in this section, from the cost-to-go function $Q_{t}(\cdot, \cdot)$ of the original problem \eqref{eq: MULTI.robust}, presented in Section \ref{sec: MULTI.DP}. 
	
	\begin{remark}
		Recall the interpretation  we stated at the end of Section \ref{sec: MULTI.assessment_realization} for the conditional ambiguity sets. If $\omega_t \notin a(\Cs{S}_{t+1})$, we have $Q_{t}^{\text{CA}}(\hx{t-1}, \hxi{t}; \Cs{S}_{t+1})=Q_{t}(\hx{t-1}, \hxi{t})$. 
		An illustration is provided in Figure \ref{fig: realization_removal}.  \qed
	\end{remark}
	
	\subsubsection{Definition and Properties}
	
	We  now formally  define the notion  of the conditional effectiveness of realizations for \eqref{eq: MULTI.robust} and discuss its properties. % an illustration in Figure \ref{fig: realization_removal}. % for an illustration of such a comparison). .
	
	%To  precisely define the notion of the conditional effectiveness of realizations for \eqref{eq: MULTI.robust}, we need to compare the optimal values of \eqref{eq: MULTI.VF_t_assessment} and \eqref{eq: MULTI.VF}, for which we rely on the respective DP reformulations. 

	\begin{proposition}
		\label{prop: compare_org_asses_real}
		Consider a fixed $t=1, \ldots, T-1$, a subset of realizations $\Cs{S}_{t+1} \subset \Omega_{t+1}$, a partial optimal policy  $\hstarx{t}=[x^*_1, \ldots, x^*_t]$ to \eqref{eq: MULTI.robust} up to stage $t$, and $\hxi{t}$.  %Suppose that we are given  $x_{[t-1]}$ and $\hxi{t}$. 
		Then, %$Q_{t}^{\text{CA}}(\hstarx{t-1}, \hxi{t}; \Cs{S}_{t+1})$, the optimal value of the conditional assessment problem \eqref{eq: MULTI.VF_assessment} of  realizations $\Cs{S}_{t+1}$ at $\hstarx{t-1}$, is smaller than or equal to  $Q_{t}(\hstarx{t-1}, \hxi{t})$, the optimal cost-to-go function \eqref{eq: MULTI.VF} at stage $t$, i.e., $Q_{t}^{\text{CA}}(\hstarx{t-1}, \hxi{t}; \Cs{S}_{t+1}) \le  Q_{t}(\hstarx{t-1}, \hxi{t}) $.
		$Q_{t}^{\text{CA}}(\hstarx{t-1}, \hxi{t}; \Cs{S}_{t+1}) \le  Q_{t}(\hstarx{t-1}, \hxi{t}) $.
		
	\end{proposition}
	
	\begin{proof} 
		For any $\hstarx{t}$ and $\hxi{t}$, the corresponding worst-case expected problem at stage $t$ 
		in \eqref{eq: MULTI.VF_assessment} for $\Cs{S}_{t+1}$ is more restricted than the corresponding problem in \eqref{eq: MULTI.VF}. 
		This, combined with the suboptimality of $x^*_t$ to \eqref{eq: MULTI.VF_assessment}, implies that
		%\begin{align*}
		%& g_{t}(x_{t}, \xi_{t}) + \max_{\bs{p}_{t+1|\hxi{t}} \in \cPcondCA{t}(\Cs{S}_{t+1}) } \ee{\bs{p}_{t+1|\hxi{t}}}{Q_{t+1}(\hx{t}, \hxi{t+1})} \\
		%\le & g_{t}(x_{t}, \xi_{t}) + \max_{\bs{p}_{t+1|\hxi{t}} \in \cPcond{t} } %\ee{\bs{p}_{t+1|\hxi{t}}}{Q_{t+1}(\hx{t}, \hxi{t+1})} 
		%\end{align*}
		%Combining the above conclusion with the suboptimality of $[x_t, \ldots, x_{T}]$ to \eqref{eq: MULTI.VF_t_assessment}, we obtain
		\begin{equation*}
		\begin{array}{ll}
		Q_{t}^{\text{CA}}(\hstarx{t-1},\hxi{t};\Cs{S}_{t+1}) \!\!\!\!\!\!& \le  g_{t}(x^*_{t}, \xi_{t}) + \max_{\condbs{p}{t+1} \in  \cPcondCA{t}(\Cs{S}_{t+1}) }  \ee{\condbs{p}{t+1}}{Q_{t+1}(\hstarx{t}, \hxi{t+1})}\\
		& \le g_{t}(x^*_{t}, \xi_{t}) + \max_{\condbs{p}{t+1} \in  \cPcond{t} }  \ee{\condbs{p}{t+1}}{Q_{t+1}(\hstarx{t}, \hxi{t+1})}\\
		& = Q_{t}(\hstarx{t-1}, \hxi{t}), 
		\end{array}  
		\end{equation*}
		%= & \min_{x_{t} \in \Cs{X}_{t}} \ g_{t}(x_{t}, \xi_{t}) + \max_{\condbs{p}{t} \in  \cPcondCA{t}(\Cs{S}_{t+1}) } \ee{\condbs{p}{t}}{Q_{t+1}(\hx{t}, \hxi{t+1})} \\
		which completes the proof. 
	\end{proof}
	
	%We are now ready to define conditionally effective realizations. 
	
	\begin{definition}[Conditional effectiveness of realizations]
		\label{def: MULTI.cond_eff}
		Consider a fixed $t=1, \ldots, T-1$, a subset of realizations $\Cs{S}_{t+1} \subset \Omega_{t+1}$, partial optimal policy  $\hstarx{t}=[x^*_1, \ldots, x^*_t]$ to \eqref{eq: MULTI.robust}, and $\hxi{t}$.  %Suppose that we are given  $x_{[t-1]}$ and $\hxi{t}$. 
		A subset of realizations $\Cs{S}_{t+1} \subset \Omega_{t+1}$  is called {\it conditionally effective}  if $Q_{t}^{\text{CA}}(x^*_{[t-1]}, \hxi{t}; \Cs{S}_{t+1}) < Q_{t}(x^*_{[t-1]}, \hxi{t})$. A subset of realizations $\Cs{S}_{t+1} \subset \Omega_{t+1}$ is called {\it conditionally ineffective} if it is not conditionally effective.
	\end{definition}
	
	%In other words, a subset of realizations $\Cs{S}_{t+1} \subset \Omega_{t+1}$   is called  conditionally effective  if the optimal value of the corresponding assessment problem \eqref{eq: MULTI.VF_t_assessment} (or, equivalently, \eqref{eq: MULTI.VF_assessment}) is strictly smaller than the optimal value of \eqref{eq: MULTI.VF} (or, equivalently, \eqref{eq: MULTI.VF}) (see Figure \ref{fig: realization_removal} for an illustration of such a comparison). 
	We conclude this section with a comment on the similarity between Definitions \ref{def: effective} and \ref{def: MULTI.cond_eff}.
	Observe that \eqref{eq: MULTI.VF} resembles the structure of \eqref{eq: robust}. Similarly, \eqref{eq: MULTI.VF_assessment} resembles the structure of \eqref{eq: removal}. Consequently, the relationship between   \eqref{eq: robust} and \eqref{eq: removal}, as stated in \eqref{eq: cost_rel}, can similarly be made about \eqref{eq: MULTI.VF} and \eqref{eq: MULTI.VF_assessment}. That is,  
	given $\hstarx{t-1}$ and $\hxi{t}$, by the proof of Proposition \ref{prop: compare_org_asses_real}, we have 
	\begin{equation}
	\label{eq: MULTI.cost_rel_realization}
	\begin{array}{ll}
	Q_{t}^{\text{CA}}(\hstarx{t-1}, \hxi{t}; \Cs{S}_{t+1})\!\!\!\!\!& \le  g_{t}(x^*_{t}, \xi_{t}) + \max_{\condbs{p}{t+1} \in  \cPcondCA{t}(\Cs{S}_{t+1}) } \! \Bs{E}_{\condbs{p}{t+1}}\!\!\left[Q_{t+1}(\hstarx{t}, \hxi{t+1})\right] \\
	& \le Q_{t}(\hstarx{t-1}, \hxi{t}). 
	\end{array}
	\end{equation}
	Observe that the relationships in \eqref{eq: MULTI.cost_rel_realization} are similar to those of \eqref{eq: cost_rel}. 
	This similarity is not by surprise. In fact, the conditional assessment problem \eqref{eq: MULTI.VF_t_assessment} is defined in a similar manner to \eqref{eq: removal}. Hence, Definitions \ref{def: effective} and \ref{def: MULTI.cond_eff} are similar. Consequently,  the results in \citep{rahimian2019} on the effective scenarios for \eqref{eq: robust} will hold for the conditional effectiveness of realizations in \eqref{eq: MULTI.robust}. This is not true for the effectiveness of scenario paths, defined in Section \ref{sec: MULTI.gen_eff_path}. 
	We shall shortly exploit this similarity in Section \ref{sec: MULTI.thms} to identify the effectiveness of scenario paths by connecting it to the conditional effectiveness of realizations along a path for \eqref{eq: MULTI.robust} formed via the total variation distance.

	\section{Multistage \dro\ with the Total Variation Distance}
	\label{sec: MULTI.TV}
	%We presented the general class of multistage \dro\ problems we are interested in studying in Section \ref{sec: MULTI}. 
	%Assuming a finite sample space, we introduced the notions of the effectiveness of scenario paths and the conditional effectiveness of realizations for a general multistage \dro\ in Section \ref{sec: MULTI.effective_scens}. 
	We now narrow down our focus to the multistage \dro\ formed via the total variation distance. 
	This model, referred to as T-\droV\ for short, is formulated in the same manner as \eqref{eq: MULTI.robust}, with the conditional ambiguity sets as follows: 
	%\begin{equation}
	%	\label{eq: MULTI.TV_robust}
	%	\begin{split}
	%		\min_{x_{1} \in \Cs{X}_{1}} \ g_{1}(x_{1}, \xi_{1}) + & \max_{\bs{p}_{2|\hxi{1}} \in \Cs{P}_{2|\hxi{1}}} \ \mathbb{E}_{\bs{p}_{2|\hxi{1}}}   \left[   \min_{x_{2} \in \Cs{X}_{2}}  g_{2}(x_{2}, \xi_{2}) + \max_{\bs{p}_{3|\hxi{2}} \in \Cs{P}_{3|\hxi{2}}} \ \mathbb{E}_{\bs{p}_{3|\hxi{2}}} \Bigg[ \ldots +  \Bigg. \right. \\
	%		& \Bigg. \left. \max_{\bs{p}_{T|\hxi{T-1}} \in  \Cs{P}_{T|\hxi{T-1}} } \ \mathbb{E}_{\bs{p}_{T|\hxi{T-1}}} \left[   \min_{x_{T} \in \Cs{X}_{T}} g_{T}(x_{T}, \xi_{T}) \right]  \ldots  \Bigg]  \right], 
	%	\end{split}
	%	\tag{T-\droV}
	%\end{equation}
	\begin{equation}
	\label{eq: MULTI.TV_ambiguity_set}
	\cPcondnode{t}:=\Big\{\condbsnode{p}{t+1} \; | \; \text{V}(\condbsnode{p}{t+1},\bs{q}_{t+1|\omega_{t}}) \leq \gamma_{t+1}, \; \sum_{\omega_{t+1} \in \Cs{C}(\omega_{t})}  p_{t+1}(\omega_{t+1}) =1, \;   \condbsnode{p}{t+1} \ge \bs{0}\Big\},
	\end{equation}
	where $\text{V}(\condbsnode{p}{t+1},\bs{q}_{t+1|\omega_{t}}):=\frac{1}{2}\sum_{\omega_{t+1} \in \Cs{C}(\omega_{t})} \big|p_{t+1}(\omega_{t+1})- q_{t+1|\omega_{t}}(\omega_{t+1}) \big| $ is the total variation distance between two (conditional) distributions  $\condbsnode{p}{t+1}$ and $\bs{q}_{t+1|\omega_{t}}$. 
	The conditional ambiguity set $\cPcondnode{t}$, $t=1, \ldots, T-1$,  contains all conditional probability distributions $\condbsnode{p}{t+1}$ whose total variation distances to the nominal conditional distribution $\bs{q}_{t+1|\omega_{t}}$ are limited above by {\it the level of robustness} $\gamma_{t+1}$.  Note that $\text{V}(\condbsnode{p}{t+1},\bs{q}_{t+1|\omega_{t}})$ is bounded above by one and bounded below by zero. Hence, without loss of  generality, we assume $0 \le \gamma_{t+1} \le 1$. Also, the levels of robustness at each stage $t$, $t=1, \ldots, T-1$, may depend on $\omega_t$, and $\cPcondnode{t}$, $t=1, \ldots, T-1$, may depend on $\gamma_{t+1}$. We drop these dependencies for ease of exposition. 
	%\subsection{Dynamic Programming Reformulation}
	We can then obtain the cost-to-go function  $Q_{t}(\hx{t-1}, \hxi{t})$ as in \eqref{eq: MULTI.VF},  %in a similar manner that we derived the DP reformulation of \eqref{eq: MULTI.robust}, 
	%\begin{equation}
	%	\label{eq: MULTI.TV.VF}
	%	Q_{t}(\hx{t-1}, \hxi{t})=\min_{x_{t} \in \Cs{X}_{t}} \ g_{t}(x_{t}, \xi_{t}) + \max_{\condbs{p}{t} \in \cPcondgamma{t}} \ee{\condbs{p}{t}}{Q_{t+1}(\hx{t}, \hxi{t+1})}, 
	%\end{equation}
	where the worst-case expected problem  $\max_{\condbs{p}{t+1} \in \cPcond{t}} \ee{\condbs{p}{t+1} }{Q_{t+1}(\hx{t}, \hxi{t+1})}$ is calculated via $\cPcond{t}$ in \eqref{eq: MULTI.TV_ambiguity_set}.

	\subsection{Risk-Averse Interpretation} \label{sec: MULTI.risk} %% 2.
	It is known that under suitable assumptions, a \dro\ model pertains to a risk-averse optimization model. As it is shown in \citep{rahimian2019} such a risk-averse interpretation plays an important role in identifying effective scenarios for a static \dro. Not surprisingly, we shall  shortly see that the risk-averse interpretation of T-\droV\ plays an important role in identifying the effectiveness of scenario paths and the conditional effectiveness of realization as well.%, which we derive it in this section. 

	The conditional ambiguity set $\cPcond{t}$, $t=1, \ldots, T-1$, defined in \eqref{eq: MULTI.TV_ambiguity_set}, is a convex compact subset of $\fM{t}$, the set of all discrete probability distributions induced by $\xi_{t+1}$, given $\hxi{t}$. Moreover, $Q_{t+1}(\hx{t}, \hxi{t+1})$ is integrable on $\fM{t}$ by the assumptions stated in Section \ref{sec: MULTI.DP}. Recall the dual representation of a coherent risk measure; see, e.g., \cite[Theorem~6.7]{Shapiro_Lecture_SP}, \cite[Theorem~3.1]{shapiro2012}, \cite[Theorem~2.2]{ruszczynski2006a}. By a similar procedure applied to conditional distributions, as done for example in \cite{pflug2014}, one can show that  the worst-case expected problem $\max_{\condbs{p}{t} \in \cPcond{t}} \ee{\condbs{p}{t}}{Q_{t+1}(\hx{t}, \hxi{t+1})}$ is equivalent to a real-valued conditional coherent risk mapping \cite{ruszczynski2006b,ruszczynski2006a,shapiro2012}.
	By induction and going backward in time, we can show that T-\droV\ is equivalent to a risk-averse stochastic optimization problem involving nested coherent risk measures.
	%\Cref{prop: MULTI.risk} presents this result. 
	
	\begin{proposition}
		\label{prop: MULTI.risk}
		T-\droV\ is equivalent to
		\begin{equation}
		\label{eq: MULTI.TV_risk}
		\min_{x_{1} \in \Cs{X}_{1}} g_{1}(x_{1}, \xi_{1}) + \rho_{2|\hxi{1}}   \left[    \ldots + \rhocond{T}  \left[  \min_{x_{T} \in \Cs{X}_{T}} g_{T}(x_{T}, \xi_{T}) \right]  \ldots \right] , 
		\end{equation}
		where $	\rro{t+1|\hxi{t}}{\cdot}$, $t=1, \ldots, T-1$,  is the conditional coherent risk measure given by 
		\begin{equation*}
		\rro{t+1|\hxi{t}}{\cdot}:=\gamma_{t+1} \  \mathrm{sup}_{|\hxi{t}}\ [\cdot] + \left( 1-\gamma_{t+1}\right) \ccvar{\gamma_t}{\cdot|\hxi{t}}.
		\end{equation*} 
		%with $\cdot$ is replaced by $Q_{t+1}(x_{[t]}, \xi_{[t+1]})$, $t=1, \ldots, T-1$. 
	\end{proposition}

	\begin{proof}
		By the argument above and defining $\rro{t+1|\hxi{t}}{\cdot}:=\max_{\condbs{p}{t} \in \cPcond{t}} \ee{\condbs{p}{t} }{\cdot}$ to denote  the conditional coherent risk measure, we obtain a risk-averse stochastic optimization with nested coherent risk measures. Now, the risk-averse interpretation of a \eqref{eq: robust}, see \citep[Proposition~3]{rahimian2019} or  \citep[Theorems~1--2]{jiang2018}, 
		completes  the proof. 
	\end{proof}

	In Proposition \ref{prop: MULTI.risk}, $\mathrm{sup}_{|\hxi{t}}\ [\cdot]$ is taken conditionally with respect to the support set $\Xi_{t+1}$ and $\ccvar{\gamma_{t+1}}{\cdot | \hxi{t}}$ is the conditional value-at-risk (CVaR)\footnote{Here, $\ccvar{\alpha}{Z}$ denotes the CVaR of $Z$ at the confidence level $\alpha$, $0<\alpha<1$, defined as  $\ccvar{\alpha}{Z}:=1/(1-\alpha)\int_{\alpha}^1 \vvar{\beta}{Z}\,d\beta$, where $
		\vvar{\beta}{Z}:=\inf\sset*{u}{\textrm{Pr}\{Z \leq u\}\geq \beta}
		$
		is the value-at-risk (VaR) at level $\beta$ \citep{rockafellar2000optimization}.} taken with respect to the nominal conditional distribution $\condbs{q}{t}$ at level $\gamma_{t+1}$. Per usual convention, we set
	$\ccvar{0}{\cdot | \hxi{t}}:=\ee{\condbs{q}{t}}{\cdot}$  and  $\ccvar{1}{\cdot | \hxi{t}}:= \mathrm{sup}_{|\hxi{t}}[\cdot]$, $t=1, \ldots, T-1$.
	Hence, %that \droV\ contains a wide range of risk-averse models.
	when $\gamma_{t+1}=0$ for $t=1, \ldots, T-1$,    T-\droV\ reduces to the risk-neutral model.
	As $\gamma_{t+1}$ increases, more weight is put on the worst-case cost, and T-\droV\ becomes more conservative. 
	When $\gamma_{t+1}=1$, T-\droV\ reduces to minimizing the  worst-case total cost $g_{1}(x_{1}, \xi_{1})+ g_{2}(x_{2}, \xi_{2})+ \ldots +g_{T}(x_{T}, \xi_{T})$ over the support $\Xi$ of $\xi$.

	\section{Identifying Effectiveness of Scenario Paths and Realizations for a Multistage \droV}
	
	%\subsection{Main Results}
	\label{sec: MULTI.thms}
	
	For a general multistage \dro\ with a finite sample space, we introduced the notions of the effectiveness of scenario paths and the conditional effectiveness of realizations in Sections \ref{sec: MULTI.gen_eff_path} and \ref{sec: MULTI.gen_eff_realization}, respectively. In this section, we focus on T-\droV\ and present easy-to-check  conditions to identify the effectiveness of scenario paths/realizations.  
	
	Let us first focus on the conditional effectiveness of realizations along a scenario path.
	As discussed at the end  of Section \ref{sec: MULTI.gen_eff_realization}, by construction, effective scenarios for \eqref{eq: robust}, as stated in Definition \ref{def: effective},  are defined similarly to the  conditionally effective realizations for \eqref{eq: MULTI.robust}, as stated in Definition \ref{def: MULTI.cond_eff}. In particular,  one can make  the following observation for T-\droV. %when the total variation distance is used to model the conditional distributional ambiguity in \eqref{eq: MULTI.robust}, i.e., T-\droV. 

	\begin{observation}
		\label{thm: MULTI.cond_eff}
		Let   $x^*=[x^*_1, \ldots, x^*_T]$ be an optimal policy to T-\droV. Given $x^*_{[t-1]}$ and $\hxi{t}$, $t=1, \ldots, T-1$, the easy-to-check conditions  stated in \citep[Theorems~1--3]{rahimian2019} for \eqref{eq: robust} via the total variation distance hold for the conditional effectiveness of realizations $\Cs{S}_{t+1} \subset \Omega_{t+1}$ for T-\droV.
	\end{observation}
	In other words,  given the history of decisions and stochastic process, one can easily identity conditionally effective and ineffective realizations for T-\droV\ by using the easy-to-check conditions proposed in \citep[Theorems~1--3]{rahimian2019}  for \eqref{eq: robust} via the total variation distance. 
	%To keep the paper self-contained, we review these results in Online Supplement \ref{sec: MULTI.ETC}, 
	We briefly mention the general ideas of these results next. 
	
	\subsection{Background  on Identifying Effective Scenarios for a Two-Stage \droV}
	\label{sec: TV.background}
	
	When the distributional ambiguity is modeled via the total variation distance, \eqref{eq: robust} can be written  as 
	\begin{equation}
	\label{eq: TV.robust}
	\min_{x  \in \Cs{X}  } \ \left\{f(x ):=g (x ) + \max_{\bs{p}   \in \Cs{P}   }  \ 	\sum_{\omega \in \Omega  } p  (\omega  ) Q(x,\omega) %\left[\min_{x_{2} \in \Cs{X}_{2}}  g_{2}(x_{2}, \omega  ) \right] 
	\right\}
	, \tag{2-\droV}
	\end{equation}
	where
	%\begin{equation*}
	%\label{eq: ambiguity_set}
	$
	\Cs{P}  :=\Big\{ \bs{p}\!:\! \text{V}(\bs{p}, \bs{q}) \leq \gamma, \
	\sum_{\omega   \in \Omega  } p  (\omega  ) = 1, \
	\bs{p}   \ge \mathbf{0} \Big\}
	$
	%\end{equation*}
	and $\text{V}(\bs{p}  ,\bs{q}  ):=\frac{1}{2}\sum_{\omega   \in \Omega  }|p(\omega  )-q  (\omega  )|$ denotes the total variation distance between two probability distributions $\bs{p}  :=[p  (\omega^1  ), \ldots, p  (\omega^n  )]^{\top}$ and $\bs{q}  :=[q  (\omega^1  ), \ldots, q  (\omega^n  )]^{\top}$. Assume that $0 < \gamma \le 1$. %, and for notational simplicity, we drop the dependence of the objective function $f(x)$ and the ambiguity set $\Cs{P}  $ on $\gamma$.
	Before we discuss the general structure of results in \citep{rahimian2019}, let us define some notation. For a fixed $x \in \Cs{X} $, let $h(x , \omega  ):=g (x )+Q(x,\omega)$. %\min_{x_{2} \in \Cs{X}_{2}}  g_{2}(x_{2}, \omega  ).. 
	For any $B \subset \Bs{R}$, we use $[\bs{h}(x ) \in B]$ % and $\Probzero{h(x, \cdot) \in B}$ 
	as shorthand notation for the set $\{\omega\!\in\!\Omega\!\!:\!\!h(x, \omega)\!\!\in\!\!B\}$.
	With some abuse of notation, let $\bs{h}(x):=[h(x, \omega^1), \ldots, h(x, \omega^n)]$. 
	For a fixed $\eta \in \Bs{R}$, we use $\Psi(x , \eta)$ %and $\Psi(x, \eta^{-})$ 
	to denote $\sum_{\omega \in [\bs{h}(x ) \le \eta]} q(\omega)$. 
	Given $\beta \in [0,1]$, let $\vvar{\beta}{\bs{h}(x )}$ be the left-side $\beta$-quantile of distribution of $\bs{h}(x )$ or  equivalently, the VaR of $\bs{h}(x )$ at level $\beta$:  $\vvar{\beta}{\bs{h}(x )}:=\inf\{\eta\,:\, \Psi(x , \eta) \geq \beta\}$ \citep{rockafellar2002}.
	%By  convention,  $\vvar{\beta}{\bs{h}(x )}=-\infty$ for $\beta=0$ and $\vvar{\beta}{\bs{h}(x )}=\sup_{\omega   \in \Omega  }h(x ,\omega  )$ for $\beta=1$.
	%Let us also assume that $\gamma>0$, unless stated explicitly otherwise.	
	
	Motivated by the risk-averse interpretation of \eqref{eq: TV.robust} (see Proposition \ref{prop: MULTI.risk} with $T=2$), the authors in \cite{rahimian2019} define the following sets, collectively referred to as {\it primal categories},  that partition the scenario set $\Omega  $:
	%\begin{itemize}[noitemsep, nolistsep, leftmargin=1.70cm]
	%\item
	$\cat{1}{x }:=\left[ \bs{h}(x ) < \vvar{\gamma}{\bs{h}(x )}\right] $, %i.e., the set of scenarios strictly below $\vvar{\gamma}{\bs{h}(x )}$,
	%\item
	$\cat{2}{x }:=\left[ \bs{h}(x ) = \vvar{\gamma}{\bs{h}(x )}\right]$, %i.e., the set of scenarios at $\vvar{\gamma}{\bs{h}(x )}$,
	%\item
	$\cat{3}{x }:=\left[ \vvar{\gamma}{\bs{h}(x )}< \bs{h}(x ) <\sup_{\omega \in \Omega} h(x ,\omega) \right] $, %i.e., the set of scenarios strictly between $\vvar{\gamma}{\bs{h}(x )}$ and the worst-case cost at $x $, and
	%\item
	and 
	$\cat{4}{x }:=\left[ \bs{h}(x )=\sup_{\omega   \in \Omega  } h(x , \omega  )\right] $.  %i.e., the set of scenarios at the worst-case cost at $x $.%
	%\end{itemize} 
	Let $(x^*,\bs{p}^*)$ solve \eqref{eq: TV.robust}. For a moment, suppose that the nominal probabilities are all positive, i.e., $q(\omega)>0$ for all $\omega \in \Omega$. Then, according to \citep{rahimian2019}, all scenarios in $\cat{3}{x^{*}} \cup \cat{4}{x^*}$ are effective, while scenarios in $\cat{1}{x^*}$ are ineffective. For scenarios in $\cat{2}{x^*}$, the situation is more delicate. For example, if $\sum_{\omega \in \cat{2}{x^*} } q(\omega)=\gamma$, all scenarios in $\cat{2}{x^*}$  are ineffective. Otherwise, if $\sum_{\omega \in \cat{2}{x^*} } q(\omega) > \gamma$ and there is only one scenario in $\cat{2}{x^*}$, that scenario must be effective. %When there are multiple scenarios in $\cat{2}{x^*}$ and $\sum_{\omega \in \cat{2}{x^*} } q(\omega)>\gamma$, the situation is more complicated; see \citep[Theorem~3]{rahimian2019} for details on this case. Furthermore, when nominal probabilities take zero values---i.e., if $q(\omega)=0$ for some scenario $\omega \in \Omega$---more details are involved. 
	Note that easy-to-check conditions may not  identify the effectiveness of all scenarios in $\cat{2}{x^*}$.
	We refer the readers to \citep[Theorems~1--3]{rahimian2019} %or Lemmas \ref{lem: ineff}--\ref{lem: eff_index} in Online Supplement \ref{sec: MULTI.ETC}  
	for detailed results. 
	
	%Using \eqref{eq: mu_lambda}, when $\lambda(x )>0$, we have  $\vvar{\gamma}{\bs{h}(x )}<\sup_{\omega   \in \Omega  } h(x , \omega  )$, and when $\lambda(x )=0$, we have  $\vvar{\gamma}{\bs{h}(x )}=\sup_{\omega   \in \Omega  } h(x , \omega  )$, respectively. 
	%Consequently, when $\lambda(x )=0$, we have $\cat{2}{x }=\cat{4}{x }$, and $\cat{3}{x }=\emptyset$. Thus, $\cat{4}{x }$ and its complement, i.e., $\stcomp{\Omega}_{4}(x ) (=\cat{1}{x })$, partition the scenario set $\Omega  $ when $\lambda(x )=0$. 
	
	%In our notation, we suppress the dependence of these sets on $\gamma$ for ease of exposition.

	\subsection{Main Result for a Multistage \droV} 
	\label{sec: MULTI.TV_main}
	%Observe that the following gives a  lower bound  on \eqref{eq: MULTI.VF_t_assessment} 
	%\begin{align*}
	%	\min_{x_{t}, \ldots, x_{T}} \; \max_{\bs{p}_{t} \in \Cs{P}^{\text{A}}_{t}(\Cs{S}_{t+1}) } & \ee{\bs{p}_{t}}{g_{t}(x_{t}, \xi_{t})+ \ldots +g_{T}(x_{T}, \xi_{T})}\\
	%	\st \quad & x_{\tau} \in \Cs{X}_{\tau}, \; \tau=t, \ldots T,
	%\end{align*}
	%where $\Cs{P}^{\text{A}}_{t}(\Cs{S}_{t+1}) $...

	Recall that according to Theorem \ref{thm: MULTI.cond_eff} one can identify the conditional effectiveness of realizations using the easy-to-check conditions stated in \citep[Theorems~1--3]{rahimian2019}. 
	We refer to the realizations whose conditional effectiveness is identified by these results as {\it identifiably conditionally effective or ineffective}.  We now present our main result, connecting the effectiveness of scenario paths to the conditional effectiveness of realizations along the path.

	\begin{theorem}
		\label{thm: MULTI.path_eff_ineff}
		A scenario path $\omega^{\prime}_{T} \in \Omega_{T}$ is effective for T-\droV\ if and only if $\Pi_{t}(\omega^{\prime}_{T})$ is identifiably  conditionally effective for T-\droV\ for all $t$, $t=2, \ldots, T$. 
	\end{theorem}
	
	%\begin{theorem}
	%	\label{thm: MULTI.path_eff}
	%	Consider a scenario path $\omega^{\prime}_{T} \in \Omega_{T}$. If $\Pi_{t}(\omega^{\prime}_{T})$ is indentifiably conditionally effective for T-\droV\ for all $t$, $t=2, \ldots, T$, then, the scenario path $\omega^{\prime}_{T}$ is effective for  T-\droV.
	%\end{theorem}
	
	%\begin{theorem}
	%	\label{thm: MULTI.path_ineff}
	%	Consider a scenario path $\omega^{\prime}_{T} \in \Omega_{T}$. If (i) $\Pi_{T}(\omega^{\prime}_{T})$ is not (trivially) conditionally effective for T-\droV\ and (ii) there exists $t$, $t=2, \dots, T$, such that $\Pi_{t}(\omega^{\prime}_{T})$ is indentifiably conditionally ineffective for T-\droV, then, the scenario path $\omega^{\prime}_{T}$ is ineffective for the  T-\droV.
	%\end{theorem}

	As a result of  Theorem \ref{thm: MULTI.path_eff_ineff}, in order to identify the effectiveness of a scenario path for T-\droV, we need to first identify the conditional effectiveness of all realizations along the path using the easy-to-check conditions stated in \citep[Theorems~1--3]{rahimian2019}. 
	If all realizations along the path are identifiably conditionally effective, then the scenario path is effective for  T-\droV. Otherwise, if there is at least one realization along the path that is identifiably conditionally ineffective, then the scenario path is ineffective  for  T-\droV. 
	As in \cite{rahimian2019}, we might not be able to identify the effectiveness of all scenario paths for T-\droV\ because the easy-to-check conditions might not be able to do so for the conditional effectiveness of realizations. 
	Figure \ref{fig: results} illustrates the usage of Theorem \ref{thm: MULTI.path_eff_ineff}  to identify the effectiveness of scenario paths for T-\droV. % under a binary scenario tree in four stages.
	%In Section \ref{sec: MULTI.numerics}, we test the performance of  Theorem \ref{thm: MULTI.path_eff_ineff} on a number of problems. %from the literature. 

	\begin{figure}[!htbp]
		\centering
		\includegraphics[width=0.43\linewidth]{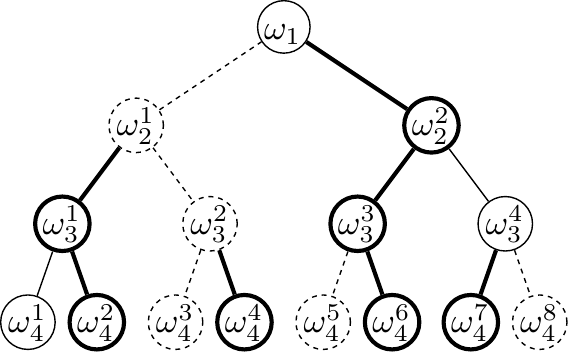}
		\caption{\label{fig: results} An illustration of the conditional effectiveness of realizations and the effectiveness of scenario paths. The nodes corresponding to conditionally effective realizations and their paths to their ancestors are shown with thick black solid  lines. Conditionally ineffective realizations  are shown with thin dotted black lines, and conditionally unidentified realizations are shown with thin solid black  lines. All the scenario paths at the left subtree (i.e., $\omega_4^1, \omega^2_4, \omega_4^3, \omega^4_4$) are ineffective because they are all going through at least one identifiably conditionally ineffective realization, for example, $\omega_2^1$. Scenario paths $\omega_4^5$ and $\omega_4^8$ are ineffective as well, as they both are going through an identifiably conditionally ineffective realization in the last stage. The only effective scenario path is $\omega_4^6$ because all realizations along its path are identifiably conditionally effective. Scenario path $\omega_4^7$ is unidentified because the realization $\omega_3^4$ is conditionally unidentified. Note that while the effectiveness of $\omega_3^4$ cannot be identified, the effectiveness of the scenario path $\omega_4^8$ can be identified because it goes  through an identifiably conditionally ineffective realization.}
	\end{figure}

	%\subsection{Proofs } 
	%\label{sec: MULTI.proofs}
	
	We now turn our attention to the proof of Theorem \ref{thm: MULTI.path_eff_ineff}. Before presenting the proof, we state  two technical lemmas. 
	Lemma \ref{lem: MULTI.cond} makes a connection between the maximizers of a worst-case expected problem and an identifiably conditionally effective/ineffective realization. 
	Lemma \ref{lem: MULTI.OPT_Cond} states  optimality conditions for the assessment problem of scenario paths in  $\Cs{S} \subset \Omega_{T}$ for T-\droV. %, i.e., \eqref{eq: MULTI.robust_assessment} where the conditional ambiguity sets are formed as \eqref{eq: MULTI.TV_ambiguity_set}. 

	\begin{lemma}
		\label{lem: MULTI.cond}
		Let   $x^*=[x^*_1, \ldots, x^*_T]$ be an optimal policy  % and $\bs{p}^*=[\bs{p}^*_1, \ldots, \bs{p}^*_T]$ be an optimal worst-case probability distribution 
		to  T-\droV. Consider a realization $\omega_t \in \Omega_t$ at stage $t$, where $t=1, \ldots, T-1$. Given $x^*_{[t]}$ and $\hxi{t}^{\omega_t}$, let $\Cs{P}^{*}_{t+1|\omega_t}$ denote the set of optimal (conditional) probability distributions to the worst-case expected problem at  stage $t$, i.e., 
		\begin{equation}
		\label{eq: Pstar}
		\Cs{P}^{*}_{t+1|\omega_t}:= \argmax_{\condbsnode{p}{t} \in  \Cs{P}_{t+1|\omega_t} } \  \ee{\condbsnode{p}{t}}{Q_{t+1}(\hstarx{t}, \hxi{t+1}^{\omega_{t+1}})},     
		\end{equation}
		where $\Cs{P}_{t+1|\omega_t}$ is defined as in \eqref{eq: MULTI.TV_ambiguity_set}. The following statements hold for $\omega_{t+1} \in \Cs{C}(\omega_t) \subset \Omega_{t+1}$. 
		\begin{enumerate}[label=(\roman*)]
			\item\label{lem: MULTI.cond_eff} If $\omega_{t+1} $ is identifiably conditionally effective for T-\droV, then,  for any  $\bs{p}_{t+1} \in \Cs{P}^{*}_{t+1|\omega_{t}}$, we have $p_{t+1}(\omega_{t+1})>0$. Moreover, 
			$g_{t}(x^*_{t}, \xi_{t}) +  \max\limits_{\condbs{p}{t} \in  \Cs{P}^{\text{CA}}_{t+1|\omega_t}(\Cs{S}_{t+1})}  \ee{\condbs{p}{t}}{Q_{t+1}(\hstarx{t}, \hxi{t+1})} 
			< Q_{t}(\hstarx{t-1}, \hxi{t}^{\omega_{t}})$.
			
			\item\label{lem: MULTI.cond_ineff}  If $\omega_{t+1}$ is identifiably conditionally ineffective for the T-\droV, then, for any  $\bs{p}_{t+1} \in \Cs{P}^{*}_{t+1|\omega_{t}}$, we have $p_{t+1}(\omega_{t+1})=0$. Moreover, $ Q_{t}^{\text{CA}}(\hstarx{t-1}, \hxi{t}^{\omega_{t}} ; \{\omega_{t+1}\})  = Q_{t}(\hstarx{t-1}, \hxi{t}^{\omega_{t}}) $. 
			
		\end{enumerate}	
	\end{lemma}
	
	\begin{proof}
		The proof of the first parts in (i) and (ii) follows from \citep[Theorems~1--3]{rahimian2019}, combined with \citep[Proposition~4]{rahimian2019}. 
		To prove the second part of (i), first note that we have the equality $Q_{t}(\hstarx{t-1}, \hxi{t}^{\omega_{t}})= g_{t}(x^*_{t}, \xi_{t}) + \max\limits_{\condbs{p}{t} \in  \Cs{P}_{t+1|\omega_t}} \ee{\condbs{p}{t}}{Q_{t+1}(\hstarx{t}, \hxi{t+1})} $ as  argued in the proof of Proposition \ref{prop: compare_org_asses_real} and by the time consistency of $x^*$. The inequality then follows from the fact that $p_{t+1}(\omega_{t+1})>0$ and the finiteness of the sample space $\hXi{t+1}$. 
		The second part of (ii) is immediate from Definition \ref{def: MULTI.cond_eff}. 
		%\alertHR{add interior-type regularity assumptions}
	\end{proof}
	
	\begin{lemma}
		\label{lem: MULTI.OPT_Cond}
		%A feasible policy $\bar{x}=[\bar{x}_{1}, \ldots, \bar{x}_{T}]$ is optimal for T-\droV\ if and only if %there exists a mapping  $\bar{\pi}=[\bar{\pi}_{1}, \ldots, \bar{\pi}_{T}]$  of dual multipliers such that  
		%\begin{equation}
		%    \label{eq: MULTI.OPT_Cond}
		%    0 \in \partial g_{t}(\bar{x}_{t}, \xi_{t})+ \Cs{N}_{\Cs{X}_{t}}(\bar{x}_{t})+ \partial  \left\lbrace \max_{\bs{p}_{t+1|\hxi{t}} \in  \Cs{P}_{t+1|\hxi{t}} } \ee{\bs{p}_{t+1|\hxi{t}}}{Q_{t+1}(\bar{x}_{[t]}, \hxi{t+1})} \right \rbrace,
		%\end{equation}
		%holds true for every $\hxi{t}$ and $t=1, \ldots, T$, 
		%where $\Cs{N}_{\Cs{X}_{t}}(\bar{x}_{t})$ denotes the normal cone of  $\Cs{X}_{t}$ at $\bar{x}_{t}$ and is defined as 
		%$$\Cs{N}_{\Cs{X}_{t}}(\bar{x}_{t})=\Set*{s \in \Bs{R}^{d_{t}}}{s^{\top}(x_{t}-\bar{x}_{t}) \le 0 \; \ \text{for all} \ x_{t} \in \Cs{X}_{t}}.$$
		
		Consider a feasible policy $\bar{x}=[\bar{x}_{1}, \ldots, \bar{x}_{T}]$ to T-\droV. 
		Let $\Cs{N}_{\Cs{X}_{t}}(\bar{x}_{t})$ denote the normal cone of  $\Cs{X}_{t}$ at $\bar{x}_{t}$, $t=1, \ldots, T$, defined as 
		$\Cs{N}_{\Cs{X}_{t}}(\bar{x}_{t})=\sset*{s \in \Bs{R}^{n_{t}}}{s^{\top}(x_{t}-\bar{x}_{t}) \le 0 \; \ \text{for all} \ x_{t} \in \Cs{X}_{t}}$. 
		Moreover, let 
		%$$\bar{\Cs{P}}^{\text{A}}_{T|\hxi{T-1}}(\Cs{S}):=\Set*{\bs{p}_{T|\hxi{T-1}} \in  \Cs{P}^{\text{A}}_{T|\hxi{T-1}}(\Cs{S})}{ \bs{p}_{T|\hxi{T-1}} \in \argmax_{\bs{p}_{T|\hxi{T-1}} \in  \Cs{P}^{\text{A}}_{T|\hxi{T-1}}(\Cs{S})} \  \ee{\bs{p}_{T|\hxi{T-1}}}{Q_{T}^{\text{A}}(\bar{x}_{[T-1]}, \hxi{T}; \Cs{S})}}$$
		\begin{equation}
		\begin{array}{ll}
		\bar{\Cs{P}}^{\text{A}}_{T|\hxi{T-1}}(\Cs{S}):=& \argmax_{\bs{p}_{T} \in  \Cs{P}^{\text{A}}_{T|\hxi{T-1}}(\Cs{S})} \  \ee{\bs{p}_{T}}{Q_{T}^{\text{A}}\Big(\bar{x}_{[T-1]}, \hxi{T}; \Cs{S}\Big)} \ \ \text{ and } 
		\\
		\bar{\Cs{P}}^{\text{A}}_{t+1|\hxi{t}}(\Cs{S})\ :=& \argmax_{\bs{p}_{t+1} \in  \Cs{P}_{t+1|\hxi{t}}} \  \ee{\bs{p}_{t+1}}{Q_{t+1}^{\text{A}}\Big(\bar{x}_{[t]}, \hxi{t+1}; \Cs{S}\Big)},  \label{eq: Pbar}
		\end{array}
		\end{equation}
		%\vspace*{-0.8em}
		%$$
		%\bar{\Cs{P}}_{t+1|\hxi{t}}:=\Set*{\bs{p}_{t+1|\hxi{t}} \in  \Cs{P}_{t+1|\hxi{t}}}{ \bs{p}_{t+1|\hxi{t}} \in \argmax_{\bs{p}_{t+1|\hxi{t}} \in  \Cs{P}_{t+1|\hxi{t}}} \  \ee{\bs{p}_{t+1|\hxi{t}}}{Q_{t+1}^{\text{A}}(\bar{x}_{[t]}, \hxi{t+1}; \Cs{S})}}
		%,$$
		$t=1, \ldots, T-2$. 
		The policy $\bar{x}$ is optimal to  the assessment problem of scenario paths in $\Cs{S} \subset \Omega_{T}$ corresponding to T-\droV\ if and only if 
		\begin{enumerate}[label=(\roman*)]
			\item $0 \in \partial g_{T}(\bar{x}_{T}, \xi_{T})+ \Cs{N}_{\Cs{X}_{T}}(\bar{x}_{T})$ holds true for every $\hxi{T}$, 
			\item $0 \in \partial g_{T-1}(\bar{x}_{T-1}, \xi_{T-1})+ \Cs{N}_{\Cs{X}_{T-1}}(\bar{x}_{T-1})+ \conv{\bigcup_{\bs{p}_{T} \in  \bar{\Cs{P}}^{\text{A}}_{T|\hxi{T-1}}(\Cs{S})}   \ee{\bs{p}_{T}}{\partial Q_{T}^{\text{A}}\Big(\bar{x}_{[T-1]}, \hxi{T}; \Cs{S}\Big)} }$ holds true for every $\hxi{T-1}$,  and 
			\item 
			$   
			0 \in \partial g_{t}(\bar{x}_{t}, \xi_{t})+ \Cs{N}_{\Cs{X}_{t}}(\bar{x}_{t})+ \conv{\bigcup_{\bs{p}_{t+1} \in  \bar{\Cs{P}}^{\text{A}}_{t+1|\hxi{t}}(\Cs{S}) } \ee{\bs{p}_{t+1}}{\partial Q_{t+1}^{\text{A}}\Big(\bar{x}_{[t]}, \hxi{t+1}; \Cs{S}\Big)}  }
			$
			%\end{equation*}
			holds true for every $\hxi{t}$ and $t=1, \ldots, T-2$, 
		\end{enumerate}
		%there exists a mapping  $\bar{\pi}=[\bar{\pi}_{1}, \ldots, \bar{\pi}_{T}]$  of dual multipliers such that  
		%\begin{equation*}
		%\label{eq: MULTI.OPT_Cond_assessment_T}
		%\begin{equation*}
		%    \begin{split}
		%\label{eq: MULTI.OPT_Cond_assessment_T-1}
		%\begin{equation*}
		%\label{eq: MULTI.OPT_Cond_Assessment_t}
		where 
		$\conv{\cdot}$ denotes the convex hull. 
		
	\end{lemma}
	
	\begin{proof}
		Let us consider the DP reformulation \eqref{eq: MULTI.VF_path_DP_T}--\eqref{eq: MULTI.VF_path_DP_t} of the assessment problem of scenario paths in $\Cs{S} \subset \Omega_{T}$ corresponding to T-\droV, similar to those in Section \ref{sec: MULTI.DP_path_assessment}, where the conditional ambiguity sets are formed as in \eqref{eq: MULTI.TV_ambiguity_set}.
		Similar to the proof of  \cite[Proposition~3.4]{Shapiro_Lecture_SP}, using \cite[Theorem~51]{shapiro2003duality} and Moreau-Rockafellar  theorem \cite[Theorem~7.4]{Shapiro_Lecture_SP}, we can write the optimality conditions at stage $t=T$ as $0 \in \partial g_{T}(\bar{x}_{T}, \xi_{T})+ \Cs{N}_{\Cs{X}_{T}}(\bar{x}_{T})$. %Moreover, 
		At $t=T-1$,  the optimality conditions can be stated as 
		%\begin{equation*}
		$  0 \in \partial g_{T-1}(\bar{x}_{T-1}, \xi_{T-1})+ \Cs{N}_{\Cs{X}_{T-1}}(\bar{x}_{T-1})+ \partial \! \left\lbrace \max_{\bs{p}_{T} \in  \Cs{P}^{\text{A}}_{T|\hxi{T-1}}(\Cs{S}) } \! \Bs{E}_{\bs{p}_{T}}\!\!\left[Q_{T}^{\text{A}}\Big(\bar{x}_{[T-1]}, \hxi{T}; \Cs{S}\Big)\right] \right\rbrace$
		%\end{equation*}
		holds true for every $\hxi{T-1}$. 
		Now, using \cite[Corollary~4.3.2]{hiriart2001} yields the optimality conditions at $t=T-1$. The  proof for $t=T-2, \ldots, 1$ follows similarly. 
	\end{proof}
	
	%\begin{remark}
	%\label{rem: MULTI.OPT_Cond}
	Note that when $\Cs{S}=\emptyset$, Lemma \ref{lem: MULTI.OPT_Cond} gives the optimality conditions for  T-\droV. 
	%\end{remark}

	\begin{proof}[Proof of Theorem \ref{thm: MULTI.path_eff_ineff}]
		
		Let $\Pi_{t}(\omega^{\prime}_{T})=\omega^{\prime}_{t}$, for $t=1, \ldots, T$, i.e., the scenario path $\omega^{\prime}_{T}$ is given with the sequence $(\omega^{\prime}_{1}, \ldots, \omega^{\prime}_{T})$, where $\omega_{1}=\omega^{\prime}_{1}$. 
		Moreover, let   $x^*=[x^*_1, \ldots, x^*_T]$ be an optimal policy to T-\droV\, and $\bs{p}^{*}=\bs{p}^*(x^*)$
		%=[\bs{p}^*_{2|\hxi{1}},\bs{p}^*_{3|\hxi{2}}, \ldots, \bs{p}^*_{T|\hxi{T-1}}]$ 
		denote an optimal worst-case probability distribution corresponding to $x^*$. 
		We first prove that if $\omega^{\prime}_{t+1}$ is identifiably conditionally effective for all $t=1, \ldots, T-1$, then $\omega^{\prime}_{T}$ is effective. 
		Then, we prove that if there exists $t$, $t=2, \dots, T$, such that $\omega^{\prime}_{t+1}$ is identifiably conditionally ineffective, then  $\omega^{\prime}_{T}$ is ineffective. In the proof, we frequently use  the time consistency of $x^*$ for T-\droV, as well as the  DP reformulations discussed in Sections \ref{sec: MULTI.DP}, \ref{sec: MULTI.DP_path_assessment}, and \ref{sec: MULTI.DP_realization_assessment}, all tailored for the conditional ambiguity sets of the form  \eqref{eq: MULTI.TV_ambiguity_set}.
		
		``$\Longrightarrow$:" To prove the result, we need to show that $Q_{1}^{\text{A}}(\{\omega^{\prime}_{T}\}) < Q_1$. % according to Definition \ref{def: MULTI.path_eff}. 
		By the hypothesis and using Definition \ref{def: MULTI.cond_eff}, we have  $Q_{t}^{\text{CA}}(\hstarx{t-1}, \hxi{t}^{\omega^{\prime}_{t}}; \{\omega^{\prime}_{t+1}\}) < Q_{t}(\hstarx{t-1}, \hxi{t}^{\omega^{\prime}_{t}})$, $t=1, \ldots, T-1$. Using Lemma \ref{lem: MULTI.cond}\ref{lem: MULTI.cond_eff}, we have 
		\begin{equation}
		\label{eq: MULTI.cond_eff_1}
		\begin{array}{ll}
		& g_{t}(x^{*}_{t}, \xi_{t}^{\omega^{\prime}_{t}}) + \max_{\bs{p}_{t+1} \in  \Cs{P}^{\text{CA}}_{t+1|\omega^{\prime}_{t}} (\{\omega^{\prime}_{t+1}\}) } \ee{\bs{p}_{t+1|\omega^{\prime}_{t}}}{Q_{t+1}(\hstarx{t}, \hxi{t+1}^{\omega_{t+1}})} \\
		& < g_{t}(x^{*}_{t}, \xi_{t}^{\omega^{\prime}_{t}}) + \max_{\bs{p}_{t+1} \in  \Cs{P}_{t+1|\omega^{\prime}_{t}} } \ee{\bs{p}_{t+1|\omega^{\prime}_{t}}}{Q_{t+1}(\hstarx{t}, \hxi{t+1}^{\omega_{t+1}})},
		\end{array}
		\end{equation}
		which for all $\bs{p}_{t+1} \in  \Cs{P}^{\text{CA}}_{t+1|\omega^{\prime}_{t}} (\{\omega^{\prime}_{t+1}\})$ implies 
		\begin{equation}
		\label{eq: MULTI.cond_eff_2}
		\ee{\bs{p}_{t+1}}{Q_{t+1}(\hstarx{t}, \hxi{t+1}^{\omega_{t+1}})} < \max_{\bs{p}_{t+1} \in  \Cs{P}_{t+1|\omega^{\prime}_{t}} } \ee{\bs{p}_{t+1}}{Q_{t+1}(\hstarx{t}, \hxi{t+1}^{\omega_{t+1}})}. 
		\end{equation}
		Let us look into the values of $Q_{t}^{\text{A}}(\hstarx{t-1}, \hxi{t}^{\omega}; \{ \omega^{\prime}_{T} \})$ and $Q_{t}(\hstarx{t-1}, \hxi{t}^{\omega_{t}})$ on the whole scenario tree.  Recall by Remark \ref{rem: MULTI.DP_path} that the only realization whose  cost-to-go function at stage $t$ is affected by the removal of the scenario path $\omega^{\prime}_{T}$ is $\omega^{\prime}_{t}$, $t=2, \ldots, T-1$.  That is, for a feasible policy $x$ to T-\droV,  \eqref{eq: MULTI.Proj_Observation} simplifies to 
		\begin{equation}
		\label{eq: MULTI.Proj_Observation_Simple}
		Q_{t}^{\text{A}}(\hx{t-1}, \hxi{t}^{\omega_{t}};\{\omega^{\prime}_{T}\})=Q_{t}(\hx{t-1}, \hxi{t}^{\omega_{t}}) \quad \text{if} \quad \omega_{t} \neq \omega^{\prime}_{t}, \quad t=2, \ldots, T-1. 
		\end{equation}
		Now, we look into $Q_{t}^{\text{A}}(\hstarx{t-1}, \hxi{t}^{\omega^{\prime}}; \{ \omega^{\prime}_{T} \})$ and $Q_{t}(\hstarx{t-1}, \hxi{t}^{\omega^{\prime}_{t}})$ along the scenario path $\omega^{\prime}_{T}$.
		Starting at stage $t=T-1$, we have
		\begin{equation}
		\label{eq: MULTI.cond_eff_T}
		\begin{array}{ll}
		& Q_{T-1}^{\text{A}}(\hstarx{T-2}, \hxi{T-1}^{\omega^{\prime}_{T-1}}; \{\omega^{\prime}_{T}\})  \\
		& \le g_{T-1}(x^{*}_{T-1}, \xi_{T-1}^{\omega^{\prime}_{T-1}}) + \max_{\bs{p}_{T} \in  \Cs{P}^{\text{A}}_{T| \omega^{\prime}_{T-1}(\{\omega^{\prime}_{T}\}) } }  \ee{\bs{p}_{T}}{Q_{T}(\hstarx{T-1}, \hxi{T}^{\omega_{T}})} \\
		& =g_{T-1}(x^{*}_{T-1}, \xi_{T-1}^{\omega^{\prime}_{T-1}}) + \max_{\bs{p}_{T} \in  \Cs{P}^{\text{CA}}_{T| \omega^{\prime}_{T-1}(\{\omega^{\prime}_{T}\}) } }  \ee{\bs{p}_{T}}{Q_{T}(\hstarx{T-1}, \hxi{T}^{\omega_{T}})}  \\
		& < g_{T-1}(x^{*}_{T-1}, \xi_{T-1}^{\omega^{\prime}_{T-1}}) + \max_{\bs{p}_{T} \in  \Cs{P}_{T| \omega^{\prime}_{T-1}} }  \ee{\bs{p}_{T}}{Q_{T}(\hstarx{T-1}, \hxi{T}^{\omega_{T}})}  \\
		& = Q_{T-1}(\hstarx{T-2}, \hxi{T-1}^{\omega^{\prime}_{T-1}}),
		\end{array}
		\end{equation}
		where the first inequality is due to the suboptimality of $x^*$ to the assessment problem of the scenario path $\omega^{\prime}_{T}$, the first equality is due to the fact $\Cs{P}^{\text{A}}_{T| \omega^{\prime}_{T-1}(\{\omega^{\prime}_{T}\})}$ and $\Cs{P}^{\text{CA}}_{T| \omega^{\prime}_{T-1}(\{\omega^{\prime}_{T}\})}$ are the same given $\hxi{T-1}^{\omega^{\prime}_{T-1}}$, the second inequality is due to \eqref{eq: MULTI.cond_eff_1}, and the second equality is  due to the  time consistency of $x^*$.  
		Thus, for any $\bs{p}_{T-1} \in \Cs{P}_{T-1|\omega^{\prime}_{T-2}}$ with $p_{T-1}(\omega^{\prime}_{T-1})=0$, i.e., $\bs{p}_{T-1} \in \Cs{P}^{\text{CA}}_{T-1|\omega^{\prime}_{T-2}}(\{\omega^{\prime}_{T-1}\})$, we have 
		\begin{equation}
		\label{eq: MULTI.cond_eff_T-1_1}	  
		\begin{array}{ll}
		& g_{T-2}(x^{*}_{T-2}, \xi_{T-2}^{\omega^{\prime}_{T-2}}) + \ee{\bs{p}_{T-1}}{Q_{T-1}^{\text{A}}(\hstarx{T-2}, \hxi{T-1}^{\omega_{T-1}}; \{\omega^{\prime}_{T}\}) } \\
		& = g_{T-2}(x^{*}_{T-2}, \xi_{T-2}^{\omega^{\prime}_{T-2}}) + \ee{\bs{p}_{T-1}}{Q_{T-1}(\hstarx{T-2}, \hxi{T-1}^{\omega_{T-1}}) }  \\
		& < g_{T-2}(x^{*}_{T-2}, \xi_{T-2}^{\omega^{\prime}_{T-2}})+ \max_{\bs{p}_{T-1} \in  \Cs{P}_{T-1|\omega^{\prime}_{T-2}} }  \ee{\bs{p}_{T-1}}{Q_{T-1}(\hstarx{T-2}, \hxi{T-1}^{\omega_{T-1}}) } \\
		& = Q_{T-2}(\hstarx{T-3}, \hxi{T-2}^{\omega^{\prime}_{T-2}}),	
		\end{array}
		\end{equation}
		where the first equality is due  to \eqref{eq: MULTI.Proj_Observation_Simple}, \eqref{eq: MULTI.cond_eff_T}, and the finiteness of $\hXi{T-1}$. %,  and the  fact that $p_{T-1|\omega^{\prime}_{T-2}}(\omega^{\prime}_{T-1})=0$. 
		The  inequality comes from \eqref{eq: MULTI.cond_eff_2} and the second equality is due to the time consistency of $x^*$. 
		On the other hand, for any $\bs{p}_{T-1} \in \Cs{P}_{T-1|\omega^{\prime}_{T-2}}$ with $p_{T-1}(\omega^{\prime}_{T-1})>0$, i.e., $\bs{p}_{T-1|\omega^{\prime}_{T-2}} \notin \Cs{P}^{\text{CA}}_{T-1|\omega^{\prime}_{T-2}}(\{\omega^{\prime}_{T-1}\})$, we have
		\begin{equation}
		\label{eq: MULTI.cond_eff_T-1_2}	  
		\begin{array}{ll}
		& g_{T-2}(x^{*}_{T-2}, \xi_{T-2}^{\omega^{\prime}_{T-2}}) + \ee{\bs{p}_{T-1}}{Q_{T-1}^{\text{A}}(\hstarx{T-2}, \hxi{T-1}^{\omega_{T-1}}; \{\omega^{\prime}_{T}\}) }  \\ 
		& < g_{T-2}(x^{*}_{T-2}, \xi_{T-2}^{\omega^{\prime}_{T-2}}) + \ee{\bs{p}_{T-1}}{Q_{T-1}(\hstarx{T-2}, \hxi{T-1}^{\omega_{T-1}}) }  \\
		& < g_{T-2}(x^{*}_{T-2}, \xi_{T-2}^{\omega^{\prime}_{T-2}})+ \max_{\bs{p}_{T-1} \in  \Cs{P}_{T-1|\omega^{\prime}_{T-2}} }  \ee{\bs{p}_{T-1}}{Q_{T-1}(\hstarx{T-2}, \hxi{T-1}^{\omega_{T-1}}) }  \\
		& =  Q_{T-2}(\hstarx{T-3}, \hxi{T-2}^{\omega^{\prime}_{T-2}}),	 
		\end{array}
		\end{equation}
		where the first inequality comes from \eqref{eq: MULTI.Proj_Observation_Simple},  \eqref{eq: MULTI.cond_eff_T}, and the finiteness of $\hXi{T-1}$.  
		Now, note that $\Cs{P}_{T-1|\omega^{\prime}_{T-2}}$ can be written as the union  
		\begin{equation*}
		\sset*{p_{T-1} \in \Cs{P}_{T-1|\omega^{\prime}_{T-2}} }{ p_{T-1}(\omega^{\prime}_{T-1})=0} 
		\cup 
		\sset*{p_{T-1} \in \Cs{P}_{T-1|\omega^{\prime}_{T-2}} }{ p_{T-1}(\omega^{\prime}_{T-1})>0}.
		\end{equation*}
		Thus, putting the  inequalities \eqref{eq: MULTI.cond_eff_T-1_1} and \eqref{eq: MULTI.cond_eff_T-1_2} together, we conclude 
		\begin{equation}
		\label{eq: MULTI.cond_eff_T-1}
		\begin{array}{ll}
		& Q_{T-2}^{\text{A}}(\hstarx{T-3},\hxi{T-2}^{\omega^{\prime}_{T-2}}; \{\omega^{\prime}_{T}\}) \\
		& \le g_{T-2}(x^{*}_{T-2},\xi_{T-2}^{\omega^{\prime}_{T-2}}) +   \max_{\bs{p}_{T-1} \in  \Cs{P}_{T-1|\omega^{\prime}_{T-2}} }  \ee{\bs{p}_{T-1}}{Q_{T-1}^{\text{A}}(\hstarx{T-2}, \hxi{T-1}^{\omega_{T-1}}; \{\omega^{\prime}_{T}\}) }  \\
		& < Q_{T-2}(\hstarx{T-3},\hxi{T-2}^{\omega^{\prime}_{T-2}}), 
		\end{array}
		\end{equation}
		where the first inequality is due to the suboptimality of $x^*$ to the assessment problem of the scenario path $\omega^{\prime}_{T}$. 
		We can reach a similar conclusion in stage $t=T-2$, where we use \eqref{eq: MULTI.cond_eff_T-1}  instead of \eqref{eq: MULTI.cond_eff_T} and use \eqref{eq: MULTI.Proj_Observation_Simple} for $t=T-2$. Continuing this process backward in time, we conclude
		$ Q_{1}^{\text{A}}(\{\omega^{\prime}_{T}\})
		< Q_{1}$, implying the scenario path $\omega^{\prime}_{T}$ is effective for  T-\droV\ using Definition \ref{def: MULTI.path_eff}.
		
		%%%%%%%%%%%%%%%%%%%%%%%%%%%%%%%%%%%%%%%%%%%%%%%%%%%%%%%%%%%%%%%%
		%%%%%%%%%%%%%%%%%%%%%%%%%%%%%%%%%%%%%%%%%%%%%%%%%%%%%%%%%%%%%%%%
		%%%%%%%%%%%%%%%%%%%%%%%%%%%%%%%%%%%%%%%%%%%%%%%%%%%%%%%%%%%%%%%%
		%%%%%%%%%%%%%%%%%%%%%%%%%%%%%%%%%%%%%%%%%%%%%%%%%%%%%%%%%%%%%%%%
		
		``$\Longleftarrow$:"  To prove the result, we construct  a policy $\bar{x}:=[\bar{x}_1, \ldots, \bar{x}_T]$, defined below, and we show that such a policy is optimal to the assessment problem of the scenario path $\omega^{\prime}_{T}$, corresponding to T-\droV.
		Then, by  Definition \ref{def: MULTI.path_eff},  to prove that the scenario path $\omega^{\prime}_{T}$ is ineffective, it suffices to show that for the policy $\bar{x}$, we have 
		$Q_1^{\text{A}}(\{\omega^{\prime}_{T}\})=g_{1}(\bar{x}_{1}, \xi_{1}) +   \max_{\bs{p}_{2} \in  \Cs{P}_{2|\hxi{1}} } \ \ee{\bs{p}_{2}}{ Q_{2}^{\text{A}}(\bar{x}_{[1]}, \hxi{2}; \{\omega^{\prime}_{T}\})}  
		%= g_{1}(x^{*}_{1}, \xi_{1})+ \max_{\bs{p}_{2} \in  \Cs{P}_{2|\hxi{1}} } \ \ee{\bs{p}_{2}}{ Q_{2}(x^{*}_{[1]}, \hxi{2})}
		=Q_1$.  %we show that the objective function value of the assessment problem at $\bar{x}$ is equal to $Q_1$, indicating that the scenario path $\omega^{\prime}_T$ is effective by Definition \ref{def: MULTI.path_eff}. 
		Note that associated with $\bar{x}$ is an optimal probability $\bar{\bs{p}}:=\bar{\bs{p}}(\bar{x}):=[\bar{\bs{p}}_{2|\hxi{1}},\bar{\bs{p}}_{3|\hxi{2}}, \ldots, \bar{\bs{p}}_{T|\hxi{T-1}}]$.

		For the scenario path $\omega^{\prime}_{T}$, suppose that $\hat{t}+1<T$ denotes the largest stage, where there is an identifiably conditionally ineffective realization. That is, $\omega^{\prime}_{t}$ is either conditionally unidentified or effective, $t=\hat{t}+2, \ldots, T$, while $\omega^{\prime}_{\hat{t}+1}$ is identifiably conditionally ineffective. %We admit the convention that $\hat{t}=T+1$ if $\xi_{T}$ is not conditionally effective (i.e., $\xi_{T}$ is conditionally unknown or ineffective).  
		%Let us denote the policy that is obtained from solving the assessment problem of the scenario path $\omega^{\prime}_{T}$, corresponding to T-\droV, as $\bar{x}=[\bar{x}_1, \ldots, \bar{x}_T]$. 
		We construct a policy $\bar{x}$ as  $\bar{x}=[x^{*}_{1}, \ldots, x^{*}_{\hat{t}}, \hat{x}_{\hat{t}+1}, \ldots, \hat{x}_{T}]$, where $[\hat{x}_{\hat{t}+1}, \ldots, \hat{x}_{T}]$ is a (partial) policy obtained from solving the assessment problem in stage $\hat{t}+1$, given $\bar{x}_{[\hat{t}]}$ and $\hxi{\hat{t}+1}$, with the ambiguity sets  formed as in \eqref{eq: MULTI.TV_ambiguity_set}. %for $t=\hat{t}+1, \ldots, T$, and .
		%Observe that except for at the subtree rooted at $\omega^{\prime}_{\hat{t}+1}$, we have $\bar{x}=x^{*}$. 

		We start the proof by showing that  the policy $\bar{x}$, constructed as above, is optimal for the assessment problem of the scenario path $\omega^{\prime}_{T}$. 
		%In order to show the optimality of $\bar{{x}}$, we construct a dual multiplier function $\bar{\pi}:=[\pi^{*}_{1}, \ldots, \pi^{*}_{\hat{t}-2}, \hat{\pi}_{\hat{t}}, \ldots, \hat{\pi}_{T}]$, where $\hat{\pi}_{\hat{t}-1}, \ldots, \hat{\pi}_{T}$ is the corresponding dual multipliers for the partial policy $\hat{x}_{\hat{t}-1}, \ldots, \hat{x}_{T}$.
		In order to do this, we first claim that  $\bs{\bar{p}}_{\hat{t}+1}=\bs{p}^{*}_{\hat{t}+1}$
		for all $\bs{p}^*_{\hat{t}+1 } \in \Cs{P}^{*}_{t+1|\omega^{\prime}_{\hat{t}} }$ (defined as in \eqref{eq: Pstar} for $t=\hat{t}$ and $\omega^{\prime}_{\hat{t}}$) and  $\bar{\bs{p}}_{\hat{t}+1 } \in \bar{\Cs{P}}^{\text{A}}_{t+1|\omega^{\prime}_{\hat{t}} }(\{\omega^{\prime}_{T}\})$ (defined as in \eqref{eq: Pbar} for $\Cs{S}=\{\omega^{\prime}_{T}\}$, $t=\hat{t}$, and $\hxi{t}=\hxi{t}^{\omega^{\prime}_{\hat{t}}}$). %, where 
		%$$\bs{p}^*_{\hat{t}+1 | \omega^{\prime}_{\hat{t}}} \in \Cs{P}_{t+1|\omega^{\prime}_{\hat{t}} }:= \argmax_{\bs{p}_{t+1| \omega^{\prime}_{\hat{t}} } \in  \Cs{P}_{t+1|\omega^{\prime}_{\hat{t}}} } \  \ee{\bs{p}_{t+1|\omega^{\prime}_{\hat{t}}}}{Q_{t+1}(\hstarx{\hat{t}}, \hxi{ \hat{t}+1})}$$  and 
		%$$\bar{\bs{p}}_{\hat{t}+1 | \omega^{\prime}_{\hat{t}}} \in \bar{\Cs{P}}^{\text{A}}_{t+1|\omega^{\prime}_{\hat{t}} }(\{\omega^{\prime}_{T}\})= \argmax_{\bs{p}_{t+1| \omega^{\prime}_{\hat{t}} } \in  \Cs{P}_{t+1|\omega^{\prime}_{\hat{t}}} } \  \ee{\bs{p}_{t+1|\omega^{\prime}_{\hat{t}}}}{Q_{t+1}^{\text{A}}(\hbarx{\hat{t}}, \hxi{ \hat{t}+1}; \{\omega^{\prime}_{T}\})}.$$ 
		Then, we use the optimality conditions stated in  Lemma \ref{lem: MULTI.OPT_Cond} for the assessment problem. 
		To prove the above claim, first observe that 
		\begin{equation}
		\label{eq: equality}
		Q_{\hat{t}+1}^{\text{A}}(\hbarx{\hat{t}}, \hxi{\hat{t}+1}^{\omega_{\hat{t}+1}};\{\omega^{\prime}_{T}\}) = Q_{\hat{t}+1}^{\text{A}}(\hstarx{\hat{t}}, \hxi{\hat{t}+1}^{\omega_{\hat{t}+1}};\{\omega^{\prime}_{T}\}) \le Q_{\hat{t}+1}(\hstarx{\hat{t}}, \hxi{\hat{t}+1}^{\omega_{\hat{t}+1}}), \quad \omega_{\hat{t}+1} \in \Cs{C}(\omega^{\prime}_{\hat{t}}),     
		\end{equation}
		where the equality is due to $\hstarx{\hat{t}}=\hbarx{\hat{t}}$ by construction, and  the inequality follows a similar argument as in the proof of Proposition \ref{prop: compare_org_asses_path}. 
		Moreover, observe that if $ \omega_{\hat{t}+1} \neq \omega^{\prime}_{\hat{t}+1}$, then  using \eqref{eq: MULTI.Proj_Observation_Simple}  at $t=\hat{t}+1$ and for the feasible policy $\bar{x}$, we have  
		$  Q_{\hat{t}+1}^{\text{A}}(\bar{x}_{[\hat{t}]}, \hxi{\hat{t}+1}^{\omega_{\hat{t}+1}};\{\omega^{\prime}_{T}\})=  Q_{\hat{t}+1}(\bar{x}_{[\hat{t}]}, \hxi{\hat{t}+1}^{\omega_{\hat{t}+1}})$. 
		%\ge   Q_{\hat{t}+1}(\hstarx{\hat{t}}, \hxi{\hat{t}+1}^{\omega_{\hat{t}+1}}),  
		%, and the inequality  comes from the suboptimality of $\bar{x}$ to T-\droV. 
		Hence, considering that $g_{\hat{t}}(x^{*}_{\hat{t}}, \xi_{\hat{t}}^{\omega^{\prime}_{\hat{t}}})= g_{\hat{t}}(\bar{x}_{\hat{t}}, \xi_{\hat{t}}^{\omega^{\prime}_{\hat{t}}})$ by construction, for $\omega_{\hat{t}+1} \neq \omega^{\prime}_{\hat{t}+1}$, we have 
		\begin{equation}
		\label{eq: MULTI.Proj_Observation_Simple_1}
		g_{\hat{t}}(\bar{x}_{\hat{t}}, \xi_{\hat{t}}^{\omega^{\prime}_{\hat{t}}})+Q_{\hat{t}+1}^{\text{A}}(\bar{x}_{[\hat{t}]}, \hxi{\hat{t}+1}^{\omega_{\hat{t}+1}};\{\omega^{\prime}_{T}\})=g_{\hat{t}}(x^{*}_{\hat{t}}, \xi_{\hat{t}}^{\omega^{\prime}_{\hat{t}}})+Q_{\hat{t}+1}(\hstarx{\hat{t}}, \hxi{\hat{t}+1}^{\omega_{\hat{t}+1}}). 
		\end{equation}
		If $Q_{\hat{t}+1}(\hstarx{\hat{t}}, \hxi{\hat{t}+1}^{\omega^{\prime}_{\hat{t}+1}})\!=\!Q_{\hat{t}+1}^{\text{A}}(\bar{x}_{[\hat{t}]}, \hxi{\hat{t}+1}^{\omega^{\prime}_{\hat{t}+1}};\{\omega^{\prime}_{T}\})$, and hence, $g_{\hat{t}}(x^{*}_{\hat{t}}, \xi_{\hat{t}}^{\omega^{\prime}_{\hat{t}}})+Q_{\hat{t}+1}(\hstarx{\hat{t}}, \hxi{\hat{t}+1}^{\omega^{\prime}_{\hat{t}+1}})\!\!=\!\! \!\!g_{\hat{t}}(\bar{x}_{\hat{t}}, \xi_{\hat{t}}^{\omega^{\prime}_{\hat{t}}})+  Q_{\hat{t}+1}^{\text{A}}(\bar{x}_{[\hat{t}]}, \hxi{\hat{t}+1}^{\omega^{\prime}_{\hat{t}+1}};\{\omega^{\prime}_{T}\})$, combined  with \eqref{eq: MULTI.Proj_Observation_Simple_1}, we conclude that $\bar{p}_{\hat{t}+1}(\omega_{\hat{t}+1})  =p^{*}_{\hat{t}+1}(\omega_{\hat{t}+1})$ for $\omega_{\hat{t}+1} \in \Cs{C}(\omega^{\prime}_{\hat{t}}) $ and for all $\bs{p}^*_{\hat{t}+1} \in \Cs{P}^{*}_{\hat{t}+1|\omega^{\prime}_{\hat{t}} }$ and  $\bar{\bs{p}}_{\hat{t}+1 } \in \bar{\Cs{P}}^{\text{A}}_{\hat{t}+1|\omega^{\prime}_{\hat{t}} }(\{\omega^{\prime}_{T}\})$. 
		
		Now, suppose that $Q_{\hat{t}+1}^{\text{A}}(\bar{x}_{[\hat{t}]}, \hxi{\hat{t}+1}^{\omega^{\prime}_{\hat{t}+1}};\{\omega^{\prime}_{T}\}) <  Q_{\hat{t}+1}(\hstarx{\hat{t}}, \hxi{\hat{t}+1}^{\omega^{\prime}_{\hat{t}+1}})$. To prove the claim, we first investigate which of the four primal categories (recall the definition of primal categories  in Section \ref{sec: TV.background} and in the context of \eqref{eq: TV.robust}) a realization $\omega_{\hat{t}+1} \in \Cs{C}(\omega^{\prime}_{\hat{t}})$ can belong to at $\bar{x}$. 
		Observe that considering that $g_{\hat{t}}(x^{*}_{\hat{t}}, \xi_{\hat{t}}^{\omega^{\prime}_{\hat{t}}})= g_{\hat{t}}(\bar{x}_{\hat{t}}, \xi_{\hat{t}}^{\omega^{\prime}_{\hat{t}}})$ and because we assume that $Q_{\hat{t}+1}^{\text{A}}(\bar{x}_{[\hat{t}]}, \hxi{\hat{t}+1}^{\omega^{\prime}_{\hat{t}+1}};\{\omega^{\prime}_{T}\}) <  Q_{\hat{t}+1}(\hstarx{\hat{t}}, \hxi{\hat{t}+1}^{\omega^{\prime}_{\hat{t}+1}})$, the primal category of $\omega^{\prime}_{\hat{t}+1}$ at $\bar{x}$ will be  either the same as that at $x^{*}$ or lower  (in index) than that at $x^{*}$. 
		%We aim to exploit the results stated in Section \ref{sec: TV.background}. So, 
		For  ease of exposition, we adopt the following notation:
		\begin{align*}
		& \gamma:=\gamma_{\hat{t}+1}; \quad  \Cs{C}:= \Cs{C}(\omega^{\prime}_{\hat{t}}); \quad 
		\omega:= \omega_{\hat{t}+1}, \quad \omega_{\hat{t}+1} \in \Cs{C}(\omega^{\prime}_{\hat{t}});  \quad 
		\omega^{\prime}:=\omega^{\prime}_{\hat{t}+1}, \\
		& q(\omega):=q_{\hat{t}+1|\omega^{\prime}_{\hat{t}}}(\omega_{\hat{t}+1}), \quad 
		p^{*}(\omega):=p^*_{\hat{t}+1}(\omega_{\hat{t}+1}), \quad 
		\bar{p}(\omega):=\bar{p}_{\hat{t}+1}(\omega_{\hat{t}+1}), \quad \omega_{\hat{t}+1} \in \Cs{C}(\omega^{\prime}_{\hat{t}}), \\
		& y^{*}:=x^{*}_{[\hat{t}]}; \quad \bar{y}:=\bar{x}_{[\hat{t}]}, \\
		& h(y^{*},\omega):=g_{\hat{t}}(x^{*}_{\hat{t}}, \xi_{\hat{t}}^{\omega^{\prime}_{\hat{t}}})+Q_{\hat{t}+1}(\hstarx{\hat{t}}, \hxi{\hat{t}+1}^{\omega_{\hat{t}+1}}), \quad \omega_{\hat{t}+1} \in \Cs{C}(\omega^{\prime}_{\hat{t}}),  \\
		& h^{\text{A}}(\bar{y}, \omega):=g_{\hat{t}}(\bar{x}_{\hat{t}}, \xi_{\hat{t}}^{\omega^{\prime}_{\hat{t}}}) + Q_{\hat{t}+1}^{\text{A}}(\bar{x}_{[\hat{t}]}, \hxi{\hat{t}+1}^{\omega_{\hat{t}+1}};\{\omega^{\prime}_{T}\}), \quad \omega_{\hat{t}+1} \in \Cs{C}(\omega^{\prime}_{\hat{t}}), \\
		& \sup_{\omega \in \Cs{C}} h(y^{*},\omega):=\mathrm{sup}_{|\hxi{\hat{t}}^{\omega^{\prime}_{\hat{t}}}}\ \left[g_{\hat{t}}(x^{*}_{\hat{t}}, \xi_{\hat{t}}^{\omega^{\prime}_{\hat{t}}})+Q_{\hat{t}+1}(\hstarx{\hat{t}}, \hxi{\hat{t}+1}^{\omega_{\hat{t}+1}})\right],\\
		& \sup_{\omega \in \Cs{C}} h^{\text{A}}(\bar{y},\omega):=\mathrm{sup}_{|\hxi{\hat{t}}^{\omega^{\prime}_{\hat{t}}}}\ \left[g_{\hat{t}}(\bar{x}_{\hat{t}}, \xi_{\hat{t}}^{\omega^{\prime}_{\hat{t}}}) + Q_{\hat{t}+1}^{\text{A}}(\bar{x}_{[\hat{t}]}, \hxi{\hat{t}+1}^{\omega_{\hat{t}+1}};\{\omega^{\prime}_{T}\})\right],\\
		& \vvar{\gamma}{h(y^{*},\omega)}:=\vvar{\gamma_{\hat{t}+1}}{g_{\hat{t}}(x^{*}_{\hat{t}}, \xi_{\hat{t}}^{\omega^{\prime}_{\hat{t}}})+Q_{\hat{t}+1}(\hstarx{\hat{t}}, \hxi{\hat{t}+1}^{\omega_{\hat{t}+1}}) \; \Big| \;  \hxi{t}^{\omega^{\prime}_{\hat{t}} } },\\
		& \vvar{\gamma}{h^{\text{A}}(\bar{y},\omega)}:=\vvar{\gamma_{\hat{t}+1}}{g_{\hat{t}}(\bar{x}_{\hat{t}}, \xi_{\hat{t}}^{\omega^{\prime}_{\hat{t}}}) + Q_{\hat{t}+1}^{\text{A}}(\bar{x}_{[\hat{t}]}, \hxi{\hat{t}+1}^{\omega_{\hat{t}+1}};\{\omega^{\prime}_{T}\})  \; \Big| \;  \hxi{\hat{t}}^{\omega^{\prime}_{\hat{t}}}}.
		\end{align*}
		Now, let us form $\cat{i}{y^{*}}$ and $\cat{i}{\bar{y}}$, $i=1,2,3,4$, as described in Section \ref{sec: TV.background}. 
		Using the above notation, \eqref{eq: MULTI.Proj_Observation_Simple_1} is stated as $h(y^{*},\omega)=h^{\text{A}}(\bar{y},\omega)$, $\omega \neq \omega^{\prime}$. Also, by the hypothesis, we have $h(y^{*},\omega^{\prime})>h^{\text{A}}(\bar{y},\omega^{\prime})$. 
		Because $\omega^{\prime}$ is identifiably conditionally ineffective, then by  Lemma \ref{lem: MULTI.cond}\ref{lem: MULTI.cond_ineff}, we have $p^*(\omega^{\prime})=0$. Moreover, by \citep[Theorems~1]{rahimian2019}, we have $\omega^{\prime} \notin \cat{4}{y^{*}}$. Thus, $\omega^{\prime} \notin \cat{4}{\bar{y}}$. This implies $\cat{4}{y^{*}}=\cat{4}{\bar{y}}$ and $\sup_{\omega \in \Cs{C}} h(y^{*},\omega)=\sup_{\omega \in \Cs{C}} h^{\text{A}}(\bar{y},\omega)$.
		Let us consider the following cases: 
		\begin{nested}
			\item $\omega^{\prime} \in \cat{1}{y^{*}}$: Thus, $\omega^{\prime} \in \cat{1}{\bar{y}}$ and $\vvar{\gamma}{h(y^{*},\omega)}=\vvar{\gamma}{h^{\text{A}}(\bar{y},\omega)}$. Consequently, the formation of primal categories is the same at $y^{*}$ and $\bar{y}$, and we have $\bar{p}(\omega)=p^{*}(\omega)$, $\omega \in \Cs{C}$. 
			
			\item $\omega^{\prime}\!\in\!\cat{2}{y^{*}} \cup \cat{3}{y^{*}}$ and $q(\omega^{\prime})=0$: Because $q(\omega^{\prime})\!=\!0$, we have $\vvar{\gamma}{h(y^{*},\omega)}  =\vvar{\gamma}{h^{\text{A}}(\bar{y},\omega)}$. On the other hand, $\bar{p}(\omega^{\prime})=0$ by \citep[Proposition~4]{rahimian2019}. Thus, although the primal category of $\omega^{\prime}$ might change, we can conclude $\bar{p}(\omega)=p^{*}(\omega)$, $\omega \in \Cs{C}$, because the primal category of all other scenarios $\omega^{\prime} \neq \omega \in \Cs{C}$ does not change. 
			
			\item $\omega^{\prime} \in \cat{2}{y^{*}}$ and $q(\omega^{\prime})>0$: If there is no other scenario in $\cat{2}{y^{*}}$ with a positive nominal probability, then, $\vvar{\gamma}{h(y^{*},\omega)}>\vvar{\gamma}{h^{\text{A}}(\bar{y},\omega)}$. Because $\sum_{\omega \in \cat{2}{y^{*}}} p^{*}(\omega)=0$, we have $\sum_{\omega \in \cat{1}{y^{*}} \cup \cat{2}{y^{*}}} q(\omega)=\gamma$ by  \citep[Proposition~4]{rahimian2019}. On the other hand, 
			we must have $q(\omega)=0$ for all $\omega \in \Cs{M}_{3}:=\big\{\omega \in \Cs{C}: \vvar{\gamma}{h^{\text{A}}(\bar{y},\omega)} < h^{\text{A}}(\bar{y},\omega) < \vvar{\gamma}{h(y^{*},\omega)}\big\}$ because 
			$\gamma \le \sum_{\omega \in \cat{1}{\bar{y}} \cup \cat{2}{\bar{y}}} q(\omega)\le \sum_{\omega \in \cat{1}{y^{*}} \cup \cat{2}{y^{*}}} q(\omega) = \gamma$.
			Thus, although scenarios $\omega \in \Cs{M}_{3}$ move from $\cat{1}{y^{*}}$ to $\cat{3}{\bar{y}}$, by \citep[Proposition~4]{rahimian2019}, we must have $\bar{p}(\omega)=0$ for $\omega \in \Cs{M}_{3}$ because $q(\omega)=0$.  
			Moreover, because $\sum_{\omega \in \cat{1}{\bar{y}} \cup \cat{2}{\bar{y}}} q(\omega)=\gamma$, we must have $\bar{p}(\omega)=0$, $\omega \in \Cs{M}_{2}:=\big\{\omega \in \Cs{C}: h^{\text{A}}(\bar{y},\omega) = \vvar{\gamma}{h^{\text{A}}(\bar{y},\omega)} \big\} $ by  \citep[Proposition~4]{rahimian2019}.  Thus, although  scenarios $\omega \in \Cs{M}_{2}$ move from $\cat{1}{y^{*}}$ to $\cat{2}{\bar{y}}$, there is no change in their worst-case probabilities. All the other scenarios keep the same primal category as that at  $y^{*}$. Consequently, $\bar{p}(\omega)=p^{*}(\omega)$, $\omega \in \Cs{C}$. 
			Now, if there exists another scenario in $\cat{2}{y^{*}}$ with a positive nominal probability, then $\vvar{\gamma}{h(y^{*},\omega)}=\vvar{\gamma}{h^{\text{A}}(\bar{y},\omega)}$. 
			Thus, all scenarios, except for $\omega^{\prime}$ that moves to $\cat{1}{\bar{y}}$, keep their primal categories as that at $y^{*}$. Nevertheless, we have $\bar{p}(\omega)=p^{*}(\omega)$, $\omega \in \Cs{C}$. 
		\end{nested}
		In all the above cases, we concluded $\bar{p}(\omega)=p^{*}(\omega)$, $\omega \in \Cs{C}$.
		Now, let us bring back our general notation. 
		The above is equivalent to $\bs{p}^{*}_{\hat{t}+1}=\bar{\bs{p}}_{\hat{t}+1}$, for all $\bs{p}^*_{\hat{t}+1} \in \Cs{P}^{*}_{\hat{t}}$ and  $\bar{\bs{p}}_{\hat{t}+1} \in \bar{\Cs{P}}^{\text{A}}_{\hat{t}+1|\omega^{\prime}_{\hat{t}} }(\{\omega^{\prime}_{T}\})$;  thus, the claim is proved.
		
		To prove the optimality of $\bar{x}$ to the assessment problem we use Lemma \ref{lem: MULTI.OPT_Cond}. Recall that  $[\hat{x}_{\hat{t}+1}, \ldots, \hat{x}_{T}]$ is a (partial) policy obtained from solving the assessment problem in stage $\hat{t}+1$, given $\bar{x}_{[\hat{t}]}$ and $\hxi{\hat{t}+1}$, with the conditional ambiguity sets  formed as in \eqref{eq: MULTI.TV_ambiguity_set}. Thus, optimality conditions stated in Lemma \ref{lem: MULTI.OPT_Cond} hold for $t=T, \ldots, \hat{t}+1$ by construction. Now, let us focus on stage $\hat{t}$. Recall by Remark \ref{rem: MULTI.DP_path} that the only realization whose  cost-to-go function at stage $\hat{t}$ is affected by the removal of the scenario path $\omega^{\prime}_{T}$ is $\omega^{\prime}_{\hat{t}}$. Thus, it remains to verify the optimality conditions for  realization $\omega^{\prime}_{\hat{t}}$. 
		Consider the optimality conditions 
		$
		0 \in \partial g_{\hat{t}}(x^{*}_{\hat{t}}, \xi_{\hat{t}}^{\omega^{\prime}_{\hat{t}}})+ \Cs{N}_{\Cs{X}_{\hat{t}}}(x^{*}_{\hat{t}})+ \conv{\bigcup_{\bs{p}_{\hat{t}+1} \in  \Cs{P}^{*}_{\hat{t}+1|\omega^{\prime}_{\hat{t}}} } \ee{\bs{p}_{\hat{t}+1}}{\partial Q_{\hat{t}+1}(\hstarx{\hat{t}}, \xi_{[\hat{t}+1]}^{\omega_{\hat{t}+1}})}  }, \ \textrm{and} \ 
		0 \in \partial g_{\hat{t}}(\bar{x}_{\hat{t}}, \xi_{\hat{t}}^{\omega^{\prime}_{\hat{t}}})+ \Cs{N}_{\Cs{X}_{\hat{t}}}(\bar{x}_{\hat{t}})+ \conv{\bigcup_{\bs{p}_{\hat{t}+1} \in  \bar{\Cs{P}}^{\textrm{A}}_{\hat{t}+1|\omega^{\prime}_{\hat{t}}}(\{\omega^{\prime}_{T}\}) } \ee{\bs{p}_{\hat{t}+1}}{\partial Q_{\hat{t}+1}^{\text{A}}(\bar{x}_{[\hat{t}]}, \xi_{[\hat{t}+1]}^{\omega_{\hat{t}+1}}; \{\omega^{\prime}_{T}\}) }}. 
		$
		Observe that the right-hand sides of the optimality conditions above are the same because of the following reasons: (i) $\hstarx{\hat{t}}=\hbarx{\hat{t}}$ by construction, (ii)  by \eqref{eq: MULTI.Proj_Observation_Simple},  $Q_{\hat{t}+1}(\hstarx{\hat{t}}, \xi_{[\hat{t}+1]}^{\omega_{\hat{t}+1}})$ and  $Q_{\hat{t}+1}^{\text{A}}(\bar{x}_{[\hat{t}]}, \xi_{[\hat{t}+1]}^{\omega_{\hat{t}+1}}; \{\omega^{\prime}_{T}\}) $ are the exact same functions, for  $\omega_{\hat{t}+1} \neq \omega^{\prime}_{\hat{t}+1}$; hence, they have the same subdifferentials, 
		and (iii)   $\Cs{P}^{*}_{\hat{t}+1|\omega^{\prime}_{\hat{t}}}=\bar{\Cs{P}}^{\textrm{A}}_{\hat{t}+1|\omega^{\prime}_{\hat{t}}}$; in particular, using Lemma \ref{lem: MULTI.cond}\ref{lem: MULTI.cond_ineff} we have $\bar{p}_{\hat{t}+1}(\omega^{\prime}_{\hat{t}+1})=p^{*}_{\hat{t}+1}(\omega^{\prime}_{\hat{t}+1})=0$ for any $\bar{\bs{p}}_{\hat{t}+1} \in \bar{\Cs{P}}^{\textrm{A}}_{\hat{t}+1|\omega^{\prime}_{\hat{t}}}(\{\omega^{\prime}_{T}\})$ and for any $\bs{p}^{*}_{\hat{t}+1} \in \Cs{P}^{*}_{\hat{t}+1|\omega^{\prime}_{\hat{t}}}$. 
		Consequently, although the subdifferentials of $Q_{\hat{t}+1}(\hstarx{\hat{t}}, \xi_{[\hat{t}+1]}^{\omega^{\prime}_{\hat{t}+1}})$ and   $Q_{\hat{t}+1}^{\text{A}}(\bar{x}_{[\hat{t}]}, \xi_{[\hat{t}+1]}^{\omega^{\prime}_{\hat{t}+1}}; \{\omega^{\prime}_{T}\}) $ might be different, they do not contribute to the right-hand sides of the optimality conditions above. 
		Hence, because $x^{*}_{\hat{t}}$ satisfies the optimality condition for T-\droV\ at stage $\hat{t}$, then  $\bar{x}_{\hat{t}}=x^{*}_{\hat{t}}$ satisfies the optimality condition for the respective assessment problem at stage $\hat{t}$. Moreover, it can be seen that 
		$Q_{\hat{t}}(\hstarx{\hat{t}-1}, \xi_{[\hat{t}]}^{\omega^{\prime}_{\hat{t}}})= Q_{\hat{t}}^{\text{A}}(\bar{x}_{[\hat{t}-1]}, \xi_{[\hat{t}]}^{\omega^{\prime}_{\hat{t}}}; \{\omega^{\prime}_{T}\})$. 
		%Now, let us focus on stage $\hat{t}-1$. Again, recall by Remark \ref{rem: MULTI.DP_path} that the only realization whose  cost-to-function at stage $\hat{t}-1$ is affected by the removal of the scenario path $\omega^{\prime}_{T}$ is $\omega^{\prime}_{\hat{t}-1}$. Thus, it remains to verify the optimality conditions at the realization $\omega^{\prime}_{\hat{t}-1}$. 
		%Consider 
		%\begin{equation}
		%\label{eq: MULTI.OPT_Cond_hatt_x*}
		%   0 \in \partial g_{\hat{t}-1}(x^{*}_{\hat{t}-1}, \xi_{\hat{t}-1}^{\omega^{\prime}_{\hat{t}-1}})+ \Cs{N}_{\Cs{X}_{\hat{t}-1}}(x^{*}_{\hat{t}-1})+ \conv{\bigcup_{\bs{p}_{\hat{t}|\omega^{\prime}_{\hat{t}-1}} \in  \Cs{P}^{*}_{\hat{t}|\omega^{\prime}_{\hat{t}-1}} } \ee{\bs{p}_{\hat{t}|\omega^{\prime}_{\hat{t}-1}}}{\partial Q_{\hat{t}}(\hstarx{\hat{t}-1}, \xi_{[\hat{t}]}^{\omega_{\hat{t}}})}  }   
		%\end{equation} 
		%and 
		%\begin{equation}
		%\label{eq: MULTI.OPT_Cond_hatt_barx}
		%    0 \in \partial g_{\hat{t}-1}(\bar{x}_{\hat{t}-1}, \xi_{\hat{t}-1}^{\omega^{\prime}_{\hat{t}-1}})+ \Cs{N}_{\Cs{X}_{\hat{t}-1}}(\bar{x}_{\hat{t}-1})+ \conv{\bigcup_{\bs{p}_{\hat{t}|\omega^{\prime}_{\hat{t}-1}} \in  \bar{\Cs{P}}_{\hat{t}|\omega^{\prime}_{\hat{t}-1}}(\{\omega^{\prime}_{T}\}) } \ee{\bs{p}_{\hat{t}|\omega^{\prime}_{\hat{t}-1}}}{\partial Q_{\hat{t}}^{\text{A}}(\bar{x}_{[\hat{t}-1]}, \xi_{[\hat{t}]}^{\omega_{\hat{t}}}; \{\omega^{\prime}_{T}\}) }} .
		%\end{equation}  
		%This, combined with the reason (iii), implies that 
		Therefore, this fact and the implication from Remark \ref{rem: MULTI.DP_path} guarantee that because $x^{*}_{\hat{t}-1}$ satisfies the optimality condition for T-\droV\ at stage $\hat{t}-1$, then  $\bar{x}_{\hat{t}-1}=x^{*}_{\hat{t}-1}$ satisfies the optimality condition for the respective assessment problem at stage $\hat{t}-1$. By induction, we conclude that $\bar{x}$ is optimal to the assessment problem of the scenario path $\omega^{\prime}_{T}$ corresponding to T-\droV. 
		
		Additionally, observe that with a similar reasoning we have  $Q_{1}^{\text{A}}(\omega^{\prime}_{T})= Q_{1}$, i.e., the policy $\bar{x}$ incurs an objective function value for the assessment problem of $\omega^{\prime}_{T}$ corresponding to  T-\droV\ that is equal to the optimal value of  T-\droV, incurred by the policy $x^{*}$. 
		Consequently, by Definition \ref{def: MULTI.path_eff}, we conclude that the scenario path $\omega^{\prime}_{T}$ is ineffective. 
		The case that $\hat{t}+1=T$, i.e., $\omega^{\prime}_{T}$ is conditionally ineffective, can be argued similarly, although the notation slightly simplifies. This completes the proof. 
	\end{proof}
	
	% % % % % % % % % % % % % % % % % % % % % % % % % % % % % % % % % % % % % % % % % % % %	
	\section{Numerical Illustration}
	% % % % % % % % % % % % % % % % % % % % % % % % % % % % % % % % % % % % % % % % % % % %	
	\label{sec: MULTI.numerics}	
	%We explain our experimental setup in Section \ref{sec: MULTI.setup}.
	%Then,  we discuss how conditionally effective realizations and effective scenarios paths  can be used to help decision makers gain insight about the underlying uncertainty of the problem in Section \ref{sec: insight}. 
	
	%We study the efficacy of our proposed easy-to-check conditions in identifying the conditional effectiveness of realizations and  effectiveness of scenario paths in Section \ref{sec: num.MULTI.ETC}.

	\subsection{Experimental Setup}
	\label{sec: MULTI.setup}
	To conduct numerical analyses, we used a version of the water allocation problem  described in \citep{zhang2016decomposition}, where now water treatment facilities are subject to disruptions. 
	This problem addresses the allocation of Colorado River water among different users, while meeting uncertain water demand and not exceeding uncertain water supply.
	The problem originally has  4 stages with the  planning period 2010--2050, where both supply  and demand uncertainties appear in the right-hand side of constraints. 
	%We generated a variant of this problem with 50 realizations per node (a total of $50^3=125 \times 10^3$ scenario paths), and we denote it by WATERS4N50.  
	To examine more complex uncertainties besides only the right-hand-side of the constraints and for visual illustration of managerial insights, we generated a smaller three-stage variant of this problem with a 26-year planning horizon as follows. We considered low (L) and high (H) realizations for both the supply and demand uncertainties. In addition, we introduced a random variable to specify whether one of the wastewater treatment plants is disrupted (D) or nondisrupted (N). This leads to randomness in the recourse matrix and can be interpreted as disruption due to natural or man-made disasters. 
	The treatment plant treats wastewater and pumps it back into the system to meet nonpotable (i.e., nondrinkable) water demands like irrigation of parks. 
	This variant has 8 realizations per node ($8^2=64$ scenario paths), and we denote it by WATERS3N8. 
	For reference, the labels in the second stage of Figure \ref{fig: MULTI.WATERS3N8}, given as a triplet of (supply, demand, treatment facility), illustrate the second-stage realizations. This pattern repeats in the third stage as well. With respect to the nominal distribution, the stochastic process is interstage independent and all realizations are equally likely for simplicity. We tested interstage dependent cases as well, and the results were similar. 
	We modeled the distributional ambiguity via the total variation distance and   assumed that the levels of robustness  for all stages are  the same (i.e., $\gamma_t=\gamma$ for all  $t$, $t=2, \ldots, T$). To run the experiments, we varied $\gamma$ between 0 and 1 in increments of $0.05$. %Hence, each test problem is solved a total of $21$ times.
	We implemented a nested Benders' decomposition algorithm %(see Algorithm \ref{alg: MULTI.nested_linear_decomp} in Online Supplement \ref{sec: MULTI.primal}) 
	in C++ on top of SUTIL \cite{SUTIL} to solve the problems. %A Stochastic Programming Utility Library 
	All problems were solved using CPLEX 12.9 on a Linux Ubuntu environment  with an Intel Core i7-2640M 2.8 GHz  processor
	and 8.00 GB RAM.

	\subsection{Managerial Insights} 
	\label{sec: insight}
	
	%We solved WATERS3N8 for different values of the robustness parameter $\gamma$. 
	Our proposed easy-to-check  conditions (Theorem \ref{thm: MULTI.cond_eff}) identify the conditional effectiveness of all realizations at all tested levels of robustness for this problem. Hence, by Theorem \ref{thm: MULTI.path_eff_ineff}, the effectiveness of  all scenario paths is identified. 
	Figure \ref{fig: MULTI.WATERS3N8} illustrates  the results for some ranges of values of $\gamma$ to examine the evolution of effective scenario paths as the level of robustness increases. In all figures, conditionally effective realizations are shown with  solid lines and conditionally ineffective realizations are shown with dotted lines. The right panels in Figures \ref{fig: MULTI.WATERS3N8_1}--\ref{fig: MULTI.WATERS3N8_6} illustrate the  worst-case probability distribution on the stochastic process $\xi$ at levels of robustness $\gamma=0.25, 0.5, 0.8, 1$, respectively. 
	At all depicted $\gamma$ values, we have a positive worst-case probability for effective scenario paths and 
	zero worst-case probability for ineffective scenario paths.\textbf{}
	
	Figure \ref{fig: MULTI.WATERS3N8_1} illustrates that the effective/ineffective pattern can change across stages. For instance, a LLN realization can be conditionally effective in one stage and not the other. Indeed, we see that, for low levels of $\gamma$,   in stage~2 the conditionally effective realizations are those with either low supply $(\textrm{L},\cdot, \cdot)$ or high demand $(\cdot, \textrm{H}, \cdot)$, regardless of the availability of the treatment facility. In contrast, in stage~3  the conditionally effective realizations are those with either high demand $(\cdot, \textrm{H}, \cdot)$ or disrupted treatment facility $(\cdot, \cdot, \textrm{D})$, regardless of the supply level. This suggests that, for this level of robustness, in the short term the supply of water  is more critical than the availability of the treatment facility, but  in the long term the opposite happens. One explanation for this fact is that when the treatment facility is not available, the system has to procure water from expensive external sources to meet the demand and water shortages are more pronounced in the third stage. 
	
	In  Figure \ref{fig: MULTI.WATERS3N8_3}, we see that for mid-levels of $\gamma$ the conditionally effective realizations are the same in both stages 2 and 3, and correspond only to those with  high demand. As the value of $\gamma$ increases in Figure \ref{fig: MULTI.WATERS3N8_5}, we see some interesting dynamics: although  the conditionally effective realizations  in stage~3 still include some  scenarios with high supply (as long as demand is high and the facility is disrupted), in terms of effectiveness of the \textit{entire path}, the only critical scenarios are those with low supply and high demand in both stages.
	This again highlights that high water deficit %(demand$-$supply) 
	is more critical than the treatment plant disruptions for the studied problem.
	Finally, for the highest values of $\gamma$, Figure \ref{fig: MULTI.WATERS3N8_6} shows that there is only one critical scenario path, which corresponds to the case where all the uncertainty is unfavorable:  there is low supply, high demand, and a disrupted treatment facility in both stages. 
	We infer from the analysis that  demand appears to be the most critical uncertainty component, followed by supply. The availability of the treatment facility, although has some impact especially at later stages, appears to be less critical than the other two components.

	\begin{figure}[!htbp]
		\centering
		\subfloat[][{$\gamma \in [0.25, 0.35]$}.]{\label{fig: MULTI.WATERS3N8_1}\includegraphics[width=0.32\linewidth]{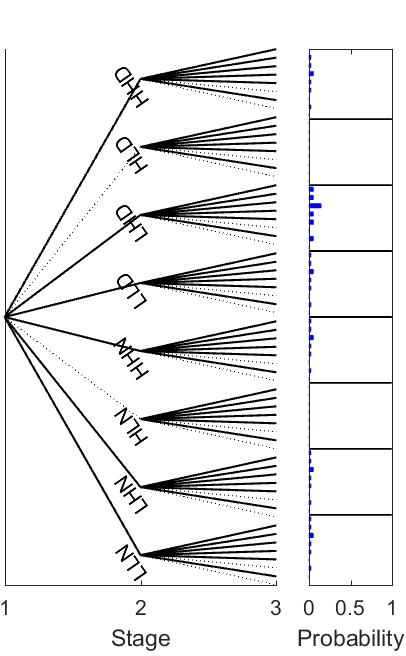}}
		%\subfloat[][{$\gamma \in [0.4, 0.45]$}.]{\label{fig: MULTI.WATERS3N8_2}\includegraphics[width=0.45\linewidth,angle=270,origin=c]{managerial/eff_prob_0_45_1.png}}
		\subfloat[][{$\gamma \in [0.5, 0.60]$}.]{\label{fig: MULTI.WATERS3N8_3}\includegraphics[width=0.32\linewidth]{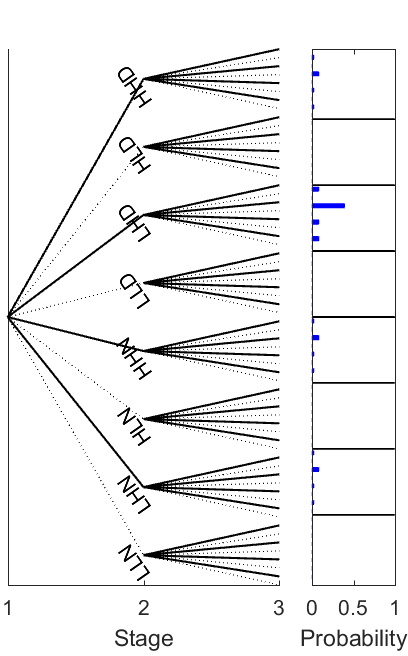}}\\
		%\subfloat[][{$\gamma \in [0.65, 0.7]$}.]{\label{fig: MULTI.WATERS3N8_4}\includegraphics[width=0.45\linewidth,,angle=270,origin=c]{managerial/eff_0_70_1.png}}
		\subfloat[][{$\gamma \in [0.75, 0.85]$}.]{\label{fig: MULTI.WATERS3N8_5}\includegraphics[width=0.32\linewidth]{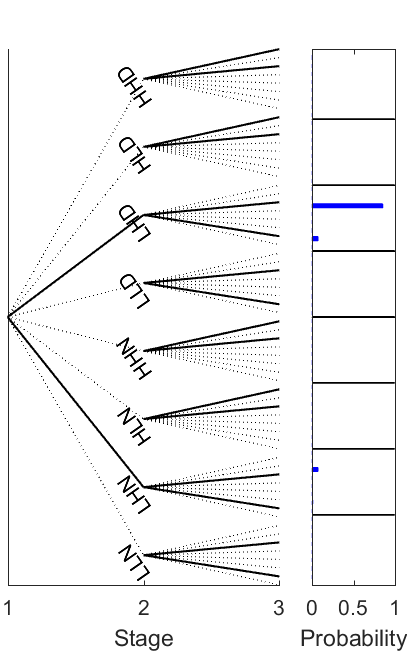}}
		\subfloat[][{$\gamma \in [0.9, 1]$}.]{\label{fig: MULTI.WATERS3N8_6}\includegraphics[width=0.32\linewidth]{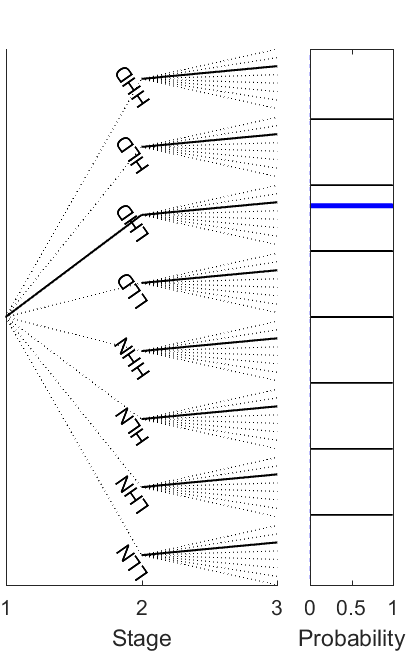}}
		\caption{\label{fig: MULTI.WATERS3N8} The effectiveness of realizations and scenario paths, identified by the easy-to-check conditions, as well as the  worst-case probability distribution of the stochastic process, for  WATERS3N8. Conditionally effective realizations are shown with thick solid black   lines, and conditionally ineffective realizations are shown with thin  dotted  black lines on the scenario tree. Labels on the second stage refer to a realization with a triplet (supply, demand, treatment facility), where supply and demand $\in \{\textrm{L},\textrm{H}\}$ and treatment plant $\in \{\textrm{D}, \textrm{N}\}$. For instance, the triplet $(\textrm{L},\textrm{H}, \textrm{N})$ stands for a realization with a low supply, high demand, and a nondisrupted treatment plant.}
	\end{figure}

	% % % % % % % % % % % % % % % % % % % % % % % % % % % % % % % % % % % % % % % % % % % %	
	\section{Conclusions} \label{sec: MULTI.discuss}
	% % % % % % % % % % % % % % % % % % % % % % % % % % % % % % % % % % % % % % % % % % % %	
	
	In this paper, we investigated a general class of multistage distributionally robust convex stochastic optimization (multistage \dro) problems. % to hedge against the distributional ambiguity.
	Under a finite sample space, we %analyzed the critical scenarios we are concerned about in this setting. 
	defined the notions of the conditional effectiveness of realizations and the effectiveness of scenario paths for a multistage \dro.  
	By exploiting the specific structures of the ambiguity set formed via the total variation distance, we  proposed easy-to-check conditions to identify the conditional effectiveness of  realizations and effectiveness of scenario paths for a  multistage \droV. 
	%Numerical results showed that these conditions work quite well and the number of unidentified scenario paths is small relative to the total number of scenario paths for the tested problems.
	Our main result establishes an important and practical connection between conditional effectiveness and the effectiveness of the entire scenario path. This allows to easily identify the critical scenarios in multistage \droV, which can be  significantly more difficult than the two-stage setting. 
	By means of a practical application to a water resources allocation problem, we  illustrated how these notions can be used to help decision makers gain insight on the underlying uncertainty of a multistage problem. 
	
	Future work includes investigating easy-to-check conditions for other classes of ambiguity sets as well as the insights obtained from the the conditional effectiveness of  realizations and effectiveness of scenario paths in   general multistage \dro s. From a computational perspective, it would also be interesting to explore how these notions can be used to reduce scenarios in the multistage setting and to refine approximations of  cost-to-go functions in decomposition algorithms to accelerate such algorithms. % when the multistage \dro\ problem is solved with a variant of the nested decomposition algorithm. 
	
	\begin{comment}	
	To summarize, the contributions of this paper  are as follows:
	%\vspace{-\topsep}
	\begin{enumerate}[label=(\roman*)]
	
	\item The notions of effectiveness of realizations and scenario paths are applicable to a general class of multistage \dro s.
	For multistage \dro, the effectiveness of realizations and scenario paths can be identified by re-solving the corresponding assessment problem and using the general definition. 
	
	\item For multistage \droV, we proposed easy-to-check conditions.
	Numerical results showed that these conditions work quite well and the number of undetermined scenario paths is small relative to the total number of scenario paths for the tested problems.
	
	\end{enumerate}
	\end{comment}

	%%%%%%%%%%%%%%%%%%%%%%%%%%%%%%%%%%%%%%%%%%%%%%%%%%%%%%%%%%%%%%%%%%%%%%%%%%%%%%%%
	%\section*{Acknowledgements}  \vspace*{-0.1in}
	%T. Homem-de-Mello acknowledges the support of grant ANID-FONDECYT 1171145.
	%%%%%%%%%%%%%%%%%%%%%%%%%%%%%%%%%%%%%%%%%%%%%%%%%%%%%%%%%%%%%%%%%%%%%%%%%%%%%%%%

	{\small
		\bibliographystyle{siamplain} 
		\bibliography{bibfile}  
	}
\end{document}